\documentclass[a4paper,12pt]{amsart}
\usepackage{amssymb}
\usepackage{amsmath}
\usepackage{lscape}
\usepackage{tikz}
\usepackage{xcolor}
\usetikzlibrary{positioning, arrows.meta, shapes.multipart, calc}
\usepackage{capt-of}
\usepackage{makecell}
\newtheorem{theorem}{Theorem}[section] 
\newtheorem{lemma}[theorem]{Lemma} 
\newtheorem{corollary}[theorem]{Corollary} 
\newtheorem{proposition}[theorem]{Proposition}
 
\theoremstyle{definition}
\newtheorem{definition}[theorem]{Definition} 
\newtheorem{remark}[theorem]{Remark}

\newtheorem{example}[theorem]{Example}

\title[]{Classification of degenerate Verma modules over $E(4,4)$}

\author{Nicoletta Cantarini}\author{Fabrizio Caselli}\author{Victor Kac}

\address{Nicoletta Cantarini and Fabrizio Caselli, Dipartimento di matematica, Universit\`a di Bologna, Piazza di Porta San Donato 5, 40126 Bologna, Italy; Victor Kac, Math.\ Department, MIT, Cambridge, MA02139}

\email{nicoletta.cantarini@unibo.it}
\email{fabrizio.caselli@unibo.it}
\email{kac@math.mit.edu}

 
  \DeclareMathOperator{\Hom}{Hom}

  \DeclareMathOperator{\diver}{div}

  \DeclareMathOperator{\Ind}{Ind}

\newcommand{\C}{\mathbb{C}}
\newcommand{\Z}{\mathbb{Z}}

\newcommand{\slq}{\mathfrak{sl}_4}

\newcommand{\de}{\partial}
\newcommand{\phat}{\hat{\mathfrak{p}}(4)}
\newcommand{\inlinewedge}{\textrm{\raisebox{0.6mm}{\footnotesize $\bigwedge$}}}
\newcommand{\displaywedge}{\textrm{\raisebox{0.6mm}{\tiny $\bigwedge$}}}

\begin{document}
%\today

\begin{abstract}
In this paper we classify degenerate Verma modules over the linearly compact Lie superalgebra $E(4,4)$. This completes the description of Verma modules over the exceptional linearly compact Lie superalgebras. As in the other cases all degenerate modules and morphisms between them give rise to infinite bilateral complexes which may be viewed as a generalization of de Rham complexes. 
\end{abstract}

\maketitle
\section{Introduction}
Let $\Omega^0(n)$ be the linearly compact algebra of formal power series in $x_1, \dots, x_n$ over $\C$ and let $\Omega(n)$ be the superalgebra of differential forms with coefficients in $\Omega^0(n)$. Recall that $\Omega(n)$ is a unital (super) commutative associative superalgebra over the even subalgebra $\Omega^0(n)$ with odd generators $dx_1, \dots, dx_n$. The superalgebra $\Omega(n)$ carries an odd square zero derivation $d$, defined by
\[d(x_i)=dx_i,\,\, d(dx_i)=0,\,\, i=1, \dots, n,\]
and a $\Z$-gradation 
\begin{equation}
\Omega(n)=\bigoplus_{j=0}^n\Omega^j(n),
\label{Omegagradation}
\end{equation}
defined by
\[
\deg x_i=0,\,\, \deg(dx_i)=1.
\]
The formal de Rham complex is the following exact sequence of linearly compact vector spaces over $\C$:
\begin{equation}
\Omega^0(n)\xrightarrow{d}\Omega^1(n)\xrightarrow{d}\cdots\xrightarrow{d}\Omega^{n-1}(n)\xrightarrow{d}\Omega^n(n).\label{Omegasequence}
\end{equation}
The (linearly compact) Lie algebra of formal vector fields
$W(n):=\{\sum_{i=1}^nP_i\de_i\,|\, P_i\in\Omega^0(n)\}$ acts continuously on the superalgebra $\Omega(n)$ via Lie derivatives, this action preserves gradation \eqref{Omegagradation}, hence acts on each member of the sequence \eqref{Omegasequence}, and commutes with $d$. Sequence \eqref{Omegasequence} is called the {\it formal de Rham complex}, and $d$ is called the {\it de Rham differential}.

Taking duals of the linearly compact members of \eqref{Omegasequence}, we obtain the following exact sequence of vector spaces with discrete topology:
\begin{equation}
\Omega^0(n)^*\xleftarrow{d^*}\Omega^1(n)^*\xleftarrow{d^*}\cdots
\xleftarrow{d^*}\Omega^{n-1}(n)^*\xleftarrow{d^*}\Omega^{n}(n)^*,
\label{Omegadual}
\end{equation}
with a continuous action of the Lie algebra $W(n)$.

More explicitly the $W(n)$-modules $\Omega^j(n)^*$ can be expressed as the induced modules:
\begin{equation}
M(\displaywedge^j(\C^{n^*})):=\Ind_{W_{\geq 0}(n)}^{W(n)}\displaywedge^j(\C^{n^*}),
\label{induced}
\end{equation}
where we use the $\Z$-gradation $W(n)=\prod_{j\geq -1}W_j(n)$ defined by $\deg x_i=-\deg\de_i=1$, and the $W_0(n)$-module $\inlinewedge^j(\C^{n^*})$ is defined by the usual action of $W_0(n)\cong
\mathfrak{gl}_n(\C)$, the $W_j(n)$ with $j>0$ acting trivially. It is easy to see that
the differential $d^*$ in \eqref{Omegadual} is given by
\begin{equation}
d^*=\sum_{k=1}^n\de_k\otimes \iota_k,
\label{d*}
\end{equation}
where $\iota_k: \inlinewedge(\C^{n^*})\rightarrow \inlinewedge(\C^{n^*})$ is the odd derivation, defined by $\iota_k(x_j^*)=\delta_{kj}$ (contraction by $x_k$), where $\{x_i\}$ is the standard basis of $\C^n$, and $\{x_i^*\}$ is the dual basis of $\C^{n*}$.

Let $F$ be an irreducible $W_0(n)=\mathfrak{gl}_n(\C)$-module, extended trivially to $W_{>0}(n)$, and let
$$M(F):=\Ind_{W_{\geq 0}(n)}^{W(n)}F$$
be the corresponding induced $W(n)$-module. The famous theorem of Rudakov \cite{R1} states that the $W(n)$-module $M(F)$ is irreducible if and only if it is not one of the modules appearing in \eqref{Omegadual}
(except for the last term); moreover, all non-trivial morphisms between the $W(n)$-modules $M(F)$ appear in \eqref{Omegadual}.

A similar result holds for the other three simple infinite-dimensional linearly compact Lie algebras $S(n)$, $H(n)$, and $K(n)$ in Cartan's classification. For $S(n)$ and $H(n)$ this follows from \cite{R2} and \cite{E}, and for
$K(n)$ this follows from \cite{Ko}, where all the singular vectors were found, and the paper of Rumin \cite{Ru}, where an analogue of  the de Rham complex was constructed for contact manifolds.

The classification of simple infinite-dimensional linearly compact Lie superalgebras over $\C$ is much richer than in the Lie algebra case. It consists of 10 infinite series and 5 exceptional examples, denoted by $E(1,6)$, $E(3,6)$, $E(3,8)$, $E(4,4)$, and $E(5,10)$ \cite{K}. Maximal open subalgebras in all these Lie superalgebras have been classified in \cite{CK}. Unlike the Lie algebra case, which contains a unique maximal open subalgebra, the simple infinite-dimensional linearly compact Lie superalgebras may contain several maximal open subalgebras, however each of them, denote it by $L$, contains essentially a unique maximal open subalgebra $L_{\geq 0}$, invariant with respect to all inner automorphisms of $L$. Moreover, most of them, including all 5 exceptional ones, admit a $\Z$-gradation
\begin{equation}
L=\prod_{j=-d}^\infty L_j,\,\,\mbox{where}\, d=1,2\,\mbox{or}\,3,
\label{maximalgraded}
\end{equation}
by finite-dimensional subspaces.

As above, given a finite-dimensional irreducible $L_0$-module $F$, we consider the induced $L$-module
\begin{equation}
M(F)=\Ind_{L_{\geq 0}}^L F,
\label{inducedmodule}
\end{equation}
where the $L_{\geq 0}$-module $F$ is obtained from the $L_{0}$-module by letting ${L_j}$ act trivially on $F$ for $j>0$.
We call the $L$-module $M(F)$ a {\it Verma module}. The first basic problem is to classify all $F$, for which the $L$-module $M(F)$ is not irreducible; such a Verma module is called {\it degenerate}. 
In fact every continuous irreducible locally finite module over $L$ is a 
quotient of a unique Verma module. By a locally finite module $V$ we mean a module such that $L_{\geq 0}v$ has finite dimension for every $v\in V$. 

The Rudakov classification of degenerate Verma modules over $L=W(n)$ is described above.
The method of classification consists in finding in $M(F)$ all singular vectors of positive degree. Note that the $\Z$-gradation
\eqref{maximalgraded} induces on $M(F)$ the $\Z$-gradation
\begin{equation}
M(F)=\bigoplus_{j\in\Z_{\geq 0}}M(F)_{-j},\,\,\mbox{where}\, M(F)_0=F.
\label{gradedVerma}
\end{equation}
If $v\in M(F)_{-s}$, we say that $v$ has {\it degree} $s$. A non-zero vector $v\in M(F)$ is called {\it singular} if 
\[L_jv=0\,\,\mbox{for}\, j>0.\]
Obviously, all non-zero homogeneous components of a singular vector in gradation \eqref{gradedVerma} are singular as well, and $M(F)$ is degenerate if and only if it has a singular vector of positive degree.

Along these lines, which were used also in \cite{R1} and \cite{R2}, all degenerate Verma modules were classified in \cite{KR1, KR2, KR3}, \cite{KR4},
\cite{BKL1, B}, and \cite{CCK2} for $E(3,6)$, $E(3,8)$, $E(1,6)$, and $E(5,10)$ respectively. The simplifying property in all these cases is that $L_0$ is a reductive Lie algebra, which is not the case for the remaining Lie superalgebra $L=E(4,4)$, studied in the present paper. In this case,
$L_0=\phat$, the non-trivial central extension  of the simple Lie superalgebra $\mathfrak{p}(4)$. It is well-known that the simple Lie superalgebra of the ``strange" series $\mathfrak p(n)$ \cite{K77} admits a non-trivial central extension if and only if $n=4$.

The classification of irreducible modules over $\phat$ was obtained in \cite{S}, and an exposition of it is given in Section 3 of the present paper. In Section 4 we recall the classification of singular vectors of Verma modules over $W(4)$ and $S(4)$ from \cite{R1} and \cite{R2}, which, due to the fact that the even part of $E(4,4)$ is isomorphic to $W(4)$, allows us to prove that the degrees of singular vectors in Verma modules over $E(4,4)$ do not exceed 5 (Proposition \ref{bound}).

In Sections 6--9 we describe all singular vectors of degree 1,2,3 and 4 
(Theorems \ref{listagrado1}, \ref{listagrado2}, \ref{listagrado3}, and 
\ref{listagrado4}). The absence of singular vectors of degree 5 is proved with the help of a computer.

Of course, each singular vector in a Verma module $M$ gives rise to a morphism from another Verma module to $M$. In Proposition \ref{explicitmorphisms}
we describe the morphisms of degree 1 explicitly.

A remarkable, and still unexplained observation is that all degenerate Verma modules for all exceptional $L$ can be arranged in a beautiful picture, consisting of infinitely many sequences of modules and we show that this holds for $L=E(4,4)$ as well (see Figure \ref{all}).

In a small number of cases the composition of two consecutive maps
\[M_1\xrightarrow{\varphi_1} M_2\xrightarrow{\varphi_2}M_3\]
between Verma modules is not zero, but replacing it by $M_1\xrightarrow{\varphi_2\circ\varphi_1} M_3$, we can make the composition of consecutive maps zero, to obtain, what we call the 
{\it exceptional de Rham complex}, associated with $L$.
For $L=E(4,4)$ this complex is presented  in Figures \ref{deRham1} and \ref{deRham2}. An interesting open problem is to compute the cohomology of this complex.

\section{The Lie superalgebra $E(4,4)$}\label{Section2}
The Lie superalgebra $L=E(4,4)$ is defined as follows (\cite[\S 5.3]{CK2}). Its even part $L_{\bar{0}}$ is the Lie algebra $W(4)$ of vector fields in four even indeterminates  $x_1,x_2,x_3,x_4$ with coefficients in $\Omega^0(4)$, the algebra of formal power series over $\C$, and its odd part $L_{\bar{1}}$ is the space $\Omega^1(4)$ of all  one-forms in $x_1,x_2,x_3,x_4$  with coefficients in  $\Omega^0(4)$, with the following action of $L_{\bar 0}$: for $X\in L_{\bar{0}}$,
$\omega\in L_{\bar{1}}$,
$$[X,\omega]=L_X(\omega)-\frac{1}{2}\diver(X)\omega,$$
where $L_X(\omega)$ denotes the Lie derivative of the one-form $\omega$ along the vector field $X$.  Furthermore, for $\omega_1, \omega_2\in L_{\bar{1}}$ the commutator is
$$[\omega_1, \omega_2]=d\omega_1\wedge \omega_2+\omega_1\wedge d\omega_2,$$
where the differential three-form in the right hand side is identified with vector fields via contraction with the standard volume form. 
%In other words, $L_{\bar{1}}$ can be interpreted as the space $\Omega_1(4)^{-\frac{1}{2}}$ of formal sections of the bundle $T^*\otimes K^{-\frac{1}{2}}$ over the formal four disk.

We will use the notation $\de_{i}$ to denote the vector field $\frac{\de}{\de x_i}$, and we will denote by $d_i$ the differential one-form $dx_i$.

The Lie superalgebra $L$ has, up to conjugacy, a unique irreducible $\Z$-grading $L=\prod_{j\geq -1} L_j$, called the principal grading, defined by setting $\deg(x_i)=1$ and $\deg d=-2$ (\cite[Corollary 9.8]{CK}).
The subalgebra $L_0$ is isomorphic to the unique nontrivial central extension $\phat$ of the strange Lie superalgebra $\mathfrak p(4)$ \cite[\S 2.1.3]{K77} (where it is denoted by $\mathfrak p(3)$). This is a simple finite-dimensional Lie superalgebra with the following  consistent irreducible  $\Z$-grading:
$$\mathfrak p(4)=\mathfrak p(4)_{-1}\oplus \mathfrak p(4)_0\oplus \mathfrak p(4)_1$$
where $\mathfrak p(4)_0\cong \slq$ and, as $\slq$-modules, $\mathfrak p(4)_{-1}\cong \inlinewedge^2(\C^{4^*})$ and 
$\mathfrak p(4)_1\cong S^2(\C^4)$. 
Observe that $\inlinewedge^2(\C^{4^*})$ is isomorphic to the standard $\mathfrak{so}_6$-module, with scalar product $(\cdot,\cdot)$; define 
$\varphi:\inlinewedge^2(\mathfrak p(4))\rightarrow \C$ by setting, for $x\in \mathfrak p(4)_i, y\in \mathfrak p(4)_j$,
$$\varphi(x,y):=\left\{\begin{array}{cc}
	(x,y) & {\mbox{if}}\, i=j=-1,\\
	0 & {\mbox{otherwise.}}
\end{array}\right.$$
Then $\phat$ is the central extension of $\mathfrak p(4)$ defined by this cocycle (\cite[Example 3.6]{K}, \cite{S}).
Therefore one has the following $\Z$-graded Lie superalgebra
$$\phat=\phat_{-2}\oplus\phat_{-1}\oplus \phat_0\oplus \phat_1,$$
where $\phat_{-2}=\C C$ is central. The element $C$ acts on $L$ as a grading element with respect to the principal grading, i.e., $[C,a]=ja$
for $a\in L_j$.

In the isomorphism $L_0\cong \phat$ we have that $\phat_{1}$ is spanned by the closed differential one-forms \[b_{ij}=x_id_j+x_jd_i\] (with $i,j=1,2,3,4$), $\phat_0\cong \slq$ is spanned by the vector fields $x_i\de_{j}$ and $h_{ij}=x_i\de_{i}-x_j\de_{j}$ (with $i,j=1,2,3,4$, $i\neq j$), $\phat_{-1}$ is spanned by the differential one-forms \[a_{ij}=x_id_j-x_jd_i\] (with $i,j=1,2,3,4)$ and $\phat_{-2}$ is spanned by $C=x_1\de_{1}+x_2\de_{2}+x_3\de_{3}+x_4\de_{4}$. Note that $\phat_0\oplus \phat_{-2}\cong \mathfrak{gl}_4$.

\bigskip
Throughout the paper we shall use the following explicit formulas for the brackets. First, $[b_{ij}, b_{lm}]=0$. Second,  the brackets between elements $b_{ij}\in \phat_1$ and $a_{lm}\in \phat_{-1}$ are given by the following rules:
\begin{itemize}
	\item $[b_{ij},a_{lm}]=0$ if $|\{i,j,l,m\}|\leq 2$
	\item $[b_{ij},a_{lm}]=2\epsilon_{ijlm} h_{ij}$ if $|\{i,j,l,m\}|= 4$;
	\item $[b_{jj},a_{lm}]=4\epsilon_{ijlm} x_j \de_i$;
	\item $[b_{ij},a_{im}]=-2\epsilon _{ijlm} x_i \de_l$.
\end{itemize}
Here we denote by $\epsilon_{ijlm}$  the sign of the permutation ${ijlm}$, where we mean that $\epsilon_{ijlm}=0$ whenever two of the indices coincide.
In particular we have
\begin{itemize}
	\item $[b_{44},a_{12}]=4x_4\de_3$;
	\item $[b_{44},a_{13}]=-4x_4\de_2$;
	\item $[b_{44},a_{23}]=4x_4 \de_1$.
\end{itemize}

Finally, $[a_{ij},a_{lm}]=2\epsilon_{ijlm}C$.

Let us fix the Borel subalgebra $B=\langle x_i\partial_j, h_{ij}  ~|~ i<j\rangle$ of $\phat_0=\mathfrak{sl}_4$ and  consider the usual set of simple roots of the corresponding root system, denoted by $\{\alpha_{12},\alpha_{23},\alpha_{34}\}$. We let $\Lambda$ be the weight lattice of $\mathfrak{sl}_4$ and we express all weights of $\mathfrak{sl}_4$ using their coordinates with respect to the fundamental weights $\omega_{12},\omega_{23},\omega_{34}$, i.e., for $\lambda\in \Lambda$ we write $\lambda=(\lambda_{12},\lambda_{23},\lambda_{34})$ for some $\lambda_{i\,i+1}\in \mathbb Z$ to mean $\lambda=\lambda_{12}\omega_{12}+\lambda_{23}\omega_{23}+\lambda_{34}\omega_{34}$.
If $\lambda=(a,b,c)\in \Lambda$ is a dominant weight, i.e., $a,b,c\in\Z_{\geq 0}$ and $t\in \C$, we shall denote by $F_t(a,b,c)$ the irreducible $\mathfrak{gl}_4$-module of highest weight $(a,b,c)$ where the central element $C$ acts as multiplication by $t$.  

The component $L_{-1}$ has superdimension $(4|4)$ and is an irreducible $\phat$-module, spanned by $\de_1,\de_2,\de_3,\de_4,d_1,d_2,d_3,d_4$, which is isomorphic as a $\mathfrak{gl}_4$-module to $F_{-1}(1,0,0)\oplus F_{-1}(0,0,1)$. 
The component $L_1$  is also an irreducible $\phat$-module and, as a $\mathfrak{gl}_4$-module, it decomposes as follows:
$L_1\cong F_1(2,0,1)\oplus F_1(1,0,0)\oplus F_1(3,0,0)\oplus F_1(1,1,0)$.

Recall that by Weyl dimension formula we have
\[
\dim F_t(a,b,c)=\frac{1}{12}(a+1)(b+1)(c+1)(a+b+2)(b+c+2)(a+b+c+3).
\]
We let $\mathfrak p^+= \phat_{-2}\oplus \phat_0 \oplus \phat_1$ and we extend the action of $\mathfrak{gl}_4=\phat_{-2}+\phat_0$ on $F_t(a,b,c)$ to $\mathfrak p^+$  letting $\phat_1$ act trivially.
We let 
\[
K_t(a,b,c)=\mathcal U(\phat)\otimes _{\mathcal U(\mathfrak p^+ )}F_t(a,b,c)
\]
be the corresponding (thin) Kac module. This is a finite dimensional vector super space
with $F_t(a,b,c)$ an even subspace, and as a $\slq$-module we have
\begin{equation}\label{uno}
K_t(a,b,c)\cong \mathcal U(\phat_{-1})\otimes_{\C} F_t(a,b,c).
\end{equation}
Every irreducible finite dimensional $\phat$-module is a quotient of a Kac module $K_t(a,b,c)$. We will dedicate the following section to a description of these modules.

\section{Irreducible $\phat$-modules}
In this section we provide an explicit description of some results contained 
in \cite{FT} for $t=0$ and in \cite{S} for $t\neq 0$. We refer to these papers for all details although we will use slightly different notation.

It is straightforward that if $t=0$ the Kac module $K_0(a,b,c)$ has a unique irreducible quotient that we denote by $W_0(a,b,c)$. Nevertheless, if $t\neq 0$,  the situation is different as it is described by the following proposition which can be deduced from \cite[\S 4]{S}.

\begin{proposition}\label{unique} If $t\neq 0$ the Kac module $K_t(a,b,c)$ has a unique irreducible quotient unless $a=c=0$ and $b\geq 1$. Moreover, for all $b\geq 2$, $K_t(0,b,0)$ is the direct sum of two irreducible submodules and $K_t(0,1,0)$ is the direct sum of two submodules one of which is irreducible.
\end{proposition}

We will denote by $W_t(a,b,c)$ the irreducible quotient of $K_t(a,b,c)$ when it is unique. The Kac modules $K_t(0,b-1,0)$ and $K_t(0, b,0)$, for $b\geq 2$, share an irreducible component that we denote by $W_t(0,b,0)$ so that for $b\geq 2$,
\[K_t(0,b,0)\cong W_t(0,b,0)\oplus W_t(0,b+1,0).\]
Finally $K_t(0,1,0)$ is the direct sum of $W_t(0,2,0)$ and a submodule $S$ which is not irreducible. We will denote by $W_t(0,1,0)$ the unique irreducible quotient of $S$.

\begin{example}\label{homcomp}
For $d=-1$ and $d>0$ the homogeneous component $E(4,4)_d$ of $E(4,4)$ in the principal grading is the irreducible $\phat$-module $W_d(d+2,0,0)$. %This module has a unique and  the unique $\phat$-singular vector (i.e.\ eigenvector of $B$) in $E(4,4)_d$ is $x_1^{d+1}d_1$. 

For $d=0$ the adjoint module $E(4,4)_0$ is not irreducible: nevertheless we have $[E(4,4)_0,E(4,4)_0]$ $\cong W_0(2,0,0)$ and so $E(4,4)_0\cong W_0(2,0,0)\oplus W_0(0,0,0)$.
\end{example}

\begin{example}[The standard representation of $\phat$]\label{standard}
The irreducible representation $W_t(1,0,0)$ can be easily described as follows for all $t\in \C$. Recall by the previous example that $W_{-1}(1,0,0)\cong E(4,4)_{-1}$. Now we want to realize  $W_t(1,0,0)$ as a deformation of it. So we let $W_t(1,0,0)$ be $E(4,4)_{-1}$ as a vector super space. The action of $\phat$ on $W_t(1,0,0)$ (denoted here by $\cdot$) is given by the following simple rules:
\begin{itemize}
	\item $C\cdot v=tv$ for all $v\in E(4,4)_{-1}$;
	\item $a\cdot v=-t[a,v]$ for all $a\in \phat_{-1}$ and for all even vectors $v\in E(4,4)_{-1}$;
	\item $a\cdot v=[a,v]$ for all $a\in \phat_{-1}$ and for all odd vectors $v\in E(4,4)_{-1}$;
	\item $b\cdot v=[b,v]$ for all $b\in \phat_{\geq 0}$ and all $v\in E(4,4)_{-1}$.	
\end{itemize}
\end{example}
\begin{example}\label{Wtd+2}
Examples \ref{homcomp} and \ref{standard} can be simultaneously generalized as follows: for $d=-1$ or $d> 0$ the irreducible representation $W_t(d+2,0,0)$ is a deformation of $E(4,4)_d$ given by the following rules:
\begin{itemize}
	\item $C\cdot v=tv$ for all $v\in E(4,4)_{d}$;
	\item $a\cdot v=\frac{t}{d}[a,v]$ for all $a\in \phat_{-1}$ and for all even vectors $v\in E(4,4)_{d}$;
	\item $a\cdot v=[a,v]$ for all $a\in \phat_{-1}$ and for all odd vectors $v\in E(4,4)_{d}$;
	\item $b\cdot v=[b,v]$ for all $b\in \phat_{\geq 0}$ and all $v\in E(4,4)_{d}$.	
\end{itemize}
\end{example}
\begin{example}\label{Wt200}
The special case $W_t(2,0,0)$ for $t\neq 0$ can also be obtained as a deformation of $E(4,4)_0$, the adjoint representation of $\phat$, although we recall that the latter is not irreducible. In order to define such deformation we recall from \cite{CaCaKac25} the definition of the integral of a $k$-form: if $\omega \in \Omega^k(4)$ is a $k$-form whose coefficients are homogeneous polynomials of degree $d$ we let 
\[
\tiny \int \omega=\frac{1}{d+k}\iota_C\omega,
\]
where $\iota_C(\omega)$ denotes the contraction of $\omega$ along the vector field $C$. It turns out that $\int$  commutes with the action of $\slq$ on $\Omega(4)$, $\int^2=0$ and $d\int+\int d=id$. In particular one can verify that $\int b_{ij}=x_ix_j$ and $\int a_{ij}=0$ for all $i,j=1,\ldots,4$.
The irreducible representation $W_t(2,0,0)$, for $t\neq 0$ can be obtained as a deformation of $\phat$ in the following way: for all $X\in \phat_{\bar 0}$ with $\diver X=0$, $\omega,\omega'\in \phat_{\bar 1} $ and $v\in \phat$ we let
\begin{itemize}
	\item $C\cdot v=tv$;
	\item $X\cdot v=[X,v]$;
	\item $\omega \cdot X=[\omega,X]-\frac{t}{2}d\iota_X(\int d \omega)$;
	\item $\omega \cdot C=\frac{t}{2}\int d\omega-\frac{t}{2}d\int \omega$;
	\item $\omega \cdot \omega'=[\omega,\omega']+\frac{t}{4}\int d\omega\wedge d\omega'-\frac{t}{4}d\omega\wedge \int d\omega'$.
\end{itemize}
\end{example}
The proof that the definitions given in Examples \ref{Wtd+2} and \ref{Wt200} provids genuine $\phat$-representations is a long and tedious but elementary  verification.
\begin{remark}\label{defor}
If $t,s\neq 0$, the irreducible module $W_s(a,b,c)$ can also be obtained from $W_t(a,b,c)$ by means of a deformation. Namely, let $u\in \C$ be a square root of $t/s$. Then for all $v\in W_t(a,b,c)$ and $x\in \phat_d$ we have
\[
\phi_s(x)v=u^d\phi_t(x)v,
\]
where $\phi_t$ and $\phi_s$ denote the action of $\phat$ on $W_t(a,b,c)$ and $W_s(a,b,c)$ respectively.
\end{remark}

Following Serganova \cite{S}, we let $\mathfrak h=\phat_{\leq 0}$ and $Rep^t_{\mathfrak h}$ be the category of finite-dimensional $\mathfrak h$-modules where $C$ acts as multiplication by $t$. It turns out that if $t\neq 0$ then  $Rep^t_{\mathfrak h}$ is semisimple (\cite{S}). Let $V_t$ be the restriction of  the standard $\phat$-module $W_t(1,0,0)$ to $\mathfrak h$ and let $\tilde{F}_t(a,b,c)=F(a,b,c)\otimes V_t$ where the action of $\slq$ on $F(a,b,c)$ is extended trivially to $\mathfrak h$. Then, if $t\neq 0$, $\tilde{F}_t(a,b,c)$ is an irreducible $\mathfrak h$-module and every irreducible $\mathfrak h$-module is isomorphic to a unique 
$\tilde{F}_t(a,b,c)$ (\cite{S}).

We now want to define some morphisms between Kac modules which will be useful in the description of their composition series and of their irreducible quotients (see also \cite[\S 4,5]{S}). We point out that a morphism $f: K_t(\lambda) \rightarrow K_t(\mu)$ is completely determined by the image of the highest weight vector $v_{\lambda}$ of  $F_t(\lambda)$. We note that $f(v_{\lambda})$ satifies the following properties:
\begin{itemize}
	\item $(x_i\de_{i+1})w=0$ for all $i=1,2,3$;
	\item $bw=0$ for all $b\in \phat_1$.
\end{itemize}
We call such a vector a $\phat$-singular vector.
Conversely, any $\phat$-singular vector $w$ of weight $\lambda$ in a $\phat$-module $M$, such that $Cw=tw$ for some $t\in\C$,  gives rise to a morphism of $\phat$-modules  $K_t(\lambda)\rightarrow M$. 

Recall that $K_t(a,b,c)^*\cong K_{-t}(c,b,a)$ (see \cite[Lemma 4.4]{S}).
\begin{definition}\label{morfismipi4}
Let $\{x_1, \dots, x_4\}$ be the standard basis of the irreducible $\slq$-module 
$\C^4$,  $\{x_1^*, \dots, x_4^*\}$ its dual basis, and $x_{ij}=x_i\wedge x_j\in \inlinewedge^2(\C^4)$.
The following vectors are $\phat$-singular vectors in Kac modules:
\begin{itemize}
\item[i)] $a_{12}v$ in $K_t(a,b-1,0)$ with $a\geq 0$, $b\geq 1$, where 
$v=x_1^ax_{12}^{b-1}$ is the highest weight vector in $F_t(a,b-1,0)$;
\item[ii)] $a_{12}a_{13}a_{14}v$ in $K_t(a-2,0,0)$ with $a\geq 2$,
where $v=x_1^{a-2}$ is the highest weight vector in $F_t(a-2,0,0)$;
\item[iii)] $a_{14}x_1^{a-1}(x_4^*)^{c+1}+a_{13}x_1^{a-1}x_3^*(x_4^*)^{c}
+a_{12}x_1^{a-1}x_2^*(x_4^*)^{c}$ in $K_t(a-1,0,c+1)$ with $a\geq 1$, $c\geq 0$.
\end{itemize}
Therefore we have the following list of morphisms:
\begin{itemize}
\item[a)] $\theta_{a,b}: K_t(a,b,0)\rightarrow K_t(a,b-1,0)$, with $a\geq 0$, $b\geq 1$, determined by the $\phat$-singular vector in i);
\item[b)] $\varphi_{a}: K_t(a,0,0)\rightarrow K_t(a-2,0,0)$, with $a\geq 2$, determined by the $\phat$-singular vector in ii);
\item[c)] $\eta_{a,c}: K_t(a,0,c) \rightarrow K_t(a-1,0,c+1)$, with $a\geq 1$, $c\geq 0$, determined by the $\phat$-singular vector in iii);
\item[d)] $\xi_{b,c}: K_t(0,b,c)\rightarrow K_t(0,b+1,c)$, with $b,c\geq 0$, dual of $\theta_{c,b+1}$;
\item[e)] $\psi_{c}: K_t(0,0,c) \rightarrow K_t(0,0,c+2)$, with $c\geq 0$, dual of
$\varphi_{c+2}$.
\end{itemize}
\end{definition}

\begin{proposition}\label{quozienti}  The following is a description of all the irreducible quotients of Kac modules
for all $t\neq 0$: 
\begin{itemize}
\item[1)] $W_t(a,b-1,0)\cong K_t(a,b-1,0)/Im(\theta_{a,b})$ for $b\geq 3$ or $a\geq 1$ and $b=1$;
\item[2)] $W_t(0,1,0)\cong K_t(0,1,0)/\ker(\psi_0\theta_{0,1})$;
\item[3)] $W_t(a,1,0)\cong K_t(a,1,0)/\ker(\eta_{a,0}\theta_{a,1})$ for $a\geq 1$;
\item[4)] $W_t(0,0,0)\cong K_t(0,0,0)/Im(\varphi_2)$;
\item[5)] $W_t(a,0,c)\cong K_t(a,0,c)/Im(\eta_{a+1,c-1})$ for $a\neq 1$ and $c\geq 1$;
\item[6)] $W_t(1,0,c)\cong K_t(1,0,c)/\ker(\xi_{0,c+1}\eta_{1,c})$;
\item[7)] $W_t(0,b,c)\cong K_t(0,b,c)/Im(\xi_{b-1,c})$ for $b\geq 1$ and $c\geq 1$.
\end{itemize} 
\end{proposition}
\begin{proof}
All cases are  given in \cite{S} except for 4). In this case  we have $K_t(0,0,0)=\tilde F_t(1,0,0)\oplus \tilde F_t(0,0,1)$ as $\mathfrak h$-modules. Since $\tilde F_t(0,0,1)$ is not an $\mathfrak h$-submodule  of $K_t(2,0,0)$ the image of $\varphi_{2}$ is necessarily $\tilde F_t(1,0,0)$, therefore $K_t(0,0,0)/Im(\varphi_2)$ is isomorphic, as an $\mathfrak{h}$-module, to $\tilde{F}_t(1,0,0)$ and is therefore irreducible also as a $\phat$-module. Note that $Im(\varphi_2)$ is the unique maximal $\phat$-submodule of $K_t(0,0,0)$ by Proposition \ref{unique}.
\end{proof}

Thanks to the decomposition of Kac modules into direct sums of irreducible $\mathfrak h$-submodules and to the decription of morphisms between Kac modules, the following decompositions can be deduced.
\begin{proposition}\label{decomposizioniK}
For $t\neq 0$ the following identities hold in the Grothendieck group of finite-dimensional representations of $\phat$:
\begin{enumerate}
\item $[K_t(0,0,0)]=[W_t(0,0,0)]+[W_t(2,0,0)]$;
\item $[K_t(0,1,0)]=[W_t(0,1,0)]+[W_t(2,0,0)]+[W_t(0,2,0)]$;
\item $[K_t(a,0,0)]=[W_t(a,0,0)]+[W_t(a,1,0)]+[W_t(a+2,0,0)]$ for $a\geq 1$;
\item $[K_t(0,0,1)]=[W_t(0,0,1)]+[W_t(1,0,0)]+[W_t(1,1,0)]$;
\item $[K_t(0,0,c)]=[W_t(0,0,c)]+[W_t(1,0,c-1)]+[W_t(0,0,c-2)]$ for $c\geq 2$;
\item $[K_t(0,1,1)]=[W_t(0,1,1)]+[W_t(1,1,0)]+[W_t(0,0,1)]$;
\item $[K_t(1,0,1)]=[W_t(1,0,1)]+[W_t(2,0,0)]+[W_t(2,1,0)]+[W_t(0,0,0)]$;
\item $[K_t(1,0,c)]=[W_t(1,0,c)]+[W_t(2,0,c-1)]+[W_t(0,0,c-1)]$ for $c\geq 2$;
\item $[K_t(a,0,1)]=[W_t(a,0,1)]+[W_t(a+1,0,0)]+[W_t(a+1,1,0)]$ for $a\geq 2$;
\item $[K_t(0,1,c)]=[W_t(0,1,c)]+[W_t(0,0,c)]+[W_t(1,0,c-1)]$ for $c\geq 2$;
\item $[K_t(a,1,0)]=[W_t(a,1,0)]+[W_t(a,2,0)]+[W_t(a+2,0,0)]$ for $a\geq 0$;
\item $[K_t(a,b,0)]=[W_t(a,b,0)]+[W_t(a,b+1,0)]$ for $a\geq 0$, $b\geq 2$;
\item $[K_t(a,0,c)]=[W_t(a,0,c)]+[W_t(a+1,0,c-1)]$ for $a\geq 2$, $c\geq 2$;
\item $[K_t(0,b,c)]=[W_t(0,b,c)]+[W_t(0,b-1,c)]$ for $b\geq 2$, $c\geq 1$;
\item $[K_t(a,b,c)]=[W_t(a,b,c)]$ for $abc\neq 0$.
\end{enumerate}
\end{proposition}

\begin{remark}\label{decomposizioniK0}
The decomposition of $K_0(a,b,c)$ in the Grothendieck group is given in \cite[p.\ 367, Eqns 1--19] {S}. Note that in \cite{S} the isomorphism $\phat_0\cong \mathfrak{so}(6)$ is used to describe weights, therefore $K_t(a,b,c)$ is denoted by $K_t(c,b,a)$ in \cite{S}.
Moreover $W_0(a,b,c)$ is denoted by $V_0(c,b,a)$. We also point out that Equations 14., 15. and 18. in \cite{S} should be fixed as follows (in our notation):
\begin{itemize}
\item[14.] $[K_0(0,0,1)]=[W_0(0,0,1)]+2[W_0(1,0,0)]+[W_0(1,1,0)]$;
\item[15.] $[K_0(1,0,0)]=2[W_0(1,0,0)]+[W_0(1,1,0)]+[W_0(3,0,0)]$;
\item[18.] $[K_0(0,1,0)]=2[W_0(0,1,0)]+[W_0(0,2,0)]+2[W_0(0,0,0)]+2[W_0(2,0,0)]$;
\end{itemize}
\end{remark}

\begin{proposition}\label{hmodules}
For $t\neq 0$ the irreducible $\phat$-modules $W_t(a,b,c)$ decompose as direct sums of  irreducible $\mathfrak h$-modules as follows:
\begin{itemize}	
\item $W_t(a,0,0)=\tilde F_t(a-1,0,0)$ for all $a\geq 1$;
\item $W_t(0,0,c)=\tilde F_t (0,0,c+1)$ for all $c\geq 0$;
%\item $V_t(0,1,0)=\tilde L(0,0,1)$ (I think $=V_t(0,0,0)$ as $\g$-modules);
%\item $W_t(0,b,0)=\tilde L(1,b-1,0)\oplus \tilde L(0,b-1,1)$ for all $b\geq 2$;
\item $W_t(a,1,0)=\tilde  F_t(a-1,1,0)+\tilde F_t(a,0,1)$ for all $a\geq 1$;
\item $W_t(1,0,c)=\tilde F_t(1,0,c+1)+\tilde F_t(0,1,c)$ for all $c\geq 1$;
%\item $W_t(0,1,c)=\tilde L(1,1,c)+\tilde L(0,2,c-1)+\tilde L(0,1,c+1)$ for all $c\geq 1$;
%\item $W_t(a,0,1)=\tilde L(a-1,1,1)+\tilde L(a-1,0,1)+\tilde L(a,0,2)$ for all $a\geq 0$;
\item $W_t(a,b,0)=\tilde F_t(a+1,b-1,0)+\tilde F_t(a-1,b,0)+\tilde F_t(a,b-1,1)$ for all $a\geq 0$ and  $b\geq 2$;
\item $W_t(a,0,c)=\tilde F_t(a-1,1,c)+\tilde F_t(a-1,0,c)+\tilde F_t(a,0,c+1)$ for all $a\geq 2$ and $c\geq 1$;
\item $W_t(0,b,c)=\tilde F_t(1,b,c)+\tilde F_t(0,b+1,c-1)+\tilde F_t(0,b,c+1)$ for all $b,c\geq 1$.
\end{itemize}
\end{proposition}
\begin{proof}
	This result is proved in \cite[\S 5]{S} with different notation. Namely one has $V_t(0,b,a)=W_t(a,b,0)$ ($b\geq 2$, $a\geq 1$), $W_t(0,b,c)=V_t(c,b+1,0)$ ($b,c\geq 1$), $W_t(a,0,c)=V_t(c,0,a-1)$ ($a\geq 2, c\geq 1$), $W_t(0,0,c)=V_t(c+1,0,0)$ ($c\geq 0$), $W_t(a,0,0)=V_t(0,0,a-1)$ ($a\geq 1$), $W_t(a,1,0)=V_t(0,1,a)$ ($a\geq 1$), $W_t(1,0,c)=V_t(c+1,1,0)$ ($c\geq 1$), $W_t(0,b,0)=V_t(0,b-1,0)$, ($b\geq 2$). Moreover, for all $a,b,c\geq 0$, the module $\tilde F_t(a,b,c)$ is denoted  in \cite{S} by $\tilde L_t(c,b,a)$.
\end{proof}
The reason why we preferred to change notation is that in this way one has that the irreducible module $W_t(a,b,c)$ is a quotient of the Kac module $K_t(a,b,c)$ and therefore it has a $\phat$-singular vector of weight $(a,b,c)$. This singular vector is unique with one exception: for $t\neq 0$ the irreducible $\phat$-module $W_t(0,0,0)$ has one more $\phat$-singular vector of weight $(0,1,0)$. Indeed, the modules $W_t(a,b,c)$ are all non isomorphic with the unique exception $W_t(0,0,0)\cong W_t(0,1,0)$ for all $t\neq 0$. This follows from the decomposition given in Proposition \ref{hmodules} and the following observation.
\begin{remark}\label{a12} According to Definition \ref{morfismipi4}, if $t\neq 0$, then $W_t(0,0,0)$ has a non-zero singular vector $a_{12}v$. Indeed this vector generates $W_t(0,0,0)$ since the vector $a_{34}v$ is generated by $a_{12}v$ and  $(a_{12}a_{34}+a_{34}a_{12})v=2tv$.
	Hence we have a surjective map $K_t(0,1,0)\rightarrow W_t(0,0,0)$, and $W_t(0,1,0)$ is isomorphic to $W_t(0,0,0)$ since $W_t(0,1,0)$ is the unique irreducible quotient of $K_t(0,1,0)$. 
\end{remark}

The description of the irreducible quotients of Kac modules, their decomposition in the Grothendieck group as $\phat$-modules and morphisms between Kac modules  are summarized in Figure \ref{DISEGNO}.

\begin{remark}\label{K0(000)} Here we explicitly describe the $\phat$-module $K_0(0,0,0)$. This module contains the following three $\phat$-singular vectors: $z_1=a_{12}$, $z_2=a_{12}a_{13}a_{14}$ and 
$z_3=a_{12}a_{13}a_{14}a_{23}a_{24}a_{34}$. Let us denote by $M_1$, $M_2$ and $M_3$ the submodules generated by $z_1$, $z_2$ and $z_3$, respectively. We have the following isomorphisms:
\begin{itemize}
\item[a)] $K_0(0,0,0)/M_1\cong M_3\cong W_0(0,0,0)$ (the trivial module);
\item[b)] $M_1/M_2\cong W_0(0,1,0)$;
\item[c)] $M_2/M_3\cong W_0(2,0,0)$.
\end{itemize}
\end{remark}

The grading of the Kac module $K_0(a,b,c)=\mathcal{U}(\phat_{-1})\otimes F(a,b,c)$ with respect to $\mathcal{U}(\phat_{-1})$ induces a grading on $W_0(a,b,c)$.
We now give the graded decomposition of the irreducible modules $W_0(a,b,c)$ into irreducible $\slq$-modules. In this description the term $F(a,b,c)q^d$ denotes an irreducible submodule $F(a,b,c)$ of degree $d$.
\begin{proposition}\label{sl4modules} The irreducible $\phat$-modules $W_0(a,b,c)$ decompose as direct sums of irreducible $\slq$-modules as follows:
\begin{align*}
	  W_0(a,0,0)&=F(a,0,0)+F(a-1,0,1)q+F(a-2,1,0)q^2+F(a-2,0,0)q^3,\,\,  \text{for}\,\, a\neq 2;\\
	 W_0(2,0,0)&=F(2,0,0)+F(1,0,1)q+F(0,1,0)q^2;&\\
	W_0(0,b,0)&=F(0,b,0)+F(1,b-1,1)q+(F(2,b-1,0)+F(0,b-1,2))q^2+F(1,b-1,1)q^3&\\
	&+F(0,b,0)q^4,\,\,  \text{for}\,\, b\geq 2;\\
	W_0(0,1,0)&=F(0,1,0)+F(1,0,1)q+F(0,0,2)q^2;&\\
	W_0(0,0,c)&=F(0,0,c)+F(0,1,c)q+F(1,0,c+1)q^2+F(0,0,c+2)q^3,\,\, \text{for}\,\, c\geq 1;\\
			W_0(a,b,0)&=F(a,b,0)+(F(a+1,b-1,1)+F(a-1,b,1)+F(a,b-1,0))q&\\
		&+(F(a+2,b-1,0)+F(a,b,0)+F(a+1,b-2,1)+F(a,b-1,2)&\\
		&+F(a-2,b+1,0)+F(a-1,b-1,1))q^2
		+(F(a,b-2,2)+F(a+1,b-1,1)&\\
		&+F(a+2,b-2,0)+F(a-1,b,1)+F(a-2,b,0)+F(a,b-1,0))q^3&\\
		&+(F(a,b,0)+F(a+1,b-2,1)+F(a-1,b-1,1))q^4+F(a,b-1,0),\,\,  \text{for}\,\, a\geq 1, b\geq 2;\\
			W_0(a,1,0)&=F(a,1,0)+(F(a+1,0,1)+F(a-1,1,1)+F(a,0,0))q&\\
		&+(F(a,1,0)+F(a,0,2)+F(a-2,2,0)+F(a-1,0,1))q^2&\\
		& +(F(a-1,1,1)+F(a-2,1,0)+F(a,0,0))q^3+F(a-1,0,1))q^4 ,\,\,\text{for}\,\, a\geq 2;\\
 W_0(1,1,0)&=F(1,1,0)+(F(2,0,1)+F(0,1,1)+F(1,0,0))q&\\
		&+(F(1,1,0)+F(1,0,2)+F(0,0,1))q^2+F(0,1,1)q^3\\
	W_0(0,b,c)&=F(0,b,c)+(F(0,b+1,c)+F(1,b-1,c+1)+F(1,b,c-1))q&\\
		&+(F(1,b,c+1)+F(1,b+1,c-1)+F(2,b-1,c)+F(0,b-1,c+2)&\\
		&+F(0,b,c)+F(0,b+1,c-2))q^2
		+(F(2,b,c)+F(1,b-1,c+1)&\\
		&+F(1,b,c-1)+F(0,b,c+2)+F(0,b+2,c-2)+F(0,b+1,c))q^3\\&+(F(1,b,c+1)+F(1,b+1,c-1)+F(0,b,c))q^4+F(0,b+1,c)q^5,\,\, \text{for}\,\, b,c\geq 1;\\		
		W_0(1,0,c)&= F(1,0,c)+(F(1,1,c)+F(0,1,c-1)+F(0,0,c+1))q&\\
		&+(F(2,0,c+1)+F(0,1,c+1)+F(0,2,c-1)+F(1,0,c))q^2&\\
		&= (F(1,0,c+2)+F(1,1,c)+F(0,0,c+1))q^3+F(0,1,c+1)q^4,\,\, \text{for}\,\, c\geq 1;\\
		W_0(a,0,c)&=F(a,0,c)+(F(a,1,c)+F(a-1,1,c-1)+F(a-1,0,c+1))q &\\
		&+(F(a+1,0,c+1)+F(a-1,1,c+1)+F(a-1,2,c-1)+F(a,0,c)&\\
		&+F(a-2,1,c)+F(a-1,0,c-1))q^2
		+(F(a-2,2,c)+F(a,0,c+2)&\\
		&+F(a,1,c)+F(a-1,1,c-1)+F(a-1,0,c+1)+F(a-2,0,c))q^3&\\
		&+(F(a-1,1,c+1)+F(a,0,c)+F(a-2,1,c))q^4+F(a-1,0,c+1)q^5,\,\, \text{for}\,\, a\geq 2, c\geq 1.
\end{align*}
\end{proposition}
\begin{proof} Using Littlewood-Richardson rule one can decompose $K_0(a,b,c)=\mathcal{U}(\phat_{-1})\otimes F(a,b,c)$ into a direct sum of irreducible $\slq$-modules.
	By \cite[Theorem 4.10]{S} the decomposition of the modules $W_0(a,b,c)$ is uniquely determined by that of $K_0(a,b,c)$.
\end{proof}
\begin{corollary}\label{duali}
The following isomorphisms hold for every $t\in\C$.
\begin{itemize}
	\item $W_t(1,0,0)^*=W_{-t}(1,0,0)$;
	\item $W_t(2,0,0)^*=W_{-t}(0,1,0)$;
	\item $W_t(a,0,0)^*=W_{-t}(0,0,a-2)$ for all $a\geq 3$;
	\item $W_t(0,b,0)^*=W_{-t}(0,b,0)$ for all $b\geq 2$;
	\item $W_t(1,1,0)^*=W_{-t}(1,1,0)$;
	\item $W_t(a,1,0)^*=W_{-t}(1,0,a-1)$ for all $a\geq 2$;
	\item $W_t(a,b,0)^*=W_{-t}(0,b-1,a)$ for all $a\geq 1,\,b\geq 2$;
	\item $W_t(a,0,c)^*=W_{-t}(c+1,0,a-1)$ for all $a\geq 2,\, c\geq 1$.
\end{itemize}
\end{corollary}
\begin{proof}
	The duality  for $t=0$ follows from Proposition \ref{sl4modules}, due to the compatibility with the decomposition of $W_0(a,b,c)$ as a direct sum of irreducible $\slq$-modules, and for $t\neq 0$ from Proposition \ref{hmodules}, due to the compatibility with the decomposition of $W_t(a,b,c)$ as a direct sum of irreducible $\mathfrak{h}$-modules.  
\end{proof}

%\begin{lemma}\label{Wt(a00)}
%Let $a\geq 1$ and $z$ be the $\phat$-singular vector of $W_t(a,0,0)$. Then $a_{14}z=0$.
%For $a\geq 1$  let $z$ we have $W_t(a,0,0)\cong K_t(a,0,0)/Im(\theta_{a1})$ if and only if $(a,t)\neq (2,0)$. In particular we have $a_{12}z=0$ in $W_t(a,0,0)$.
%\end{lemma}
%\begin{proof}
%If $t\neq 0$ this is a consequence of Example \ref{homcomp} and Remark \ref{defor} since $a_{14}z=[a_{14},x_1^{a-1}d_1]=0$. Similarly, if $t=0$ and $a=1$, $a_{14}z=0$ by Example \ref{standard}. If $t=0$ and $a\geq 2$, we note that $K_0(a-2,0,0)$ contains the $\phat$-singular vector  $w=a_{12}a_{13}a_{14}x_1^{a-2}$ of weight $(a,0,0)$, hence $W_0(a,0,0)$ can be realized as the quotient of the submodule generated by $w$, therefore it is clear that $a_{14}z=0$, since $z$ is the projection of $w$ onto $W_0(a,0,0)$.
 %\end{proof}

\begin{landscape}
\begin{figure}
\hspace*{-2cm}
%\begin{center}
\begin{tikzpicture}[
    scale=0.58, transform shape,      % adjust to taste (0.62 gives readable text and fits A4)
    %rotate=-180,                       % 90 degrees clockwise rotation for everything
    box/.style={rectangle, draw=blue!75!black, thick, rounded corners,
        minimum width=6.2cm, minimum height=1.4cm, align=center, inner sep=3pt},
    smallbox/.style={rectangle, draw=blue!75!black, thick, rounded corners,
        minimum width=4.4cm, minimum height=1.25cm, align=center, inner sep=2pt},
        %arrow/.style={-{Stealth[length=1.6mm]}, very thick, green!65!black},
    arrowg/.style={-{Stealth[length=3.2mm]}, very thick, green!65!black},
    arrowb/.style={-{Stealth[length=3.2mm]}, very thick, blue!65!black},
    arrowk/.style={-{Stealth[length=2.6mm]}, thick, black!60!black},
    lab/.style={font=\small}
]

% ---------- Row spacing helpers ----------
% We'll define y-coordinates for rows to place nodes precisely
\def\rowA{5.0}   % Row 1 (top)
\def\rowB{1.5}   % Row 2
\def\rowC{-2.0}   % Row 3
\def\rowD{-5.5}   % Row 4
\def\rowE{-9.0}  % Row 5
\def\rowF{-12.5}  % Row 6
\def\rowG{-15.5}  % Row 7
\def\rowH{-3.0}   % Row 3
\def\rowI{-15}

% column x positions (centered grid)
\def\xone{-34.3}
\def\xtwo{ -25.8}
\def\xthr{  -17.6}
\def\xfo{ -8.4}
\def\xfi{ -1}
\def\xsi{4}
\def\xsev{-23.8}

% -----------------------
% Row 1 : 3 boxes
% -----------------------
\node[box] (R1C1) at (\xtwo,\rowA) {$\displaystyle K_t(0,2,0)$\\[2pt]
  \textcolor{red}{$\tilde{F}(1,2,0)+\tilde{F}(0,2,1)$}\ \textcolor{green}{$W_t(0,3,0)$}\\
  \textcolor{red}{$+\tilde{F}(1,1,0)+\tilde{F}(0,1,1)$}
  \textcolor{green}{$W_t(0,2,0)$} };
\node[box] (R1C2) at (\xthr,\rowA) {$\displaystyle K_t(1,2,0)$\\[2pt]
  \textcolor{red}{$\tilde{F}(2,2,0)+\tilde{F}(1,3,0)+\tilde{F}(2,2,1)$}\, \textcolor{green}{$W_t(1,3,0)$}\\
  % \quad
  \textcolor{red}{$+\tilde{F}(2,1,0)+\tilde{F}(0,2,0)+\tilde{F}(1,1,1)$} \textcolor{green}{$W_t(1,2,0)$} };

\node[box] (R1C3) at (\xfo,\rowA) {$\displaystyle K_t(2,2,0)$\\[2pt]
  \textcolor{red}{$\tilde{F}(3,2,0)+\tilde{F}(2,3,0)+\tilde{F}(3,2,1)$} \textcolor{green}{$W_t(2,3,0)$} \\
  \textcolor{red}{$+\tilde{F}(3,1,0)+\tilde{F}(1,2,0)+\tilde{F}(2,1,1)$}\textcolor{green}{$W_t(2,2,0)$} };

% -----------------------
% Row 2 : 3 boxes
% -----------------------
\node[box] (R2C1) at (\xtwo,\rowB) {$\displaystyle K_t(0,1,0)$\\[2pt]
  \textcolor{red}{$\tilde{F}(1,1,0)+\tilde{F}(0,1,1)$}\ \textcolor{green}{$W_t(0,2,0)$}\\
    \textcolor{red}{$+\tilde{F}(0,0,1)$}\ \textcolor{green}{$W_t(0,0,0)$}\\
  \textcolor{red}{$+\tilde{F}(1,0,0)$}\  \textcolor{green}{$W_t(2,0,0)$}};
\node[box] (R2C2) at (\xthr,\rowB) {$\displaystyle K_t(1,1,0)$\\[2pt]
  \textcolor{red}{$\tilde{F}(2,1,0)+\tilde{F}(0,2,0)+\tilde{F}(1,1,1)$}\textcolor{green}{$W_t(1,2,0)$} \\
  % \quad
  \textcolor{red}{$+\tilde{F}(0,1,0)+\tilde{F}(1,0,1)$}\ \textcolor{green}{$W_t(1,1,0)$}\\
 \textcolor{red}{$+\tilde{F}(2,0,0)$}\  \textcolor{green}{$W_t(3,0,0)$} };
\node[box] (R2C3) at (\xfo,\rowB) {$\displaystyle K_t(2,1,0)$\\[2pt]
  \textcolor{red}{$\tilde{F}(3,1,0)+\tilde{F}(1,2,0)+\tilde{F}(2,1,1)$\ \textcolor{green}{$W_t(2,2,0)$} }\\
  % \quad
  \textcolor{red}{$+\tilde{F}(1,1,0)+\tilde{F}(2,0,1)$}\ \textcolor{green}{$W_t(2,1,0)$} \\
  \textcolor{red}{$+\tilde{F}(3,0,0)$}\textcolor{green}\ \textcolor{green}{$W_t(1,0,0)$} };

% -----------------------
% Row 3 : 6 boxes
% -----------------------
\node[smallbox] (R3C1) at (\xone,\rowC) {$K_t(0,1,0)$\\[2pt] 
\textcolor{red}{$\tilde{F}(1,1,0)+\tilde{F}(0,1,1)$}\ \textcolor{green}{$W_t(0,2,0)$}\\
    \textcolor{red}{$+\tilde{F}(0,0,1)$}\ \textcolor{green}{$W_t(0,0,0)$}\\
  \textcolor{red}{$+\tilde{F}(1,0,0)$}\  \textcolor{green}{$W_t(2,0,0)$}};
\node[smallbox] (R3C2) at (\xtwo,\rowC)  {$K_t(0,0,0)$\\[2pt] 
\textcolor{red}{$\tilde{F}(0,0,1)$}\  \textcolor{green}{$W_t(0,1,0)$} \\
\textcolor{red}{$\tilde{F}(1,0,0)$}\ \textcolor{green}{$W_t(2,0,0)$}};
\node[smallbox] (R3C3) at (\xthr,\rowC) {$K_t(1,0,0)$\\[2pt] 
\textcolor{red}{$\tilde{F}(0,1,0)+\tilde{F}(1,0,1)$}\ \textcolor{green}{$W_t(1,1,0)$}\\
\textcolor{red}{$+\tilde{F}(0,0,0)$}\ \textcolor{green}{$W_t(1,0,0)$}\\
\textcolor{red}{$+\tilde{F}(2,0,0)$}\  \textcolor{green}{$W_t(3,0,0)$}};
\node[smallbox] (R3C4) at (\xfo,\rowC)  {$K_t(2,0,0)$\\[2pt] 
\textcolor{red}{$\tilde{F}(1,1,0)+\tilde{F}(2,0,1)$}\ \textcolor{green}{$W_t(2,1,0)$}\\
\textcolor{red}{$+\tilde{F}(1,0,0)$}\ \textcolor{green}{$W_t(2,0,0)$}\\
\textcolor{red}{$+\tilde{F}(3,0,0)$}\  \textcolor{green}{$W_t(4,0,0)$}};
\node[smallbox] (R3C5) at (\xfi,\rowC)  {$K_t(3,0,0)$\\[2pt] 
\textcolor{red}{$\tilde{F}(2,1,0)+\tilde{F}(3,0,1)$}\ \textcolor{green}{$W_t(3,1,0)$}\\
\textcolor{red}{$+\tilde{F}(2,0,0)$}\ \textcolor{green}{$W_t(3,0,0)$}\\
\textcolor{red}{$+\tilde{F}(4,0,0)$}\  \textcolor{green}{$W_t(5,0,0)$}};
\node[] (R3C6) at (\xsi,\rowH)  {}; 
%\node[smallbox] (R3C6) at (\xsi,\rowC){$K_t(4,0,0)$\\[2pt] 
%\textcolor{red}{$L(3,1,0)+L(4,0,1)$}\ \textcolor{green}{$W_t(4,1,0)$}\\
%\textcolor{red}{$+L(3,0,0)$}\ \textcolor{green}{$W_t(4,0,0)$}\\
%\textcolor{red}{$+L(5,0,0)$}\  \textcolor{green}{$W_t(6,0,0)$}};
% -----------------------
% Row 4 : 5 boxes
% -----------------------
\node[smallbox] (R4C1) at (\xone,\rowD) {$K_t(0,1,1)$\\[2pt] 
\textcolor{red}{$\tilde{F}(0,1,2)+\tilde{F}(0,2,0)+\tilde{F}(1,1,1)$\, \textcolor{green}{$W_t(0,1,1)$}}\\
\textcolor{red}{$+\tilde{F}(0,1,0)+\tilde{F}(1,0,1)$}\ \textcolor{green}{$W_t(1,1,0)$}\\ 
\textcolor{red}{$+\tilde{F}(0,0,2)$}\ \textcolor{green}{$W_t(0,0,1)$}};
\node[smallbox] (R4C2) at (\xtwo,\rowD)  {$K_t(0,0,1)$\\[2pt] 
\textcolor{red}{$\tilde{F}(0,1,0)+\tilde{F}(1,0,1)$}\, \textcolor{green}{$W_t(1,1,0)$}\\
\textcolor{red}{$\tilde{F}(0,0,2)$}\ \textcolor{green}{$W_t(0,0,1)$} \\
\textcolor{red}{$\tilde{F}(0,0,0)$}\ \textcolor{green}{$W_t(1,0,0)$}};
\node[smallbox] (R4C3) at (\xthr,\rowD) {$K_t(1,0,1)$\\[2pt] 
\textcolor{red}{$\tilde{F}(0,1,1)+\tilde{F}(1,0,2)$} \textcolor{green}{$W_t(1,0,1)$}\\
\textcolor{red}{$+\tilde{F}(1,1,0)+\tilde{F}(2,0,1)$}\ \textcolor{green}{$W_t(2,1,0)$}\\ 
\textcolor{red}{$+\tilde{F}(0,0,1)$}\ \textcolor{green}{$W_t(0,0,0)$}\\
\textcolor{red}{$+\tilde{F}(1,0,0)$}\ \textcolor{green}{$W_t(2,0,0)$}};
\node[smallbox] (R4C4) at (\xfo,\rowD)  {$K_t(2,0,1)$\\[2pt] 
\textcolor{red}{$\tilde{F}(1,1,1)+\tilde{F}(2,0,2)+\tilde{F}(1,0,1)$} \textcolor{green}{$W_t(2,0,1)$}\\
\textcolor{red}{$+\tilde{F}(2,1,0)+\tilde{F}(3,0,1)$}\ \textcolor{green}{$W_t(3,1,0)$}\\ 
\textcolor{red}{$+\tilde{F}(2,0,0)$}\ \textcolor{green}{$W_t(3,0,0)$}};
\node[smallbox] (R4C5) at (\xfi,\rowD)  {$K_t(3,0,1)$\\[2pt] 
\textcolor{red}{$\tilde{F}(2,1,1)+\tilde{F}(3,0,2)$} \textcolor{green}{$W_t(3,0,1)$}\\
\textcolor{red}{$+\tilde{F}(2,0,1)+\tilde{F}(3,1,0)$}\\
\textcolor{red}{$+\tilde{F}(4,0,1)+\tilde{F}(3,0,0)$}};
%\textcolor{green}{$W_t(2,1,0)$\,
%\textcolor{green}{$W_t(2,0,0)$}}};
%\node[smallbox] (R4C6) at (\xsi,\rowD)  {$K_t(3,0,2)$\\[2pt] \textcolor{red}{$L(3,0,2)$} \ \textcolor{green}{$V(3,0,2)$}};

% -----------------------
% Row 5 : 4 boxes
% -----------------------
\node[smallbox] (R5C1) at (\xone,\rowE) {$K_t(0,1,2)$\\[2pt] 
\textcolor{red}{$\tilde{F}(0,1,3)+\tilde{F}(0,2,1)+\tilde{F}(1,1,2)$\, \textcolor{green}{$W_t(0,1,2)$}}\\
\textcolor{red}{$+\tilde{F}(0,1,1)+\tilde{F}(1,0,2)$}\ \textcolor{green}{$W_t(1,0,1)$}\\ 
\textcolor{red}{$+\tilde{F}(0,0,3)$}\ \textcolor{green}{$W_t(0,0,2)$}};
\node[smallbox] (R5C2) at (\xtwo,\rowE) {$K_t(0,0,2)$\\[2pt] 
\textcolor{red}{$\tilde{F}(0,1,1)+\tilde{F}(1,0,2)$}\ \textcolor{green}{$W_t(1,0,1)$}\\
\textcolor{red}{$+\tilde{F}(0,0,1)$}\ \textcolor{green}{$W_t(0,0,0)$}\\
\textcolor{red}{$+\tilde{F}(0,0,3)$} \textcolor{green}{$W_t(0,0,2)$}};
\node[smallbox] (R5C3) at (\xthr,\rowE) {$K_t(1,0,2)$\\[2pt]
 \textcolor{red}{$\tilde{F}(0,1,2)+\tilde{F}(1,0,3)$} \ \textcolor{green}{$W_t(1,0,2)$}\\
 \textcolor{red}{$+\tilde{F}(1,1,1)+\tilde{F}(2,0,2)+\tilde{F}(1,0,1)$} \ \textcolor{green}{$W_t(2,0,1)$}\\
  \textcolor{red}{$\tilde{F}(0,0,2)$} \ \textcolor{green}{$W_t(0,0,1)$}};
\node[smallbox] (R5C4) at (\xfo,\rowE)  {$K_t(2,0,2)$\\[2pt] 
\textcolor{red}{$\tilde{F}(1,1,2)+\tilde{F}(2,0,3)+\tilde{F}(1,0,2)$} \ \textcolor{green}{$W_t(2,0,2)$}\\
\textcolor{red}{+$\tilde{F}(2,1,1)+\tilde{F}(3,0,2)+\tilde{F}(3,0,0)$} \ \textcolor{green}{$W_t(3,0,1)$}};
%\node[smallbox] (R5C4) at (\xfi,\rowE)  {$K_t(3,0,3)$\\[2pt] \textcolor{red}{$L(3,0,3)$} \ \textcolor{green}{$V(3,0,3)$}};
%\node[smallbox] (R5C5) at (\xsi,\rowE)  {$K_t(4,0,3)$\\[2pt] \textcolor{red}{$L(4,0,3)$} \ \textcolor{green}{$V(4,0,3)$}};

% -----------------------
% Row 6 : 3 boxes
% -----------------------
\node[smallbox] (R6C1) at (\xone,\rowF)  {$K_t(0,1,3)$\\[2pt] 
\textcolor{red}{$\tilde{F}(0,1,4)+\tilde{F}(0,2,2)+\tilde{F}(1,1,3)$\, \textcolor{green}{$W_t(0,1,3)$}}\\
\textcolor{red}{$+\tilde{F}(0,1,2)+\tilde{F}(1,0,3)$}\ \textcolor{green}{$W_t(1,0,2)$}\\ 
\textcolor{red}{$+\tilde{F}(0,0,4)$}\ \textcolor{green}{$W_t(0,0,3)$}};
\node[smallbox] (R6C2) at (\xtwo,\rowF)  {$K_t(0,0,3)$\\[2pt] 
\textcolor{red}{$\tilde{F}(0,1,2)+\tilde{F}(1,0,3)$}\ \textcolor{green}{$W_t(1,0,2)$}\\
\textcolor{red}{$+\tilde{F}(0,0,2)$}\ \textcolor{green}{$W_t(0,0,1)$}\\
\textcolor{red}{$+\tilde{F}(0,0,4)$} \textcolor{green}{$W_t(0,0,3)$}};
\node[smallbox] (R6C3) at (\xthr,\rowF)  {$K_t(1,0,3)$\\[2pt]
 \textcolor{red}{$\tilde{F}(0,1,3)+\tilde{F}(1,0,4)+\tilde{F}(0,0,3)$} \ \textcolor{green}{$W_t(1,0,3)$}\\
 \textcolor{red}{$+\tilde{F}(1,1,2)+\tilde{F}(2,0,3)+\tilde{F}(1,0,2)$} \ \textcolor{green}{$W_t(2,0,2)$}};
%\node[smallbox] (R6C4) at (\xfi,\rowF)  {$K_t(3,0,4)$\\[2pt] \textcolor{red}{$L(3,0,4)$} \ \textcolor{green}{$V(3,0,4)$}};

% -----------------------
% Row 7 : 0 boxes
% -----------------------

\node[] (R7C2) at (\xsev,\rowI)  {}; 

% ============================
% Arrows and labels (many, matching the style of the original)
% ============================

% Row1 -> Row2 vertical green arrows with labels
\draw[arrowg] (R1C1) -- node[midway,left, lab] {$\theta_{02}$} (R2C1);
\draw[arrowg] (R1C2) -- node[midway,left, lab] {$\theta_{12}$} (R2C2);
\draw[arrowg] (R1C3) -- node[midway,left, lab] {$\theta_{22}$} (R2C3);
%\draw[arrowg] (R1C4) -- node[midway,left, lab] {$\theta_{32}$} (R2C4);

% Row2 -> Row3 downward green arrows (some converge into central nodes)
\draw[arrowg] (R2C1) -- node[midway,left, lab] {$\theta_{01}$} (R3C2);
\draw[arrowg] (R2C2) -- node[midway,left, lab] {$\theta_{11}$} (R3C3);
\draw[arrowg] (R2C3) -- node[midway,left, lab] {$\theta_{21}$} (R3C4);

% Row5 -> Row6 vertical arrows
%\draw[arrowg] (R5C2) -- node[midway,left,lab] {$\beta_{0,3}$} (R6C1);
%\draw[arrowg] (R5C3) -- node[midway,left,lab] {$\beta_{1,3}$} (R6C2);
%\draw[arrowg] (R5C4) -- node[midway,left,lab] {$\beta_{2,3}$} (R6C3);
%\draw[arrowg] (R5C4) -- node[midway,left,lab] {$\beta_{3,3}$} (R6C4);

% Blue internal arrows within rows (decay/internal)
\draw[arrowb, bend left=15] (R3C4) to node[above,lab] {$\varphi_{2}$} (R3C2);
\draw[arrowb, bend left=15] (R3C5) to node[above,lab] {$\varphi_{3}$} (R3C3);
%\draw[arrowb, bend left=12] (R3C4) to node[above,lab] {$\gamma_{1}$} (R3C5);
%\draw[arrowb, bend left=12] (R4C3) to node[above,lab] {$\gamma_{0}'$} (R4C4);
%\draw[arrowb, bend right=12] (R4C4) to node[above,lab] {$\gamma_{1}'$} (R4C5);
\draw[arrowb, bend left=65] (R3C2) to node[right,lab] {$\psi_{0}$} (R5C2);
\draw[arrowb, bend left=63] (R4C2) to node[right,lab] {$\psi_{1}$} (R6C2);

\draw[arrowg] (R3C2) -- node[above,lab] {$\xi_{0,0}$} (R3C1);
\draw[arrowg] (R4C2) -- node[above,lab] {$\xi_{0,1}$} (R4C1);
\draw[arrowg] (R5C2) -- node[above,lab] {$\xi_{0,2}$} (R5C1);
\draw[arrowg] (R6C2) -- node[above,lab] {$\xi_{0,3}$} (R6C1);
%\draw[arrowg, bend left=18] (R2C3) to node[midway,above,lab] {$\mu_{2\to2}$} (R4C5);

% Long-range green arrows from mid layers to far right (mirrors the original long curves)
%\draw[arrowg, bend left=28] (R3C5) to node[midway,above,lab] {$\lambda_{2\to3}$} (R5C5);
%\draw[arrowg, bend left=30] (R3C6) to node[midway,above,lab] {$\lambda_{3\to4}$} (R5C4);

% Blue connections from left-most column into central stack
\draw[arrowg] (R3C3) -- node[above,lab] {$\eta_{1,0}$} (R4C2);
\draw[arrowg] (R3C4) -- node[above,lab] {$\eta_{2,0}$} (R4C3);
\draw[arrowg] (R3C5) -- node[above,lab] {$\eta_{3,0}$} (R4C4);
\draw[arrowg] (R3C6) -- node[above,lab] {$\eta_{4,0}$} (R4C5);
\draw[arrowg] (R4C3) -- node[above,lab] {$\eta_{1,1}$} (R5C2);
\draw[arrowg] (R4C4) -- node[above,lab] {$\eta_{2,1}$} (R5C3);
\draw[arrowg] (R4C5) -- node[above,lab] {$\eta_{3,1}$} (R5C4);
\draw[arrowg] (R5C3) -- node[above,lab] {$\eta_{1,2}$} (R6C2);
\draw[arrowg] (R5C4) -- node[above,lab] {$\eta_{2,2}$} (R6C3);
\draw[arrowg] (R6C3) -- node[above,lab] {$\eta_{3,2}$} (R7C2);

% ---------------------------
% Legend (rotated together with the picture)
% ---------------------------
%\node[anchor=west, font=\small] at (-22, 11.6) {
 % \textcolor{red}{red: $L(\cdot)$} \quad
  %\textcolor{green}{green: $V(\cdot)$ and flow arrows} \quad
  %\textcolor{blue!65!black}{blue arrows: internal/decay}
%};

% ---------------------------
% End tikzpicture
% ---------------------------

\end{tikzpicture}
%\end{center}
\caption{Morphisms between Kac modules}
\label{DISEGNO}
\end{figure}
\end{landscape}

\section{Singular vectors of $W(4)$ and $S(4)$}
In this section we recall the classification of singular vectors in Verma modules for the Lie algebras $W(4)$ and $S(4)$ (the Lie subalgebra of $W(4)$ consisting of 0-divergence vector fields).
We recall that $F(a,b,c)$ denotes the irreducible $\slq$-module of highest weight $(a,b,c)$ and $F_t(a,b,c)$  the irreducible $\mathfrak {gl}_4$-module with highest weight $(a,b,c)$ and central charge $t$ (i.e.\ $\sum x_i \de_i$ acts by multiplication by $t$).
We consider the principal grading of $W(4)$ and the induced grading on $S(4)$, and define the following Verma modules:
\[\Ind_{W_{\geq 0}(4)}^{W(4)}(F_t(a,b,c))=\mathcal{U}(W(4))\otimes_{\mathcal{U}(W_{\geq 0}(4))}\otimes F_t(a,b,c)\]
where $W_{>0}(4)$ acts trivially on $F_t(a,b,c)$.

Similarly we define the Verma modules
\[\Ind_{S_{\geq 0}(4)}^{S(4)}(F(a,b,c))=\mathcal{U}(S(4))\otimes_{\mathcal{U}(S_{\geq 0}(4))}\otimes F(a,b,c)\]
where $S_{>0}(4)$ acts trivially on $F(a,b,c)$.
In \cite{R1} and \cite{R2} Rudakov  obtained a complete classification of singular vectors in Verma modules for $W(4)$ and $S(4)$.

Recall that a vector in a Verma module over $W(4)$ is called singular if it is an eigenvector of the subalgebra $\mathcal{B} +W_{>0}(4)$, where
 $\mathcal{B}$ is the Borel subalgebra of $W_0(4)$, and similarly for
 $S(4)$.
  
\begin{theorem}\label{singW4}
	The following is a complete classification of singular vectors (up  to scalar) in Verma modules for $W(4)$:
	\begin{itemize}
		\item $\de_4\otimes 1\in \Ind_{W_{\geq 0}(4)}^{W(4)}(F_0(0,0,0))$;
		\item $\de_3 \otimes x_4^*-\de_4 \otimes x_3^*\in \Ind_{W_{\geq 0}(4)}^{W(4)}(F_{-1}(0,0,1))$;
		\item $\de_2 \otimes x_1\wedge x_2+\de_3 \otimes x_1\wedge x_3+\de_4 \otimes x_1\wedge x_4\in \Ind_{W_{\geq 0}(4)}^{W(4)}(F_{-2}(0,1,0))$;
		\item $\de_1 \otimes x_1+\de_2 \otimes x_2+\de_3 \otimes x_3+\de_4\otimes x_4 \in \Ind_{W_{\geq 0}(4)}^{W(4)}(F_{-3}(1,0,0))$.
	\end{itemize}
	In particular there are no singular vectors of degree $>1$.
	We have therefore the following complex of morphisms (cfr. \eqref{Omegadual}):
	%\[
	%\Ind_{W(4)_{>0}}^{W(4)}(L_{-4}(0,0,0))\rightarrow \Ind_{W(4)_{>0}}^{W(4)}(L_{-3}(1,0,0))\rightarrow
	%\Ind_{W(4)_{>0}}^{W(4)}(L_{-2}(0,1,0))\rightarrow \Ind_{W(4)_{>0}}^{W(4)}(L_{-1}(0,0,1))\rightarrow
%	\Ind_{W(4)_{>0}}^{W(4)}(L_0(0,0,0))
%	\]
\end{theorem}

\begin{figure}[h]
	\begin{center}
		$$
	\begin{tikzpicture}
			\draw[->, line width=2pt](-0.2,0)--(2.2,0);
			\draw[->, line width=2pt](2.8,0)--(5.2,0);
			\draw[->, line width=2pt](5.8,0)--(8.2,0);
			\draw[->, line width=2pt](8.8,0)--(11.2,0);
			
			\draw[fill=black]{(-0.5,0) circle(3pt)};
			\draw[fill=black]{(2.5,0) circle(3pt)};
			\draw[fill=black]{(5.5,0) circle(3pt)};
			\draw[fill=black]{(8.5,0) circle(3pt)};
		    \draw[fill=black]{(11.5,0) circle(3pt)};
			
			\node at (-0.5,0.5) {$\Ind F_{-4}(0,0,0)$};
			\node at (2.5,0.5) {$\Ind F_{-3}(1,0,0)$};
			\node at (5.5,0.5) {$\Ind F_{-2}(0,1,0)$};
			\node at (8.5,0.5) {$\Ind F_{-1}(0,0,1)$};
			\node at (11.5,0.5) {$\Ind F_0(0,0,0)$};
			
		\end{tikzpicture}$$
	\end{center}
	\caption{ All non-zero morphisms between Verma modules for $W(4)$. } 
\end{figure}

\begin{theorem}\label{S4sing}
	The following is a complete classification of singular vectors for $S(4)$
	\begin{itemize}
		\item $\de_4\otimes 1\in \Ind_{S_{\geq 0}(4)}^{S(4)}(F(0,0,0))$;
		\item $\de_3 \otimes x_4^*-\de_4 \otimes x_3^*\in \Ind_{S_{\geq 0}(4)}^{S(4)}(F(0,0,1))$;
		\item $\de_2 \otimes x_1\wedge x_2+\de_3 \otimes x_1\wedge x_3+\de_4 \otimes x_1\wedge x_4\in \Ind_{S_{\geq 0}(4)}^{S(4)}(F(0,1,0))$;
		\item $\de_1 \otimes x_1+\de_2 \otimes x_2+\de_3 \otimes x_3+\de_4 \otimes x_4\in \Ind_{S_{\geq 0}(4)}^{S(4)}(F(1,0,0))$;
		\item $\de_4(\de_1 \otimes x_1+\de_2 \otimes x_2+\de_3 \otimes x_3+\de_4\otimes x_4) \in \Ind_{S_{\geq 0}(4)}^{S(4)}(F(1,0,0))$.
	\end{itemize}
	We have therefore the following diagram of morphisms between Verma modules for $S(4)$:
\end{theorem}
\begin{figure}[h]
	\begin{center}
		$$
		\begin{tikzpicture}
			\draw[->, line width=2pt](1.8,0)--(-1.8,0);
			\draw[->, line width=2pt](2,3.8)--(2,0.2);
			\draw[->, line width=2pt](-2,0.4)--(-2,3.8);
			\draw[->, line width=2pt](-1.8,4)--(1.8,4);
			\draw[->, line width=2pt](1.8,3.8)--(-1.8,0.2);
			
			\draw[fill=black]{(2,0) circle(3pt)};
			\draw[fill=black]{(-2,0) circle(3pt)};
			\draw[fill=black]{(2,4) circle(3pt)};
			\draw[fill=black]{(-2,4) circle(3pt)};
			
			\node at (2.3,-0.5) {$\Ind F(0,0,0)$};
			\node at (-2.3,-0.5) {$\Ind F(1,0,0)$};
			\node at (3.3,4) {$\Ind F(0,0,1)$};
			\node at (-3.5,4) {$\Ind F(0,1,0)$};
			
		\end{tikzpicture}$$
	\end{center}
	\caption{\label{S4morphisms} All non-zero morphisms between Verma modules for $S(4)$. External morphisms have degree 1, and the ``diagonal'' morphism  has degree 2.} 
\end{figure}
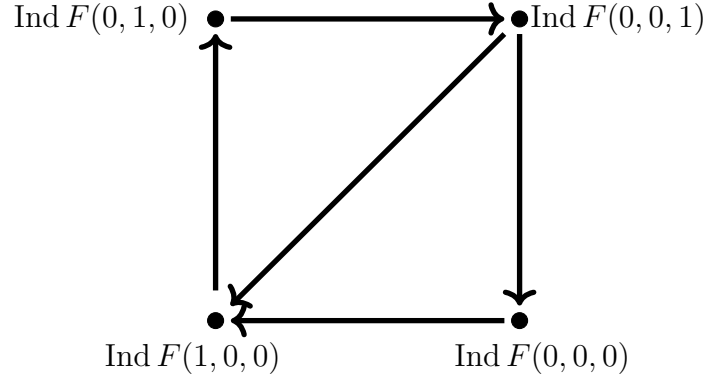

\section{$E(4,4)$-Verma modules and their singular vectors}\label{sec:sing}
Let $W$ be a finite-dimensional $\phat$-module and let us consider the $E(4,4)$-Verma module induced by $W$:
\[
M(W)=\Ind_{E(4,4)_{\geq 0}}^{E(4,4)}W
\]
where $E(4,4)_{>0}$ acts trivially on $W$.
In \cite{CaCaKac} the following result was proved:
\begin{theorem}\label{duality}
There exists a contravariant duality functor on the category of continuous $E(4,4)$-modules with discrete topology that associates $M(W^*)$ to $M(W)$.
\end{theorem}
\begin{remark} We conjectured this duality in a setup of an arbitrary Lie superalgebra and its subalgebra of finite codimension in \cite{CCK}. Michel Duflo pointed out that this conjecture has been proved long ago by Sophie Chemla in \cite{Che}.
\end{remark}

This result, together with the description of the modules $W_t(a,b,c)^*$ given in Corollary \ref{duali} will be fundamental in the classification of degenerate (i.e., reducible) $E(4,4)$-Verma modules. We also recall that $K_t(a,b,c)^*\cong K_{-t}(c,b,a)$ (see \cite[Lemma 4.4]{S}).

Now let us denote by $M_t(a,b,c)$  the generalized Verma module induced by the $E(4,4)_0=\phat$-module $W_t(a,b,c)$, extended trivially to $E(4,4)_{>0}$, i.e.,
\[
M_t(a,b,c)=\Ind_{E(4,4)_{\geq 0}}^{E(4,4)}W_t(a,b,c),
\]
and similarly define $\tilde M_t(a,b,c)$ as the Verma module induced by the Kac module $K_t(a,b,c)$.
\begin{definition}
Let $0\neq w\in M_t(a,b,c)$. We say that $w$ is a  $E(4,4)$-singular vector if
\begin{itemize}
	\item $(x_i\de_{i+1})w=0$ for all $i=1,2,3$;
	\item $bw=0$ for all $b\in \phat_1$;
	\item $Xw=0$ for all $X\in E(4,4)_{>0}$.
\end{itemize}
\end{definition}
We always assume that a singular vector is homogeneous and a weight vector for $\slq$.
\begin{lemma}Let $0\neq w\in M_t(a,b,c)$. Then $w$ is a $E(4,4)$-singular vector if and only if
	\begin{itemize}
		\item $(x_i\de_{i+1})w=0$ for all $i=1,2,3$;
		\item $b_{44}w=0$;
		\item $E\,w=0$, where $E=x_4^2d_3-x_3x_4 d_4$;
		\item $(x_4C) w=0$, where $C=\sum_i x_i \de_i$ is the central element in $\phat$.
	\end{itemize}
\end{lemma}
\begin{proof}
This is a consequence of the following facts: 
\begin{itemize}
	\item $b_{44}$ is a lowest weight vector in $\phat_1$, which is an irreducible $\slq$-module;
	\item elements $E$ and $x_4C$ generate $E(4,4)_1$ over the Borel subalgebra of $\phat$ generated by $x_i\de_j$ (with $i<j$) and $\phat_1$;
	\item $E(4,4)_{>0}$ is generated by $E(4,4)_1$.  
\end{itemize}
\end{proof}
The main result of this paper is a classification of $E(4,4)$-singular vectors in Verma modules $M_t(a,b,c)$.
We represent a generic vector $w\in M_t(a,b,c)$ in the following form
\[
w=\sum _{I,J} \de_I d_J \otimes v_{I,J},\,\, v_{I,J}\in W_t(a,b,c),
\]
where $I$ is a multiset with entries in $\{1,2,3,4\}$ and $J$ is an ordered subset of $\{1,2,3,4\}$. We sometimes drop the symbol $\otimes$ in $w$. In this notation we have, e.g. $\de_{1,2,2}d_{1,3,2}=-\de_1\de_2^2d_1 d_2 d_3$.
Our first target is to show a bound on the degree of singular vectors. Recall that if $v\in L_{-1}^d\otimes W$, we say that $v$ has degree $d$.

\begin{proposition}\label{bound}
	Let $w$ be a singular vector in $M_t(a,b,c)$. Then $w$ has degree at most 5.
\end{proposition}
\begin{proof}
	Let $V=F_t(a,b,c)$ and consider the following subspaces of $M_t(a,b,c)$
	\[
	M(k)=\C[\de_1,\ldots,\de_4]\otimes \sum_{j\leq k} \sum_{i_1<\cdots<i_j}\C d_{i_1}\cdots d_{i_j}\otimes V,
	\]
	where $k=0,\ldots,4$.
	The subspace $M(k)$ is not a $E(4,4)$-submodule but we can observe that it is indeed a $W(4)$-submodule of $M_t(a,b,c)$. So we have the following filtration of $M$ into $W(4)$-submodules
	
	\[
	M(0)=\C[\de_1,\ldots,\de_4]\otimes V \subseteq M(1)\subseteq \cdots \subseteq M(4)=M.
	\]
	We also let $N(k)=M(k)/M(k-1)$ for $k=0,\ldots,4$, where $M(-1)=0$. Let 
	\[
	F(k)=\sum_{i_1<\cdots<i_k}\C d_{i_1}\cdots d_{i_k}\otimes V
	\]
	and observe that as a vector space we have
	\[
	N(k)\cong \C[\de_1,\ldots,\de_4]\otimes  F(k).
	\]
	Also observe that $F(k)$ is annihilated by $W_{>0}(4)$ in $N(k)$. It follows that 
	\[N(k)\cong \Ind_{W_{\geq 0}(4)}^{W(4)}F(k).\]
Now consider a singular vector $w\in M_t(a,b,c)$ and let $k$ be minimum such that $w\in M(k)$. Then the projection of $w$ in the quotient $N(k)$ is necessarily a singular vector for $W(4)$. Since singular vectors of $W(4)$ have degree at most 1 by Theorem \ref{singW4}, the vector $w$ has degree at most 5 in $N(k)$, and the result follows.
\end{proof}
Let $N=N(\phat)=N(\slq)\oplus \phat_1$ be a maximal nilpotent subalgebra of $\phat$, and  $B=B(\phat)=B(\slq)\oplus \phat_1$ the Borel subalgebra containing $N$. Assume that 
\[
w=\sum _{I,J} \de_I d_J \otimes v_{I,J}
\]
is a singular vector of degree $d$. Let $U_d=\bigoplus_{|I|+|J|=d}\C \de_I d_J$ and denote by $(\de_I d_J)^*$ the dual basis of $U_d^*$.
Since $N$ acts trivially on $w$ we deduce that the map $U_d^*\rightarrow F_t(a,b,c)$ given by
\begin{equation} \label{vIJaction}
(\de_I d_J)^*\mapsto v_{I,J}
\end{equation}
is a morphism of $N$-modules. This observation significantly simplifies the action of elements in $B$ on vectors $v_{I,J}$. In particular we may observe that $b_{44}v_{I,J}\neq 0$ only if $4$ is an entry in $J$.
 
The classification of singular vectors of $S(4)$ allows us to make a further semplification of possible singular vectors of $E(4,4)$.
\begin{lemma}\label{barw}
	Let $w=\sum _{I,J} \de_I d_J \otimes v_{I,J}$ be a singular vector of degree $k$ in $M_t(a,b,c)$. Then
	\[
	\bar w=\sum_{|I|=k} \de_I \otimes v_{I,\emptyset}
	\]
	is a singular vector for $S(4)$. In particular if $k\geq 3$ we have $\bar w=0$, and if $k=1,2$ then $\bar w$ must be compatible with the classification in Theorem \ref{S4sing}.
\end{lemma}
\begin{proof}
Recall that $S_1(4)$ is an irreducible $\slq$-module with highest weight $(2,0,1)$. The vector field $x_4^2\de_1$ is a lowest weight vector and we have that $\bar w$ is a $S(4)$-singular vector if and only if it is annihilated by $x_i\de_{i+1}$ (for $i=1,2,3$) and by $x_4^2 \de_1$. It is clear that $(x_i\de_{i+1})(\bar w)$ and $(x_i \de_{i+1})(w-\bar w)$, if not zero, are linearly independent and therefore $(x_i \de_{i+1})(\bar w)=0$.
We claim that also $(x_4^2 \de_1)(\bar w)$ and $(x_4^2 \de_1)(w-\bar w)$, if not zero, are linearly independent. Indeed let 
\[M(i,j)=\bigoplus_{|I|=i,|J|=j} \C \de_I d_J \otimes F_t(a,b,c).\] It is enough to show that $(x_4^2 \de_1)\bar w \in M(k-1,0)$ and $(x_4^2 \de_1)(w-\bar w)\in \bigoplus_{i=0}^{k-2}M(i,k-1-i)$. The first part is clear. As for the second part, we observe that the only terms that could contribute to $M(k-1,0)$ are those of the form $\de_I d_1\otimes v_{I,1}$ because $[x_4^2\de_1,d_i]=\delta_{i,1} b_{44}$. But in this case we have $b_{44}v_{I,1}=0$ and the claim follows.
\end{proof}
\section{Singular vectors of degree 1}
%We will need a result like the following (che avevo dato un po' per scontato)
%\begin{lemma}\label{scontato} Let $w\in W_t(a,b,c)$ be a $\phat$-singular vector.
%	\begin{itemize}
%	\item If $\lambda(w)=(r,0,0)$ for some $r>0$ then $W_t(a,b,c)=W_t(r,0,0)$;
%	\item if $\lambda(w)=(0,1,0)$ then $W_t(a,b,c)=W_t(0,1,0)$;
%	\item if $\lambda(w)=(0,r,r+1)$ then $W_t(a,b,c)=W_t(0,r,r+1)$;
%	\item if $\lambda(w)=(0,0,r)$ for some $r>0$ then $W_t(a,b,c)=W_t(0,0,r)$;
%	\end{itemize}
%\end{lemma}
%\begin{proof}
%If $W_t(a,b,c)$ has a singular vector of weight $(r,0,0)$ there is a nonzero morphism
%\[
%K_t(r,0,0)\rightarrow W_t(a,b,c)
%\]
%and in particular $W_t(a,b,c)$ is a quotient of $K_t(r,0,0)$. By the composition series of $K_t(r,0,0)$ we deduce that $W_t(a,b,c)\cong W_t(r,0,0), W_t(r+2,0,0), W_t(r,1,0)$.
%Now recall the decomposition of $K_t(r,0,0)$ as a $\mathfrak h$-module:
%\[
%K_t(r,0,0)=\tilde F(r-1,0,0)\oplus \tilde F(r+1,0,0) \oplus \tilde F(r-1,1,0)\oplus \tilde F(r,0,1).
%\]
%If $W_t(r,1,0)\cong  \tilde F(r-1,1,0)\oplus \tilde F(r,0,1)$ is a quotient of $K_t(r,0,0)$ then necessarily the $\mathfrak h$-submodule $\tilde F(r-1,0,0)\oplus \tilde F(r+1,0,0)$ would be a $\phat$-submodule too. But this is not the case since the latter contains both copies of $F_t(r,0,0)$ in $K_t(r,0,0)$ and in particular it contains the generator $z$.
%We know that $W_t(r+2,0,0)$ is a submodule of $K_t(r,0,0)$, being the image of $\varphi_{r+2,0}$. So it is not a quotient because we know that $K_t(r,0,0)$ is indecomposable.
%\end{proof}
Let $M_t(a,b,c)$ be the generalized $E(4,4)$-Verma module induced by $W_t(a,b,c)$. We look for possible singular vectors $w$ in $M_t(a,b,c)$ of degree 1. We simplify the notation introduced in the previous section by letting $v_i=v_{i,\emptyset}$ and $z_i=v_{\emptyset,i}$ so that an arbitrary singular vector $w$ is of the form:
\begin{equation}\label{(2a)}
w=\de_1 \otimes v_1+\de_2 \otimes v_2+\de_3\otimes v_3 +\de_4\otimes v_4+d_4\otimes z_4+d_3\otimes z_3+d_2\otimes z_2+d_1\otimes z_1,
\end{equation}
where $v_i,z_i\in W_t(a,b,c)$, $i=1,2,3,4$.
We consider the eight vectors $v_1,v_2,v_3,v_4,z_4,z_3,z_2$, $z_1$ in this order. We observe that if one of them is 0 then all the others to the left of it are also 0, and if any of them is nonzero then all the others to the right of it are nonzero. This is due to the fact that $(x_1\de_2)z_1=-z_2$, $(x_2 \de_3)z_2=-z_3$, $(x_3 \de_4)z_3=-z_4$, $b_{44}z_4=-2v_4$, $(x_3 \de_4)v_4=v_3$, $(x_2 \de_3) v_3=v_2$ and $(x_1\de_2)v_2 =v_1$. 
Indeed in this case the morphism given in \eqref{vIJaction} can be exploited as follows: for all $i<j$
\begin{equation}\label{ev}
	(x_i \de_j)v_k=\delta_{j,k}v_i,
\end{equation}
\begin{equation}\label{ez}
	(x_i \de_j)z_k=-\delta_{ik} z_j,
\end{equation}
and for all $r,s,k$ 
\begin{equation}\label{brsv} b_{ij}v_k=0 
\end{equation}
and
\begin{equation}\label{brsz}
	b_{rs}z_j=-\delta_{rj}v_s-\delta_{sj}v_r.
\end{equation}

Moreover we have
\begin{equation}
	\label{Ew} \frac{1}{3} Ew=\frac{1}{2}a_{34}v_4-x_4 \de_1 z_2+x_4 \de_2 z_1=0
\end{equation}
and 
\begin{equation}
	\label{x4Cw} (x_4C)w=-(x_4\de_1)v_1-(x_4 \de_2)v_2 -(x_4 \de_3) v_3 +\frac{1}{4}(h_{14}+h_{24}+h_{34}-5t+5)v_4+\frac{5}{4}(a_{14}z_1+a_{24}z_2+a_{34}z_3)=0.
\end{equation}

By relations \eqref{ev}--\eqref{brsz}, the first nonzero vector  from left to right among $v_1,v_2,v_3,v_4,z_4,z_3,z_2$, $z_1$ in $(\ref{(2a)})$ must be a $\phat$-singular vector.%, i.e. it has weight $(a,b,c)$.

Given a weight vector $v$ we  denote by $\lambda(v)$ its $\slq$-weight. We  let $\lambda(w)=(\alpha,\beta,\gamma)$. Hence we must have  for the singular vector $(\ref{(2a)})$:
\begin{itemize}
	\item $\lambda(v_1)=(\alpha+1,\beta,\gamma) $;
	\item $\lambda(v_2)=(\alpha-1,\beta+1,\gamma) $;
	\item $\lambda(v_3)=(\alpha,\beta-1,\gamma+1) $;
	\item $\lambda(v_4)=(\alpha,\beta,\gamma-1) $;
	\item $\lambda(z_4)=(\alpha,\beta,\gamma+1) $;
	\item $\lambda(z_3)=(\alpha,\beta+1,\gamma-1) $;
	\item $\lambda(z_2)=(\alpha+1,\beta-1,\gamma) $;
	\item $\lambda(z_1)=(\alpha-1,\beta,\gamma) $.
\end{itemize}

\begin{lemma}\label{lemmagrado1}Let $W$ be a finite dimensional $\phat$-module and $w\in \Ind_{E(4,4)_{\geq 0}}^{E(4,4)}(W)$ be a  $E(4,4)$-singular vector, such that $\lambda(w)=(\alpha,\beta,\gamma)$. %Assume
	%\[
	%w= \de_1 \otimes v_1+\de_2 \otimes v_2+\de_3\otimes v_3 +\de_4\otimes v_4+d_4\otimes z_4+d_3\otimes z_3+d_2\otimes z_2+d_1\otimes z_1.\]
Let $w$ be as in $(\ref{(2a)})$.	
	%Then
\begin{enumerate}
	\item [(a)] If $v_1\neq 0$, then $(\alpha,\beta,\gamma)=(0,0,0)$ and $v_1$ is a $\phat$-singular vector of weight $\lambda (v_1)=(1,0,0)$;
	\item [(b)] if $v_1=0$ and $v_2\neq 0$, then   $(\alpha,\beta,\gamma)=(1,0,0)$ and $v_2$ is a $\phat$-singular vector of weight $\lambda (v_2)=(0,1,0)$;
	\item [(c)] if $v_1=v_2=0$ and $v_3\neq 0$, then $(\alpha,\beta,\gamma)=(0,\gamma+1,\gamma)$ and $v_3$ is a $\phat$-singular vector of weight $\lambda (v_3)=(0,\gamma, \gamma+1)$;
	\item [(d)] if $v_1=v_2=v_3=0$ and $v_4\neq 0$, then $(\alpha,\beta,\gamma)=(0,0,\gamma)$ and $v_4$ is a $\phat$-singular vector of weight $\lambda (v_4)=(0,0, \gamma-1)$;
	\item [(e)] if $v_1=v_2=v_3=v_4=0$ and $z_4\neq 0$, then $(\alpha,\beta,\gamma)=(0,0,\gamma)$ and $z_4$ is a $\phat$-singular vector of weight $\lambda (z_4)=(0,0, \gamma+1)$;
	\item[(f)]if $v_1=v_2=v_3=v_4=z_4=0$, then $z_3=z_2=0$; moreover,  if $z_1\neq 0$, then $(\alpha,\beta,\gamma)=(\alpha,0,0)$ and $z_1$ is a $\phat$-singular vector of weight $\lambda (z_1)=(\alpha-1,0,0)$.
\end{enumerate}
\end{lemma}
\begin{proof}
	Applying $b_{12}$ to \eqref{Ew} we obtain
	\begin{equation}\label{b12(7)}
		\alpha v_4+(x_4\de_1)v_1-(x_4 \de_2)v_2=0,
	\end{equation}
and applying $x_1 \de_4$ to this equation we deduce $(2\alpha+\beta+\gamma)v_1=0$. Case (a) follows since $\alpha, \beta, \gamma\in \Z_{\geq 0}$ because $w$ is a highest weight vector in a finite-dimensional $\slq$-module.
We can therefore assume $v_1=0$. By applying $x_2 \de_4$ to $\eqref{b12(7)}$ we obtain $(\alpha-\beta-\gamma-1)v_2=0$. 
Next we apply $x_1 \de_4$ to \eqref{x4Cw} and obtain
\begin{equation}\label{x1de4(8)}
		a_{12}z_2+a_{13}z_3+a_{14}z_4=0,
\end{equation}
and applying $b_{23}$ to \eqref{x1de4(8)} we obtain 
\begin{equation}\label{b23x1de4(8)}a_{12}v_3+a_{13}v_2+2\beta z_4=0.
\end{equation}
If we apply $b_{24}$ to this equation we deduce $(2\beta+\gamma)v_2=0$. In particular, if $v_2\neq 0$,  we have that $\beta=\gamma=0$ and $\alpha=1$. Case (b) follows and we can assume that $v_1=v_2=0$.

Equation \eqref{b12(7)} provides $\alpha v_4=0$ and, applying $b_{34}$ to \eqref{b23x1de4(8)}, we obtain $(\gamma+1-\beta)v_3=0$. Case (c) follows and we can assume that $v_1=v_2=v_3=0$.

By \eqref{b23x1de4(8)} we get $2\beta z_4=0$ and case (d) follows.

We can assume that $v_i=0$ for all $i=1,\ldots,4$. Applying $x_1\de_4$ to \eqref{Ew} we obtain
\begin{equation}\label{x1de4(7)}(\alpha+\beta+\gamma+1)z_2+(x_4\de_2)z_4=0
\end{equation}
and applying $x_2\de_4$ to this equation we obtain $\alpha z_4=0$. By \eqref{b23x1de4(8)} we also obtain $2\beta z_4=0$ and case (e) follows.
Now we can assume that also $z_4=0$. 
By \eqref{x1de4(7)} we have $(\alpha+\beta+\gamma+1)z_2=0$ which forces $z_2=0$ since $\alpha,\beta,\gamma$ are nonnegative integers. Hence $z_3=-(x_2\de_3)z_2=0$. Finally we can apply $x_2 \de_4$ to \eqref{Ew} to obtain $(\beta+\gamma)z_1=0$ and the proof is complete.
\end{proof}
\begin{proposition}\label{w8}
	Assume we are in case (a) of Lemma \ref{lemmagrado1} with $W=W_t(a,b,c)$ for some $a,b,c\geq 0$. Then $a=t=1$, $b=c=0$, so that we are in the standard module, described in Example \ref{standard}, and 
	\[
	w=\de_1 \otimes d_1+\de_2 \otimes d_2+\de_3 \otimes d_3 +\de_4 \otimes d_4+ d_1\otimes \de_1+d_2\otimes \de_2+ d_3 \otimes \de_3+d_4\otimes \de_4
	\]
	is indeed an $E(4,4)$-singular vector. It is unique up to a scalar factor.
\end{proposition}
\begin{proof} By Lemma \ref{lemmagrado1}, $M_t(a,b,c)=M_t(1,0,0)$. 
Weight reasons and Equations \eqref{ev}, \eqref{ez}, \eqref{brsv} and \eqref{brsz} imply that the singular vector $w$ is necessarily the one displayed in the statement. Indeed, $z_1$ is a nonzero vector in the standard module $W_t(1,0,0)$ of weight $(-1,0,0)$ and is therefore $\de_1$, up to scalar. Equation \eqref{ez} provides $z_i=-(x_1 \de_i) z_1=\de_i$ and Equation \eqref{brsz} provides $-2v_i=b_{ii} z_i=[b_{ii},\de_i]=-2d_i$.

Let us check that $w$ is a $\phat$-singular vector if and only if $t=1$.
Conditions $(x_i\de_{i+1})w=0$ are easy to be verified for every $t\in \C$.
We also have 
\[
Ew=3(x_4 \de_2)\otimes \de_1-3(x_4 \de_1)\otimes \de_2+\frac{1}{2}b_{44}\otimes d_3+(-\frac{1}{2}b_{34}+\frac{3}{2}a_{34})\otimes d_4=0
\]
for every $t\in \C$.
Finally we compute 
\begin{align*}
(x_4 C) w&= (-\frac{1}{4}b_{14}+\frac{5}{4}a_{14}])\otimes \de_1+(-\frac{1}{4}b_{24}+\frac{5}{4}a_{24}])\otimes \de_2+(-\frac{1}{4}b_{34}+\frac{5}{4}a_{34}])\otimes \de_3-\frac{1}{4}b_{44}\otimes \de_4\\
&-x_4\de_1\otimes d_1-x_4 \de_2 \otimes d_2 -x_4 \de_3 \otimes d_3+\frac{1}{4}(x_1 \de_1+x_2\de_2+x_3\de_3-3x_4\de_4-5t)\otimes d_4\\
&= \frac{1}{4}d_4+\frac{5}{4}td_4+ \frac{1}{4}d_4+\frac{5}{4}td_4+ \frac{1}{4}d_4+\frac{5}{4}td_4+\frac{1}{2}d_4-d_4-d_4-d_4-\frac{3}{4}d_4-\frac{5}{4}td_4\\
&=(-\frac{5}{2}+\frac{5}{2}t)d_4.
\end{align*}
In particular, $(x_4C) w=0$ if and only if $t=1$ and the proof is complete.
\end{proof}
\begin{proposition}\label{typef}
Assume we are in case (f) in Lemma \ref{lemmagrado1} with $W=W_t(a,b,c)$. Then $b=c=0$, and either $a=t=0$ or $a\geq 1$. Moreover, up to a scalar factor, $w=d_1\otimes z_1$, where $z_1$ is the $\phat $-singular vector of weight $(a,0,0)$ in $W_t(a,0,0)$, is an $E(4,4)$-singular vector.
\end{proposition}
\begin{proof} By Lemma \ref{lemmagrado1}  we have $b=c=0$ and $w=d_1\otimes z_1$, where $z_1=x_1^a\in F_t(a,0,0)$. We show that the vector $d_1\otimes z_1$ is a $E(4,4)$-singular vector if and only if $a=t=0$ or  $a\geq 1$. The vector $w$ is clearly annihilated by $x_i \de_{i+1}$. Also $Ew=0$ by \eqref{Ew}. By \eqref{x4Cw} we also have
\[
(x_4C)w=\frac{5}{4}a_{14}z_1.
\]
If $a=0$ and $t\neq 0$,  $a_{14}z_1$ is a non-zero vector of  $W_t(0,0,0)$ by Remark \ref{a12}. In all other cases $a_{14}z_1=0$ in $W_t(a,0,0)$. Indeed, if $a=t=0$ this is clear because $W_0(0,0,0)$ is the trivial module. If $a\geq 1$,  by Proposition \ref{quozienti}, $W_t(a,0,0)=K_t(a,0,0)/Im(\theta_{a,1})$ and in particular $a_{12}z_1=0$ in $W_t(a,0,0)$ since $a_{12}z_1$ lies in $Im(\theta_{a,1})$; it follows that  $a_{14}z_1=(x_4\de_2)a_{12}z_1=0$, hence $w$ is a singular vector for $E(4,4)$. 
\end{proof}
\begin{corollary}\label{morphtypea}
	The  $E(4,4)$-singular vector described in Proposition \ref{typef} gives rise to a non-zero morphism
	\[
	M_{t-1}(a+1,0,0)\rightarrow M_t(a,0,0)
	\]
if and only if $a=t=0$, or $a=1$ and $t\neq 1$, or $a\geq 2$.
\end{corollary}
\begin{proof}
	The existence of the singular vector $w=d_1\otimes z_1$ gives rise to a morphism
	\[
	K_{t-1}(a+1,0,0)\rightarrow M_t(a,0,0).
	\]
If $t\neq 1$ or $a\neq 1$ this map factors through the quotient $W_{t-1}(a+1,0,0)$ if and only if $a_{12} x_1^{a+1}$ is mapped to 0 (see Proposition \ref{quozienti}). But the image of $a_{12}x_1^{a+1}$ is $a_{12}d_1\otimes z_1=-d_1\otimes a_{12}z_1=0$.

If $t=1$ and $a=1$ the maximal submodule of $K_0(2,0,0)$ is not generated by $a_{12}x_1^2$. In this case we have a map $\phi:K_0(2,0,0)\rightarrow K_0(0,0,0)$ given by $x_1^2\mapsto a_{12}a_{13}a_{14}$. $Im(\phi)$ has a submodule $M$ of dimension 1 spanned by $a_{12}a_{13}a_{14}a_{23}a_{24}a_{34}\otimes 1$ whose quotient is isomorphic to $W_0(2,0,0)$. We deduce that the original map factors through $W_0(2,0,0)$ if and only if $a_{23}a_{24}a_{34}d_1\otimes d_1$ is 0 in $M_1(1,0,0)$. Indeed one can compute in $\mathcal U(E(4,4)_{<0})$
\[
a_{23}a_{24}a_{34}d_1=2d_3a_{34}-2\de_2a_{23}a_{34}-2d_2a_{34}-2\de_3a_{23}a_{34}-2\de_4a_{24}a_{34}-d_1a_{23}a_{24}a_{34},
\]
and conclude that $a_{23}a_{24}a_{34}d_1\otimes d_1\neq 0$ since, for example, $2d_3\otimes a_{34}d_1\neq 0$.
\end{proof}
\begin{proposition}
	Assume we are in case (b) in Lemma \ref{lemmagrado1} with $W=W_t(a,b,c)$. Then $t\neq 0$. Moreover, if we denote by 1 the $\phat$-singular vector in $M_t(0,1,0)\cong M_t(0,0,0)$ of weight $(0,0,0)$, we have
	\[w=4\sum_{i=2}^4 \de_i\otimes a_{1i}-2(d_4\otimes a_{12}a_{13}+d_3\otimes a_{14}a_{12}+d_2\otimes a_{13}a_{14})-d_1\otimes(a_{12}a_{34}+a_{24}a_{13}+a_{14}a_{23}-4)\]
	is an $E(4,4)$-singular vector. It is unique up to a scalar factor.
\end{proposition}
\begin{proof}
	By Corollary \ref{morphtypea}, for $a=1$ and $t\neq 1$ we have a morphism
	\[
	M_{t-1}(2,0,0)\rightarrow M_t(1,0,0).
	\]
	By Theorem \ref{duality} and Corollary \ref{duali} we have a morphism
	\[
	M_{-t}(1,0,0)\rightarrow M_{1-t}(0,1,0).
	\]
Therefore the singular vectors described in case (b) exist for all $t\neq 0$. 
 One can also describe explicitly this family of singular vectors. Indeed, consider the vector $w$ displayed in the statement. 
We recall that  $W_t(0,0,0)$ is defined as the quotient of  $K_t(0,0,0)$ modulo the maximal submodule $Im(\varphi_{2})$. In particular, for $t\neq 0$, we have $a_{12}a_{13}a_{14}=0$ in $W_t(0,0,0)$. We have
\begin{align*}
(x_1 \de_2) w&= -2d_1a_{13}a_{14}+d_1(a_{13}a_{14}-a_{14}a_{13})=0,
\\
(x_2 \de_3) w&= -4\de_3 a_{12}+4\de_3 a_{12}+2d_2a_{12}a_{14}-2d_2a_{12}a_{14}+d_1(-a_{12}a_{24}+a_{12}a_{24})=0,
\\
(x_3 \de_4) w&=-4\de_4 a_{13}+4\de_4 a_{13}-2d_3a_{12}a_{13}+2d_3a_{12}a_{13}+d_1(a_{13}a_{23}-a_{13}a_{23})=0,
\\
	b_{44}w&=-8d_4\otimes a_{14}+8d_4\otimes a_{14}=0.
\end{align*}
By \eqref{Ew} we also have
\begin{align*}
	\frac{1}{3} Ew=&\frac{1}{2}a_{34}(4a_{14})-(x_4\de_1)(-2a_{13}a_{14})-(x_4\de_2)(a_{12}a_{34}+a_{24}a_{13}+a_{14}a_{23}-4)\\&=2a_{34}a_{14}+2a_{43}a_{14}-a_{14}a_{34}-a_{14}a_{43}=0,
\end{align*}
and finally, by \eqref{x4Cw}, we have
\begin{align*}
	(x_4C) w&=-4(x_4 \de_2)a_{12} -4(x_4 \de_3) a_{13} +(h_{14}+h_{24}+h_{34}-5t+5)a_{14}\\
	&-\frac{5}{4}\Big(a_{14}(a_{12}a_{34}+a_{24}a_{13}+a_{14}a_{23}-4)+2a_{24}a_{13}a_{14}+2a_{34}a_{14}a_{12}\Big).
\end{align*}
Now we observe that in $W_t(0,0,0)$ we have $(x_4 \de_1)(a_{12}a_{13}a_{14})=0$ and therefore $a_{12}a_{34}a_{14}=-a_{24}a_{13}a_{14}$. Making this substitution we obtain
\begin{align*}
	(x_4C)w&=(-5-5t)a_{14}+5a_{14}+\frac{5}{2}a_{14}a_{12}a_{34}-\frac{5}{2}a_{34}a_{14}a_{12}=0,
\end{align*}
recalling that $a_{12}a_{34}+a_{34}a_{12}=2C$.

It remains to show that the case $t=0$ doesn't occur. This can be done as follows. Assume by contradiction that we have a singular vector in $M_0(0,1,0)$ of degree 1 and weight $(1,0,0)$; then we have a morphism
\[
\tilde M_{-1}(1,0,0)\rightarrow M_0(0,1,0),
\]
so that, by duality (see Corollary \ref{duali}), we also have a morphism
\[
M_0(2,0,0)\rightarrow \tilde M_1(0,0,1).
\]
By Lemma \ref{lemmagrado1}, this morphism necessarily belongs to case (f) of this lemma, because it implies the existence of a singular vector of degree 1 and weight $(2,0,0)$ in $\tilde M_1(0,0,1)$. In particular there exists a $\phat$-singular vector $z_1\in K_1(0,0,1)$ of weight $(1,0,0)$ such that $w=d_1\otimes z_1$. In fact, up to a costant factor,
\[
z_1=a_{14}x_4^*+a_{13}x_3^*+a_{12}x_2^*=\eta_{1,0}(x_1),
\]
%and
%\[
%z_1'=a_{12}a_{13}a_{14}(a_{24}a_{34}x_4^*+a_{23}a_{34}x_3^*+a_{23}a_{24}x_2^*).
%\]
and this is the only  $\phat$-singular vector of weight $(1,0,0)$ in $K_1(0,0,1)$ by Proposition \ref{decomposizioniK}.
Proceeding as in the proof of Corollary \ref{morphtypea}, we also need to use the further condition $a_{23}a_{24}a_{34}d_1\otimes z_1=0$, 
%(or $a_{23}a_{24}a_{34}d_1\otimes z_1'=0$). 
and one can check that in fact $a_{23}a_{24}a_{34}d_1\otimes z_1\neq0$.
% and $a_{23}a_{24}a_{34}d_1\otimes z_1'\neq 0$.

	\end{proof}
	
	\begin{proposition}
		Assume we are in case (c) in Lemma \ref{lemmagrado1} with $W=W_t(a,b,c)$. Then $W_t(a,b,c)=W_t(0,0,1)$ and, up to a scalar factor, 
		\[
		w=2\de_3\otimes x_4^*-2\de_4\otimes x_3^*-d_4\otimes  a_{12}x_4^*-d_3\otimes a_{12}x_3^*-d_2\otimes a_{12}x_2^*-d_1\otimes a_{12}x_1^*.
		\] 
	The vector $w$	is indeed an $E(4,4)$-singular vector.
	\end{proposition}
	\begin{proof}
		By Proposition \ref{lemmagrado1} we have a singular vector of degree 1 and weight $(0,\gamma+1,\gamma)$ in $M_t(0,\gamma,\gamma+1)$ and in particular we have a morphism
		\[
		\tilde M_{t-1}(0,\gamma+1,\gamma)\rightarrow M_t(0,\gamma,\gamma+1).
		\]
		If $\gamma\geq 1$, by duality, we have a morphism
		\[
		M_{-t}(\gamma+1,\gamma+1,0)\rightarrow \tilde M_{1-t}(\gamma,\gamma+1,0),
		\]
		which forces the existence of a singular vector of weight $(\gamma+1,\gamma+1,0)$ contradicting Lemma \ref{lemmagrado1}.
		Therefore we necessarily have $\gamma=0$. By Corollary \ref{morphtypea} we have a morphism
		\[
		M_{-t}(3,0,0)\rightarrow M_{-t+1}(2,0,0)
		\]
		for all $t\in \C$ and hence by duality we have a morphism
		\[
		M_{t-1}(0,1,0)\rightarrow M_t(0,0,1),
		\]
		which implies the existence of the $E(4,4)$ singular vector $w\in M_t(0,0,1)$. Also in this case we can check that the displayed vector $w$ is such singular vector. Note that
$a_{14} x_4^*+a_{13} x_3^*+a_{12} x_2^*=0$ in $W_t(0,0,1)$ because 
$a_{14} x_4^*+a_{13} x_3^*+a_{12} x_2^*$
 generates $Im(\eta_{1,0})$.		
%		 $W_t(0,0,1)$ is the quotient of $K_t(0,0,1)$ with respect to $Im(\eta_{1,0})$ therefore we have . 
We have
	\begin{align*}
		(x_1\de_2)w&=-d_1\otimes a_{12}x_2^*+d_1\otimes a_{12}x_2^*=0  \\
		(x_2\de_3)w&= -d_2\otimes a_{12}x_3^*+d_2\otimes a_{12}x_3^*=0 \\
		(x_3\de_4)w&= -2\de_4\otimes x_4^*+2\de_4\otimes x_4^*-d_3\otimes a_{12}x_4^*+d_3\otimes a_{12}x_4^*=0, \\
		b_{44} w &=4d_4\otimes x_3^*+4d_4\otimes (x_4 \de_3)x_4^*+4d_3\otimes (x_4 \de_3)x_3^*+4d_2\otimes (x_4 \de_3)x_2^*+4d_1\otimes (x_4 \de_3)x_1^*=0.   \\
	\end{align*}
By \eqref{Ew} we have
\[
\frac{1}{3}Ew=-a_{34}x_3^*+(x_4\de_1)a_{12}x_2^*-(x_4\de_2)a_{12}x_1^*=-(a_{14}x_1^*+a_{24}x_2^*+a_{34}x_3^*)=0,
\]
and finally, by \eqref{x4Cw}, we have
\begin{align*}
(x_4C)w&=-2(x_4 \de_3)x_4^*-\frac{1}{2}(h_{14}+h_{24}+h_{34}-5t+5)x_3^*-\frac{5}{4}(a_{14}a_{12}x_1^*+a_{24}a_{12}x_2^*+a_{34}a_{12}x_3^*)\\
&=\frac{5}{2}tx_3^*+\frac{5}{4}(a_{12}(a_{14}x_1^*+a_{24}x_2^*+a_{34}x_3^*)-a_{12}a_{34}x_3^*-a_{34}a_{12}x_3^*)=0,
\end{align*}
since $a_{14}x_1^*+a_{24}x_2^*+a_{34}x_3^*=-(x_4\de_1)(a_{14} x_4^*+a_{13} x_3^*+a_{12} x_2^*)=0$ and $a_{12}a_{34}+a_{34}a_{12}=C$.

	\end{proof}
	\begin{proposition}
		Case (d) in Lemma \ref{lemmagrado1} with $W=W_t(a,b,c)$ does not occur.
	\end{proposition}
\begin{proof}
If $\gamma=1$ then $M_t(0,0,0)$ contains a unique vector, up to a scalar factor, of degree 1 and weight $(0,0,1)$, namely the vector $\de_4\otimes 1$, but this is not annihilated by $b_{44}$, hence it is not singular.

If $\gamma\neq 1$ and 
	 such singular vector exists, then we have a morphism
	\[
	\tilde M_{t-1}(0,0,\gamma)\rightarrow M_t(0,0,\gamma-1).
	\]
By duality we would then have a morphism
\[
M_{-t}(\gamma+1,0,0)\rightarrow \tilde M_{1-t}(\gamma,0,0)
\]
and so we would have a singular vector as in case (f), i.e. of the form $d_1\otimes z_1$, where $z_1=x_1^\gamma$. This would be the only possibility due to Proposition \ref{decomposizioniK} and Remark \ref{decomposizioniK0}. Proceeding as in the proof of Proposition \ref{typef}, we have that $d_1\otimes z_1$ is a singular vector if and only if $a_{14}z_1=0$, and this is clearly not the case in $K_{1-t}(\gamma,0,0)$. 
\end{proof}
\begin{proposition}
	Assume we are in case (e) in Lemma \ref{lemmagrado1} and $W=W_t(a,b,c)$. Then $a=b=0$ and $c\geq 1$. % and if $c=1$ then $t\neq 1$. 
	In this case the vector 
	\[
	w=d_4\otimes (x_4^*)^c+d_3\otimes x_3^*(x_4^*)^{c-1}+d_2\otimes x_2^*(x_4^*)^{c-1}+d_1\otimes x_1^*(x_4^*)^{c-1}
	\]
	is an $E(4,4)$-singular vector. This singular vector induces a morphism $M_{t-1}(0,0,c-1)\rightarrow M_t(0,0,c)$ if and only if $(c,t)\neq (1,1)$.
\end{proposition}
\begin{proof} By Lemma \ref{lemmagrado1}, $a=b=0$ and  $c\geq 1$. By Corollary \ref{morphtypea}, for every $c\geq 1$ and every $t$ we have a morphism
	\[
	 M_{-t}(c+2,0,0)\rightarrow M_{-t+1}(c+1,0,0),
	\]
hence,  by duality, for $(c, t)\neq (1,1)$, we have a morphism
\[
M_{t-1}(0,0,c-1)\rightarrow M_{t}(0,0,c),
\]
since for every $s\neq 0$, $M_s(0,0,0)\cong M_s(0,1,0)$.
It follows that in $M_t(0,0,c)$ with $(c,t)\neq (1,1)$ there exists a singular vector of weight $(0,0,c-1)$  and degree 1. 
We prove below that the vector $w$ in the statement is a singular vector for all $c\geq 1$ and $t\in \C$. 

Indeed, it is easy to verify that $w$ is annihilated by $x_1\de_2,x_2\de_3,x_3\de_4,b_{44}$. Also $Ew=0$ follows immediately by Equation \eqref{Ew}. So we are left to verify that $(x_4C)w=0$. By Equation \eqref{x4Cw} we only have to show that
\[
a_{14}x_1^*(x_4^*)^{c-1}+a_{24}x_2^*(x_4^*)^{c-1}+a_{34}x_3^*(x_4^*)^{c-1}=0 \,\, \mbox{in}\, W_t(0,0,c).
\]
If $c=1$ this is an immediate consequence of $a_{14}x_4^*+a_{13}x_3^*+a_{12}x_2^*=0$ in $W_t(0,0,1)$, that holds by definition of $W_t(0,0,1)$,  by applying $x_4\de_1$ to this equality. 

So let $c\geq 2$. Let $u_{ij}=\sum_{k}a_{ik}x_k^*x_j^*(x_4^*)^{c-2}$ for all $i,j=1,\ldots,4$ and note that $u_{11}+u_{22}+u_{33}+u_{44}=0$ and that we have to prove $u_{44}=0$.
Recall that $u_{14}$ is a $\phat$-singular vector in the image of $\eta_{1,c-1}$ and so $u_{14}=0$ in $W_t(0,0,c)$. Notice that
\[
(x_4\de_1)u_{14}=u_{44}-(c-1)u_{11},
\] 
and similarly $(x_4 \de_2)(x_2 \de_1)u_{14}=u_{44}-(c-1)u_{22}$ and  $(x_4\de_3)(x_3 \de_1)u_{14}=u_{44}-(c-1)u_{33}$; therefore $(c+2)u_{44}=0$ in $W_t(0,0,c)$.

It remains to show that the induced morphism
\[
f:\tilde M_0(0,0,0)\rightarrow M_1(0,0,1)
\]
does not factor through the quotient $M_0(0,0,0)$. Notice that, by Remark \ref{K0(000)}, $a_{12}=0$ in
$M_0(0,0,0)$ since $W_0(0,0,0)$ is the trivial module, and so we should have $f(a_{12})=a_{12}w=0$. One can check by direct computation that in fact $a_{12}w\neq 0$.
\end{proof}

We summarize the results of this section in the following theorem.
\begin{theorem}\label{listagrado1}
	The following is, up to scalar factors, a complete list of $E(4,4)$-singular vectors  of degree 1 in Verma modules $M_t(a,b,c)$:
\begin{itemize}
	\item $w[1_A]=d_1\otimes x_1^a\in M_t(a,0,0)$ with $a=t=0$ or $a\geq 1$;
	\item $w[1_B]=\sum_{i=1}^4 d_i\otimes x_i^*(x_4^*)^{c-1}\in M_t(0,0,c)$ for all $c\geq 1$;
	\item $w[1_C]=2\de_3\otimes x_4^*-2\de_4 \otimes x_3^*-\sum_{i=1}^4 d_i\otimes a_{12}(x_i^*)\in M_t(0,0,1)$ for all $t$;
	\item $w[1_D]=4\sum_{i=2}^4 \de_i\otimes a_{1i}-2(d_4\otimes a_{12}a_{13}+d_3\otimes a_{14}a_{12}+d_2\otimes a_{13}a_{14})-d_1\otimes(a_{12}a_{34}+a_{24}a_{13}+a_{14}a_{23}-4)\in M_t(0,0,0)$ for all $t\neq 0$;
	\item $w[1_E]=\sum_{i=1}^4 (\de_i\otimes d_i+d_i\otimes \de_i )\in M_1(1,0,0)$.	
	
\end{itemize}
\end{theorem}
Morphisms induced by singular vectors of degree 1 between $E(4,4)$-Verma modules $M_t(a,b,c)$, are described in Figure \ref{degree1}. In this and all other figures, only modules corresponding to integer values of $t$ are drawn, and duality $M(W)\mapsto M(W^*)$ is given by the symmetry with respect to the origin. Note also that in Figure \ref{degree1} Verma modules $M_t(1,0,0)$ induced by the standard modules appear twice for graphical convenience.

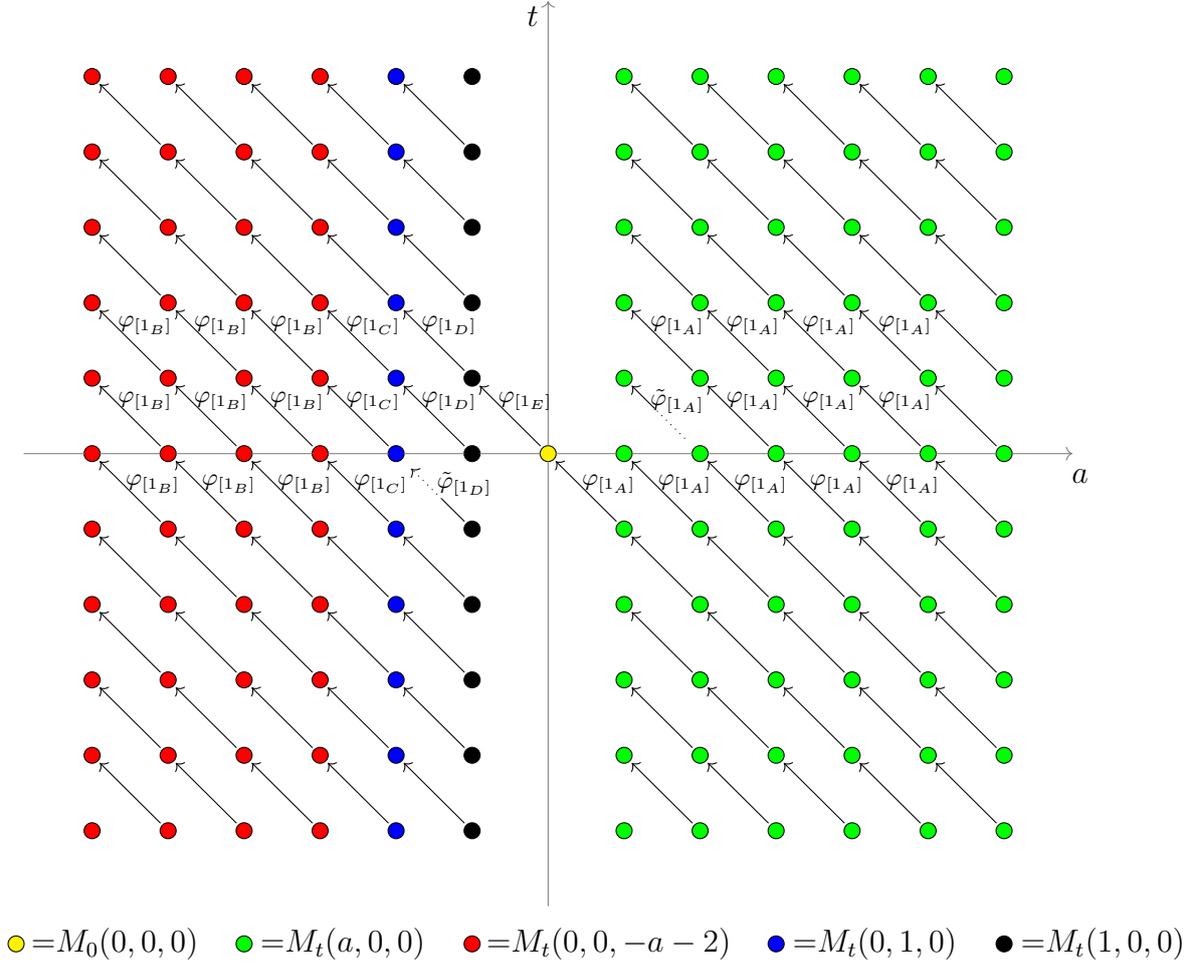
\begin{figure}[h]	
	\begin{center}
		\scalebox{1}{	$$
			\begin{tikzpicture} 
				\draw[fill=green]{(1,0) circle(3pt)}; 
				\draw[fill=green]{(2,0) circle(3pt)};
				\draw[fill=green]{(3,0) circle(3pt)};
				\draw[fill=green]{(4,0) circle(3pt)};
				\draw[fill=green]{(5,0) circle(3pt)};
				\draw[fill=green]{(6,0) circle(3pt)};
				\draw[fill=green]{(1,1) circle(3pt)};
				\draw[fill=green]{(2,1) circle(3pt)};
				\draw[fill=green]{(3,1) circle(3pt)};
				\draw[fill=green]{(4,1) circle(3pt)};
				\draw[fill=green]{(5,1) circle(3pt)};
				\draw[fill=green]{(6,1) circle(3pt)};
				\draw[fill=green]{(1,2) circle(3pt)};
				\draw[fill=green]{(2,2) circle(3pt)};
				\draw[fill=green]{(3,2) circle(3pt)};
				\draw[fill=green]{(4,2) circle(3pt)};
				\draw[fill=green]{(5,2) circle(3pt)};
				\draw[fill=green]{(6,2) circle(3pt)};
				\draw[fill=green]{(1,3) circle(3pt)};
				\draw[fill=green]{(2,3) circle(3pt)};
				\draw[fill=green]{(3,3) circle(3pt)};
				\draw[fill=green]{(4,3) circle(3pt)};
				\draw[fill=green]{(5,3) circle(3pt)};
				\draw[fill=green]{(6,3) circle(3pt)};
				\draw[fill=green]{(1,4) circle(3pt)};
				\draw[fill=green]{(2,4) circle(3pt)};
				\draw[fill=green]{(3,4) circle(3pt)};
				\draw[fill=green]{(4,4) circle(3pt)};
				\draw[fill=green]{(5,4) circle(3pt)};
				\draw[fill=green]{(6,4) circle(3pt)};
				\draw[fill=green]{(1,5) circle(3pt)};
				\draw[fill=green]{(2,5) circle(3pt)};
				\draw[fill=green]{(3,5) circle(3pt)};
				\draw[fill=green]{(4,5) circle(3pt)};
				\draw[fill=green]{(5,5) circle(3pt)};
				\draw[fill=green]{(6,5) circle(3pt)};
				
				\draw[fill=green]{(1,-1) circle(3pt)};
				\draw[fill=green]{(2,-1) circle(3pt)};
				\draw[fill=green]{(3,-1) circle(3pt)};
				\draw[fill=green]{(4,-1) circle(3pt)};
				\draw[fill=green]{(5,-1) circle(3pt)};
				\draw[fill=green]{(6,-1) circle(3pt)};
				\draw[fill=green]{(1,-2) circle(3pt)};
				\draw[fill=green]{(2,-2) circle(3pt)};
				\draw[fill=green]{(3,-2) circle(3pt)};
				\draw[fill=green]{(4,-2) circle(3pt)};
				\draw[fill=green]{(5,-2) circle(3pt)};
				\draw[fill=green]{(6,-2) circle(3pt)};
				\draw[fill=green]{(1,-3) circle(3pt)};
				\draw[fill=green]{(2,-3) circle(3pt)};
				\draw[fill=green]{(3,-3) circle(3pt)};
				\draw[fill=green]{(4,-3) circle(3pt)};
				\draw[fill=green]{(5,-3) circle(3pt)};
				\draw[fill=green]{(6,-3) circle(3pt)};
				\draw[fill=green]{(1,-4) circle(3pt)};
				\draw[fill=green]{(2,-4) circle(3pt)};
				\draw[fill=green]{(3,-4) circle(3pt)};
				\draw[fill=green]{(4,-4) circle(3pt)};
				\draw[fill=green]{(5,-4) circle(3pt)};
				\draw[fill=green]{(6,-4) circle(3pt)};
				\draw[fill=green]{(1,-5) circle(3pt)};
				\draw[fill=green]{(2,-5) circle(3pt)};
				\draw[fill=green]{(3,-5) circle(3pt)};
				\draw[fill=green]{(4,-5) circle(3pt)};
				\draw[fill=green]{(5,-5) circle(3pt)};
				\draw[fill=green]{(6,-5) circle(3pt)};
				\draw[fill=yellow]{(0,0) circle(3pt)};
				
				\draw[fill=black]{(-1,5) circle(3pt)};
				\draw[fill=black]{(-1,4) circle(3pt)};
				\draw[fill=black]{(-1,3) circle(3pt)};
				\draw[fill=black]{(-1,2) circle(3pt)};
				\draw[fill=black]{(-1,1) circle(3pt)};
				\draw[fill=black]{(-1,0) circle(3pt)};
				\draw[fill=black]{(-1,-1) circle(3pt)};
				\draw[fill=black]{(-1,-2) circle(3pt)};
				\draw[fill=black]{(-1,-3) circle(3pt)};
				\draw[fill=black]{(-1,-4) circle(3pt)};
				\draw[fill=black]{(-1,-5) circle(3pt)};
				
				\draw[fill=blue]{(-2,5) circle(3pt)};
				\draw[fill=blue]{(-2,4) circle(3pt)};
				\draw[fill=blue]{(-2,3) circle(3pt)};
				\draw[fill=blue]{(-2,2) circle(3pt)};
				\draw[fill=blue]{(-2,1) circle(3pt)};
				\draw[fill=blue]{(-2,0) circle(3pt)};
				\draw[fill=blue]{(-2,-1) circle(3pt)};
				\draw[fill=blue]{(-2,-2) circle(3pt)};
				\draw[fill=blue]{(-2,-3) circle(3pt)};
				\draw[fill=blue]{(-2,-4) circle(3pt)};
				\draw[fill=blue]{(-2,-5) circle(3pt)};
				
				\draw[fill=red]{(-6,5) circle(3pt)};
				\draw[fill=red]{(-5,5) circle(3pt)};
				\draw[fill=red]{(-4,5) circle(3pt)};
				\draw[fill=red]{(-3,5) circle(3pt)};
				\draw[fill=red]{(-6,4) circle(3pt)};
				\draw[fill=red]{(-5,4) circle(3pt)};
				\draw[fill=red]{(-4,4) circle(3pt)};
				\draw[fill=red]{(-3,4) circle(3pt)};
				\draw[fill=red]{(-6,3) circle(3pt)};
				\draw[fill=red]{(-5,3) circle(3pt)};
				\draw[fill=red]{(-4,3) circle(3pt)};
				\draw[fill=red]{(-3,3) circle(3pt)};
				\draw[fill=red]{(-6,2) circle(3pt)};
				\draw[fill=red]{(-5,2) circle(3pt)};
				\draw[fill=red]{(-4,2) circle(3pt)};
				\draw[fill=red]{(-3,2) circle(3pt)};
				\draw[fill=red]{(-6,1) circle(3pt)};
				\draw[fill=red]{(-5,1) circle(3pt)};
				\draw[fill=red]{(-4,1) circle(3pt)};
				\draw[fill=red]{(-3,1) circle(3pt)};
				\draw[fill=red]{(-6,0) circle(3pt)};
				\draw[fill=red]{(-5,0) circle(3pt)};
				\draw[fill=red]{(-4,0) circle(3pt)};
				\draw[fill=red]{(-3,0) circle(3pt)};
				\draw[fill=red]{(-6,-1) circle(3pt)};
				\draw[fill=red]{(-5,-1) circle(3pt)};
				\draw[fill=red]{(-4,-1) circle(3pt)};
				\draw[fill=red]{(-3,-1) circle(3pt)};
				\draw[fill=red]{(-6,-2) circle(3pt)};
				\draw[fill=red]{(-5,-2) circle(3pt)};
				\draw[fill=red]{(-4,-2) circle(3pt)};
				\draw[fill=red]{(-3,-2) circle(3pt)};
				\draw[fill=red]{(-6,-3) circle(3pt)};
				\draw[fill=red]{(-5,-3) circle(3pt)};
				\draw[fill=red]{(-4,-3) circle(3pt)};
				\draw[fill=red]{(-3,-3) circle(3pt)};
				\draw[fill=red]{(-6,-4) circle(3pt)};
				\draw[fill=red]{(-5,-4) circle(3pt)};
				\draw[fill=red]{(-4,-4) circle(3pt)};
				\draw[fill=red]{(-3,-4) circle(3pt)};
				\draw[fill=red]{(-6,-5) circle(3pt)};
				\draw[fill=red]{(-5,-5) circle(3pt)};
				\draw[fill=red]{(-4,-5) circle(3pt)};
				\draw[fill=red]{(-3,-5) circle(3pt)};

				%assi
				\draw[gray](-6.9,0)--(-6.1,0);
				\draw[gray](-5.9,0)--(-5.1,0);
				\draw[gray](-4.9,0)--(-4.1,0);
				\draw[gray](-3.9,0)--(-3.1,0);
				\draw[gray](-2.9,0)--(-2.1,0);
				\draw[gray](-1.9,0)--(-1.1,0);
				\draw[gray](-0.9,0)--(-0.1,0);
				\draw[gray](0.1,0)--(0.9,0);
				\draw[gray](1.1,0)--(1.9,0);
				\draw[gray](2.1,0)--(2.9,0);
				\draw[gray](3.1,0)--(3.9,0);
				\draw[gray](4.1,0)--(4.9,0);
				\draw[gray](5.1,0)--(5.9,0);
				\draw[->,gray](6.1,0)--(6.9,0);
				\draw[gray](0,-6)--(0,-0.1);
				\draw[->,gray](0,0.1)--(0,6);

				%quadrante 1
				\draw[->,black](1.4,0.6)--(1.1,0.9);
				\draw[black, dotted](1.8,0.2)--(1.4,0.6);
				\draw[->,black](2.9,0.1)--(2.1,0.9);
				\draw[->,black](3.9,0.1)--(3.1,0.9);
				\draw[->,black](4.9,0.1)--(4.1,0.9);
				\draw[->,black](5.9,0.1)--(5.1,0.9);
				\draw[->,black](1.9,1.1)--(1.1,1.9);
				\draw[->,black](2.9,1.1)--(2.1,1.9);
				\draw[->,black](3.9,1.1)--(3.1,1.9);
				\draw[->,black](4.9,1.1)--(4.1,1.9);
				\draw[->,black](5.9,1.1)--(5.1,1.9);
				\draw[->,black](1.9,2.1)--(1.1,2.9);
				\draw[->,black](2.9,2.1)--(2.1,2.9);
				\draw[->,black](3.9,2.1)--(3.1,2.9);
				\draw[->,black](4.9,2.1)--(4.1,2.9);
				\draw[->,black](5.9,2.1)--(5.1,2.9);
				\draw[->,black](1.9,3.1)--(1.1,3.9);
				\draw[->,black](2.9,3.1)--(2.1,3.9);
				\draw[->,black](3.9,3.1)--(3.1,3.9);
				\draw[->,black](4.9,3.1)--(4.1,3.9);
				\draw[->,black](5.9,3.1)--(5.1,3.9);
				\draw[->,black](1.9,4.1)--(1.1,4.9);
				\draw[->,black](2.9,4.1)--(2.1,4.9);
				\draw[->,black](3.9,4.1)--(3.1,4.9);
				\draw[->,black](4.9,4.1)--(4.1,4.9);
				\draw[->,black](5.9,4.1)--(5.1,4.9);
				%quadrante 4
				\draw[->,black](1.9,-0.9)--(1.1,-0.1);
				\draw[->,black](2.9,-0.9)--(2.1,-0.1);
				\draw[->,black](3.9,-0.9)--(3.1,-0.1);
				\draw[->,black](4.9,-0.9)--(4.1,-0.1);
				\draw[->,black](5.9,-0.9)--(5.1,-0.1);
				\draw[->,black](1.9,-1.9)--(1.1,-1.1);
				\draw[->,black](2.9,-1.9)--(2.1,-1.1);
				\draw[->,black](3.9,-1.9)--(3.1,-1.1);
				\draw[->,black](4.9,-1.9)--(4.1,-1.1);
				\draw[->,black](5.9,-1.9)--(5.1,-1.1);
				\draw[->,black](1.9,-2.9)--(1.1,-2.1);
				\draw[->,black](2.9,-2.9)--(2.1,-2.1);
				\draw[->,black](3.9,-2.9)--(3.1,-2.1);
				\draw[->,black](4.9,-2.9)--(4.1,-2.1);
				\draw[->,black](5.9,-2.9)--(5.1,-2.1);
				\draw[->,black](1.9,-3.9)--(1.1,-3.1);
				\draw[->,black](2.9,-3.9)--(2.1,-3.1);
				\draw[->,black](3.9,-3.9)--(3.1,-3.1);
				\draw[->,black](4.9,-3.9)--(4.1,-3.1);
				\draw[->,black](5.9,-3.9)--(5.1,-3.1);
				\draw[->,black](1.9,-4.9)--(1.1,-4.1);
				\draw[->,black](2.9,-4.9)--(2.1,-4.1);
				\draw[->,black](3.9,-4.9)--(3.1,-4.1);
				\draw[->,black](4.9,-4.9)--(4.1,-4.1);
				\draw[->,black](5.9,-4.9)--(5.1,-4.1);
				
				%quadrante 3
				\draw[black](-1.1,-0.9)--(-1.4,-0.6);
				\draw[->,dotted,black](-1.4,-0.6)--(-1.8,-0.2);
				\draw[->,black](-2.1,-0.9)--(-2.9,-0.1);
				\draw[->,black](-3.1,-0.9)--(-3.9,-0.1);
				\draw[->,black](-4.1,-0.9)--(-4.9,-0.1);
				\draw[->,black](-5.1,-0.9)--(-5.9,-0.1);
				\draw[->,black](-1.1,-1.9)--(-1.9,-1.1);
				\draw[->,black](-2.1,-1.9)--(-2.9,-1.1);
				\draw[->,black](-3.1,-1.9)--(-3.9,-1.1);
				\draw[->,black](-4.1,-1.9)--(-4.9,-1.1);
				\draw[->,black](-5.1,-1.9)--(-5.9,-1.1);
				\draw[->,black](-1.1,-2.9)--(-1.9,-2.1);
				\draw[->,black](-2.1,-2.9)--(-2.9,-2.1);
				\draw[->,black](-3.1,-2.9)--(-3.9,-2.1);
				\draw[->,black](-4.1,-2.9)--(-4.9,-2.1);
				\draw[->,black](-5.1,-2.9)--(-5.9,-2.1);
				\draw[->,black](-1.1,-3.9)--(-1.9,-3.1);
				\draw[->,black](-2.1,-3.9)--(-2.9,-3.1);
				\draw[->,black](-3.1,-3.9)--(-3.9,-3.1);
				\draw[->,black](-4.1,-3.9)--(-4.9,-3.1);
				\draw[->,black](-5.1,-3.9)--(-5.9,-3.1);
				\draw[->,black](-1.1,-4.9)--(-1.9,-4.1);
				\draw[->,black](-2.1,-4.9)--(-2.9,-4.1);
				\draw[->,black](-3.1,-4.9)--(-3.9,-4.1);
				\draw[->,black](-4.1,-4.9)--(-4.9,-4.1);
				\draw[->,black](-5.1,-4.9)--(-5.9,-4.1);
				%quadrante 2
				\draw[->,black](-1.1,0.1)--(-1.9,0.9);
				\draw[->,black](-2.1,0.1)--(-2.9,0.9);
				\draw[->,black](-3.1,0.1)--(-3.9,0.9);
				\draw[->,black](-4.1,0.1)--(-4.9,0.9);
				\draw[->,black](-5.1,0.1)--(-5.9,0.9);
				\draw[->,black](-1.1,1.1)--(-1.9,1.9);
				\draw[->,black](-2.1,1.1)--(-2.9,1.9);
				\draw[->,black](-3.1,1.1)--(-3.9,1.9);
				\draw[->,black](-4.1,1.1)--(-4.9,1.9);
				\draw[->,black](-5.1,1.1)--(-5.9,1.9);
				\draw[->,black](-1.1,2.1)--(-1.9,2.9);
				\draw[->,black](-2.1,2.1)--(-2.9,2.9);
				\draw[->,black](-3.1,2.1)--(-3.9,2.9);
				\draw[->,black](-4.1,2.1)--(-4.9,2.9);
				\draw[->,black](-5.1,2.1)--(-5.9,2.9);
				\draw[->,black](-1.1,3.1)--(-1.9,3.9);
				\draw[->,black](-2.1,3.1)--(-2.9,3.9);
				\draw[->,black](-3.1,3.1)--(-3.9,3.9);
				\draw[->,black](-4.1,3.1)--(-4.9,3.9);
				\draw[->,black](-5.1,3.1)--(-5.9,3.9);
				\draw[->,black](-1.1,4.1)--(-1.9,4.9);
				\draw[->,black](-2.1,4.1)--(-2.9,4.9);
				\draw[->,black](-3.1,4.1)--(-3.9,4.9);
				\draw[->,black](-4.1,4.1)--(-4.9,4.9);
				\draw[->,black](-5.1,4.1)--(-5.9,4.9);
				
				\draw[->,black](0.9,-0.9)--(0.1,-0.1);
				\draw[->,black](-0.1,0.1)--(-0.9,0.9);
				
				\draw[fill=yellow]{(-7,-6.5) circle(3pt)};
				\node at (-5.7,-6.5){=$M_0(0,0,0)$};
				\draw[fill=green]{(-4,-6.5) circle(3pt)};
				\node at (-2.7,-6.5){=$M_t(a,0,0)$};
				\draw[fill=red]{(-1,-6.5) circle(3pt)};
				\node at (0.8,-6.5){=$M_t(0,0,-a-2)$};
				\draw[fill=blue]{(3,-6.5) circle(3pt)};
				\node at (4.3,-6.5){=$M_t(0,1,0)$};
				\draw[fill=black]{(6,-6.5) circle(3pt)};
				\node at (7.3,-6.5){=$M_t(1,0,0)$};
				
				\node at (7,-0.3){$a$};
				\node at (-0.2,5.8){$t$};

				\node at (0.8,-0.4){\scriptsize$\varphi_{[1_A]}$};
				\node at (-2.2,-0.4){\scriptsize$\varphi_{{[1_{C}]}}$};
				\node at (-3.2,-0.4){\scriptsize$\varphi_{[1_B]}$};
				\node at (-4.2,-0.4){\scriptsize$\varphi_{[1_B]}$};
				\node at (-5.2,-0.4){\scriptsize$\varphi_{[1_B]}$};
					\node at (1.8,-0.4){\scriptsize$\varphi_{[1_A]}$};
				\node at (2.8,-0.4){\scriptsize$\varphi_{[1_A]}$};
				\node at (3.8,-0.4){\scriptsize$\varphi_{[1_A]}$};
				\node at (4.8,-0.4){\scriptsize$\varphi_{[1_A]}$};
				
				\node at (-0.3,0.7){\scriptsize$\varphi_{[1_E]}$};
				\node at (-1.3,0.7){\scriptsize$\varphi_{[1_D]}$};
				\node at (-2.3,0.7){\scriptsize$\varphi_{[1_C]}$};
				\node at (-3.3,0.7){\scriptsize$\varphi_{[1_B]}$};
				\node at (-4.3,0.7){\scriptsize$\varphi_{[1_B]}$};
				\node at (-5.3,0.7){\scriptsize$\varphi_{[1_B]}$};
				\node at (1.7,0.7){\scriptsize$\tilde \varphi_{[1_A]}$};
				\node at (2.7,0.7){\scriptsize$\varphi_{[1_A]}$};
				\node at (3.7,0.7){\scriptsize$\varphi_{[1_A]}$};
				\node at (4.7,0.7){\scriptsize$\varphi_{[1_A]}$};

				\node at (-1.3,1.7){\scriptsize$\varphi_{[1_D]}$};
				\node at (-1.1,-0.4){\scriptsize$\tilde \varphi_{[1_D]}$};
				\node at (-2.3,1.7){\scriptsize$\varphi_{{[1_{C}]}}$};
				\node at (-3.3,1.7){\scriptsize$\varphi_{[1_B]}$};
				\node at (-4.3,1.7){\scriptsize$\varphi_{[1_B]}$};
				\node at (-5.3,1.7){\scriptsize$\varphi_{[1_B]}$};
				\node at (1.7,1.7){\scriptsize$\varphi_{[1_A]}$};
				\node at (2.7,1.7){\scriptsize$\varphi_{[1_A]}$};
				\node at (3.7,1.7){\scriptsize$\varphi_{[1_A]}$};
				\node at (4.7,1.7){\scriptsize$\varphi_{[1_A]}$};
				
					\end{tikzpicture}$$}
				
			\end{center}
		
\caption{Nonzero morphisms of degree 1 between Verma modules.} \label{degree1}
\end{figure}

\begin{corollary}
	We have the following list of non-zero morphisms, up to scalar factors:
	\begin{itemize}
		\item $\varphi_{[1_A]}: M_{t-1}(a+1,0,0)\rightarrow M_t(a,0,0)$ with $a=t=0$, or $a=1$ and $t\neq 1$, or $a\geq 2$ and all $t$;
		\item $\tilde \varphi_{[1_A]}:\tilde M_0(2,0,0)\rightarrow M_1(1,0,0)$;
		\item $\varphi_{[1_B]}: M_{t-1}(0,0,c-1)\rightarrow M_t(0,0,c)$ for all $c\geq 1$ and $(c,t)\neq (1,1)$;
		\item $\varphi_{[1_C]}:M_{t-1}(0,1,0)\rightarrow M_t(0,0,1)$ for all $t$.
		\item $\varphi_{[1_D]}:M_{t-1}(1,0,0)\rightarrow M_t(0,0,0)$ for all $t\neq 0$;
		\item $\tilde \varphi_{[1_D]}: M_{-1}(1,0,0)\rightarrow \tilde M_0(0,1,0)$;
		\item $\varphi_{[1_E]}: M_0(0,0,0)\rightarrow M_1(1,0,0)$
	\end{itemize}
\end{corollary}

Morphisms $\tilde \varphi_{[1_A]}$ and $\tilde \varphi_{[1_D]}$ appear in Figure \ref{degree1} as dotted arrows.

\begin{remark}
	For $c=1$ and $t\neq 1$ we have, up to scalar factors: 
	$\varphi_{[1_B]}=\varphi_{[1_C]}$. Indeed, by Remark \ref{a12}, $W_{t-1}(0,0,0)\cong W_{t-1}(0,1,0)$, and  $\varphi_{[1_B]}(a_{12})=a_{12}\varphi_{[1_B]}(1)=a_{12}w[1_B]=w[1_C]$.
\end{remark}
\newpage
\begin{remark}
	We notice that the morphism $\varphi_{[1_A]}:M_{t-1}(a+1,0,0)\rightarrow M_t(a,0,0)$ is the dual to
	\begin{itemize}
		\item $\varphi_{[1_B]}:M_{-t}(0,0,a-1)\rightarrow M_{-t+1}(0,0,a)$ for $a\geq 3$;
		\item $\varphi_{[1_C]}:M_{-t}(0,1,0)\rightarrow M_{-t+1}(0,0,1)$ for $a=2$;
		\item $\varphi_{[1_D]}:M_{-t}(1,0,0)\rightarrow M_{-t+1}(0,0,0)$ for $a=1$ and $t\neq 1$;
		\item $\varphi_{[1_E]}:M_0(0,0,0)\rightarrow M_1(1,0,0)$ for $a=t=0$.
	\end{itemize}
\end{remark}

It is natural to look for an explicit expression of these morphisms. In these expression we prefer to consider the singular vector generating $K_t(\lambda)$ to be even, so that all morphisms $\varphi_{[1_A]}, \dots,\varphi_{[1_E]}$ are odd.
\color{black}
An odd linear map $\varphi: M_{t-1}(\lambda)\rightarrow M_{t}(\sigma)$ of degree one can always be associated to an odd element $\Phi\in L_{-1}\otimes \Hom(W_{t-1}(\lambda),W_{t}(\sigma))$ as follows: for $u\in {\mathcal U}(L_{-1})$ and $v\in W_{t-1}(\lambda)$ we let
$$\varphi(u\otimes v)=(-1)^{p(u)}u\varphi(v)$$
where, if $\Phi=\sum_id_i\otimes \theta_i+
\sum_i\de_i\otimes \psi_i$ with
$\theta_i, \psi_i\in\Hom(W_{t-1}(\lambda),W_{t}(\sigma))$, we let $\varphi(v)=\sum_id_i\otimes \theta_i(v)+
\sum_i\de_i\otimes \psi_i(v)$. Here $p(u)$ denotes the parity of the element $u$.
%We will say that $\varphi$ (or $\Phi$) is a morphism of degree $d$ if $u_i\in (U_-)_d$ for every $i$.

\medskip

One can easily check the following characterization of morphisms between Verma modules.

\begin{proposition} (Cf.\ \cite[Proposition 2.1]{KR4}) \label{morphisms}
Let $\varphi: M_{t-1}(\lambda)\rightarrow M_{t}(\sigma)$ be the linear map associated with the element $\Phi\in L_{-1}\otimes \Hom(W_{t-1}(\lambda),W_{t}(\sigma))$.
Then $\varphi$ is a morphism of $L$-modules if and only if the following conditions
hold:
\begin{itemize}
\item[(a)] $L_0.\Phi=0$;
\item[(b)] $x\varphi(v)=0$ for every $x\in L_1$ and for every $v\in W_{t-1}(\lambda)$.
\end{itemize} 
\end{proposition}

In order to construct the maps $\theta_i$ and $\psi_i$ in our cases we proceed as follows. We  first define an even morphism of $E(4,4)_0=\phat$-modules
\[\rho: E(4,4)^*_{-1}\rightarrow \Hom(K_{t-1}(\lambda),K_{t}(\sigma))\]
 (where $E(4,4)_{-1}^*\cong W_1(1,0,0)$ is identified with the standard module with $t=1$ as in Example \ref{standard}).

In turn, to construct such map $\rho$ it is enough to determine a map 
\[
\rho:E(4,4)_{-1}^*\rightarrow \Hom(F_{t-1}(\lambda),K_{t}(\sigma))\]
which is a morphism of $\frak p^+$-modules, where we recall that $\mathfrak p^+=\phat_{-2}\oplus \phat_0\oplus \phat_1$, i.e.,  for every $b\in \frak p^+$ we have
\begin{equation}\label{check1}
	(-1)^{p(b)p(x)}(-(b.x)(v)+bx(v))=x(bv).
\end{equation}

The extension of $\rho(x)$ to $K_{t-1}(\lambda)$ is then uniquely determined as follows.
\color{black}
In order to simplify the notation we will identify an element $x$ of $E(4,4)_{-1}^*$ with its image $\rho(x)$ and for 
$x\in E(4,4)_{-1}^*$ and $v\in F_{t-1}(\lambda)$ we define inductively
\[x(a\otimes v)=(-1)^{p(a)p(x)}(-(a.x)(v)+
ax(v))\]
for all $a\in \phat$, and
for $k\geq 2$,
$a_1, \dots, a_k\in \phat$, we let
\[x(a_1\dots a_k\otimes v)=(-1)^{p(a_1)p(x)}(-(a_1.x)(a_2\dots a_k\otimes v)+
a_1x(a_2\dots a_k\otimes v).\]

Note that \eqref{check1} is equivalent to the fact that $\rho(x)$ is a balanced map, i.e. $x(ub\otimes v)=x(u\otimes bv)$ for all $u\in \mathcal U(\phat)$, $b\in \frak p^+$, $x\in E(4,4)_{-1}^*$ and $v\in F_{t-1}(\lambda)$. Note also that the recursive definition of $\rho(x)$ immediately implies that $\rho:E(4,4)_{-1}^*\rightarrow \Hom(K_{t-1}(\lambda),K_{t}(\sigma))$ is indeed $\phat$-equivariant.

We will then show in each case that $\rho(x)$ induces a map $W_{t-1}(\lambda)\rightarrow W_{t}(\sigma)$ that will be still denoted by $x$ and we will set $\theta_i=\rho(\de_i)$ and $\psi_i=\rho(d_i)$, so that condition (a) in Proposition \ref{morphisms} will be automatically satisfied.

In what follows we will denote by $\{x_1,\dots, x_4\}$ the standard basis of $\C^4$, by
$\{x_1^*,\dots, x_4^*\}$ the corresponding dual basis of $(\C^4)^*$ and by $\{x_i\wedge x_j\,|\,1\leq i<j\leq 4\}$ the standard basis of $\inlinewedge^2(\C^4)$. Also, as above, $d_i$ stands for $dx_i$, $\de_i$ for $\de/\de x_i$, and $a_{ij}$ for $x_idx_j-x_jdx_i$. We have $F_{t}(a,0,0)\cong S^a\C^4$, 
$F_{t}(0,0,c)\cong S^c\C^{4^*}$. In order to define a morphism $\varphi: M_{t-1}(\lambda)\rightarrow M_{t}(\sigma)$ it is sufficient to define its restriction to $F_{t-1}(\lambda)$.

 \begin{proposition}\label{explicitmorphisms} The following is an explicit description of the morphisms of degree 1 between Verma modules over the Lie superalgebra $E(4,4)$, restricted to the corresponding $F_{t-1}(\lambda)$:
 \begin{itemize}
 \item $\varphi_{[1_A]}: F_{t-1}(a,0,0)\rightarrow M_{t}(a-1,0,0)$ is given by
 \[\varphi_{[1_A]}(f)=\sum_id_i\otimes \de_i(f);\]
  \item $\varphi_{[1_B]}: F_{t-1}(0,0,c-1)\rightarrow M_t(0,0,c)$ is given by
   \[\varphi_{[1_B]}(f)=\sum_id_i\otimes x_i^*f;\]
    \item $\varphi_{[1_C]}: F_{t-1}(0,1,0)\rightarrow M_{t}(0,0,1)$ is given by 
    \[\varphi_{[1_C]}(x_h\wedge x_k)=\sum_{i}(d_i \otimes a_{hk}(x_i^*)+2\de_i\otimes x_i\wedge x_h\wedge x_k),\]
    where, as usual, $\inlinewedge ^3(\C^4)$ is identified with $\C^{4^*}$ by contraction with the standard volume form;
    \item $\varphi_{[1_D]}: F_{t-1}(1,0,0)\rightarrow M_{t}(0,0,0)$ is  given by
    \[\varphi_{[1_D]}(x_r)=\sum_i(d_i\otimes \theta_i(x_r)+
    4\de_i\otimes a_{ri}),\]
where   \[\theta_i(x_r)=\begin{cases}-\varepsilon_{ijhk}(a_{ij}a_{hk}+a_{ih}a_{kj}+a_{ik}a_{jh})+4+2t& \textrm{if }r=i,\\
    	2\varepsilon_{irhk}a_{rh}a_{rk} & \textrm{if }r\neq i 
    \end{cases}
    \]
    for any choice of $j,h,k$ such that the indices of $\varepsilon$ are distinct;
    %where $ijhk$ is an even permutation of $1234$;
   \item $\varphi_{[1_E]}: F_0(0,0,0)\rightarrow M_1(1,0,0)$ is given by 
\[\varphi_{[1_E]}(1)=\sum_i(d_i\otimes \de_i+\de_i\otimes d_i).\]
 \end{itemize}
  \end{proposition}
 \begin{proof}

 	First note that condition \eqref{check1} is always automatically satisfied for $b=C$ since $C.x=x$, $bx(v)=tx(v)$ and $bv=(t-1)v$.
 	
 Let us verify condition \eqref{check1} for the morphism $\varphi_{[1_A]}$ and $b\in \phat_{\geq 0}$. If $b\in \phat_{>0}$ then the three terms in \eqref{check1} vanish. Indeed $\phat_{>0}$ acts trivially on $F_{t-1}(a,0,0)$ and $F_{t}(a-1,0,0)$,
 moreover if $x$ is even then $b.x$ is an odd element in $E(4,4)_{-1}^*$, hence it vanishes on
 $F_{t-1}(a,0,0)$; if $x$ is odd then $b.x=0$. If $b\in\phat_0=\slq$ and $x$ is even, then
 $-[b,x]+bx=xb$; if $x$ is odd then it acts trivially on $F_{t-1}(a,0,0)$ and $b.x$ is an odd element in $E(4,4)_{-1}^*$, hence it vanishes on $F_{t-1}(a,0,0)$.
 
 Now recall that for every $t$ and $a\geq 1$ or $t=a=0$, the module $W_t(a,0,0)$ is the quotient of $K_t(a,0,0)$ with respect to the submodule $S_{t,a}$  generated by the singular vector $w=a_{12}x_1^a$. In order to show that $x$ induces a map from $W_{t-1}(a,0,0)$ to $W_{t}(a-1,0,0)$, we need to prove that $xS_{t-1,a}\subset S_{t,a-1}$. We have: $\theta_1(a_{12}x_1^a)=aa_{12}x_1^{a-1}$ in $S_{t,a-1}$. Moreover, $\theta_i(a_{12}x_1^a)=0$ for every $i\neq 1$ and $\psi_j(a_{12}x_1^a)=0$ for every $j=1,2,3,4$. Now notice that if $v\in K_{t-1}(a,0,0)$ is such that $xv\in S_{t,a-1}$ for every $x\in E(4,4)_{-1}^*$, then for every $u\in \phat$ we also have $xuv\in S_{t,a-1}$. It follows that $x: W_{t-1}(a,0,0)\rightarrow W_{t}(a-1,0,0)$ is well defined.

Finally notice that $\varphi_{[1_A]}(x_1^a)=d_1\otimes ax_1^{a-1}$, so condition (b) in Proposition \ref{morphisms} is satisfied by Theorem \ref{listagrado1}.
 
For  the other morphisms the proof is similar.
%$\varphi=\varphi_{1_B}$ the proof is similar. 
%Notice that $\sum_{i=1}^4(d_i\otimes\theta_i
%+ \de_i\otimes \psi_i)((x_4^*)^{c-1})=\sum_i d_i\otimes x_i^* %(x_4^*)^{c-1}$, so condition (b) in Proposition \ref{morphisms} is again satisfied by Theorem \ref{listagrado1}.
 
 \end{proof}
% \begin{example} To illustrate Proposition \ref{morphisms} we compute explicitly $\varphi_{[1C]}(d_1\de_2\otimes a_{13} x_{12})$, where $x_{ij}$ is a shortcut for $x_{i}\wedge x_j$. We first have $\varphi_{[1C]}(d_1\de_2\otimes a_{13} x_{12})=-d_1\de_2 \otimes \varphi_{[1C]}( a_{13} x_{12})$ so it is enough to compute 
 %	\begin{align*}
 %	\varphi_{[1C]}( a_{13} x_{12})= &\sum_i d_i\otimes (-(a_{13}.\theta_i)(x_{12})+a_{13}\theta_i(x_{12}))+\de_i((a_{13}.\psi_i)(x_{12})-a_{13}\psi_i(x_{12})))\\
 %	=&d_1\otimes (-\psi_3(x_{12})+a_{13}a_{12}x_1^*)+\sum_{i=2}^4 d_i \otimes a_{13}a_{12}x_i^*\\
 %	&+ \de_2\otimes 2 \theta_4(x_{12})+2\de_3\otimes a_{13}x_4^*+2\de_4\otimes (-\theta_2(x_{12})-a_{13}x_3^*)\\
 %	&= 2d_1\otimes x_4^*+\sum_{i=1}^4d_i\otimes a_{13}a_{12}x_i^*+2\de_2\otimes a_{12}x_4^*-2\de_3\otimes a_{13}x_4^*+2\de_4\otimes (-a_{12}x_2^*+a_{13}x_3^*)
 %	\end{align*}
 %and so  \[\varphi_{[1C]}(d_1\de_2\otimes a_{13} x_{12})= -\sum_{i=2}^4 d_1d_i\de_2\otimes a_{13}a_{12}x_i^*-2d_1\de_2^2 \otimes a_{12}x_4^*+2d_1\de_2\de_3\otimes a_{13}x_4^*+2d_1\de_2\de_4\otimes(a_{12}x_2^*-a_{13}x_3^*).\]
 %Note that one can obtain the same result by using the relation
 %\[
 %\varphi_{[1C]}(d_1\de_2\otimes a_{13}x_{12})=d_1\de_2a_{13}\varphi_{[1C]}(x_{12}).
% \]
%\end{example}
 
 \color{black}
\section{Singular vectors of degree 2}\label{degree2}
In this section we classify singular vectors of degree 2 in Verma modules $M_t(a,b,c)$. We first observe that the unique possible non zero compositions of morphisms of degree 1 between finite Verma modules are $\varphi_{[1_D]}\circ \varphi_{[1_A]}:M_{t-2}(2,0,0)\rightarrow M_t(0,0,0)$ (for $t\neq 0$) and $\varphi_{[1_E]}\circ \varphi_{[1_A]}:M_{-1}(1,0,0)\rightarrow M_1(1,0,0)$. Indeed the other possible compositions are $\varphi_{[1_A]}\circ \varphi_{[1_A]} =0$ (since $\varphi_{[1_A]}$ is induced by the singular vector $d_1\otimes x_1^a$ and  $d_1^2=0$) or their duals.
Moreover, the following result shows that $\varphi_{[1_D]}\circ \varphi_{[1_A]}$ and $\varphi_{[1_E]}\circ \varphi_{[1_A]}$ are nonzero.
\begin{proposition}
	The vector $w[2_{EA}]:=d_1w{[1_E]}\in M_1(1,0,0)$ is singular of degree 2 and weight $(1,0,0)$.
	The vector $w[2_{DA}]:=d_1w{[1_D]}\in M_t(0,0,0)$
	is singular of degree 2 and weight $(2,0,0)$ for all $t\neq 0$.
\end{proposition}
\begin{proof}
	Using the expressions of $w{[1_D]}$ and $w{[1_E]}$ given in Theorem \ref{listagrado1}, one can check that $d_1w{[1_D]}\neq 0$ and $d_1w{[1_E]}\neq 0$. Now the result immediately follows by noticing that $d_1w{[1_E]}=\varphi_{[1_E]}(\varphi_{[1_A]}(x_1))$ and $d_1w{[1_D]}=\varphi_{[1_D]} (\varphi_{[1_A]}(x_1^2))$.
\end{proof}

\begin{remark}
Let us write the vector $w[2_{DA}]$ explicitly:
\[w[2_{DA}]=4\sum_{i=2}^4 d_1\de_i\otimes a_{1i}-2d_1(d_4\otimes a_{12}a_{13}+d_3\otimes a_{14}a_{12}+d_2\otimes a_{13}a_{14}).\]
This vector lies in $M_0(0,1,0)$, where, by Remark \ref{K0(000)}, we use $W_0(0,1,0)\cong M_1/M_2$, and 
one can check that it is indeed singular also for $t=0$.
\end{remark}

We show below that there are no other $E(4,4)$ singular vectors of degree 2.

In Section \ref{sec:sing} we introduced notation for vectors of degree $k$ in $M_t(a,b,c)$. For the sake of simplicity when $k=2$ we write
\begin{align}
	w&=\frac{1}{2}\sum_{i=1}^4 \de_i^2\otimes v_{ii}+\sum_{i<j}\de_i \de_j \otimes v_{ij}+\sum _{i,j}\de_i d_j \otimes y_{ij}+\sum_{i<j}d_id_j\otimes z_{ij}.
	\label{12a}
\end{align}
We assume that $w$ is an $E(4,4)$ singular vector of weight $(\alpha,\beta,\gamma)$ and for the reader's convenience we list the corresponding weights of the vector $v_{ij}$, $y_{ij}$ and $z_{ij}$ in Table \ref{tabwei}. 
\begin{table}
	\caption{Weights of vectors $v_{ij},y_{ij},z_{ij}$. \label{tabwei}}
\begin{tabular}{|c|c|c|c|c|c|}
	\hline
	$v_{11}$& $(\alpha+2,\beta,\gamma)$ & 
	$y_{11}$& $(\alpha,\beta,\gamma)$&
	$z_{12} $& $(\alpha,\beta-1,\gamma)$\\
	\hline 
	$v_{12}$& $(\alpha,\beta+1,\gamma)$ & 
	$y_{12}$& $(\alpha+2,\beta-1,\gamma)$&
	$z_{13} $& $(\alpha-1,\beta+1,\gamma-1)$\\
	\hline
	$v_{13}$& $(\alpha+1,\beta-1,\gamma+1)$ & 
	$y_{13}$& $(\alpha+1,\beta+1,\gamma-1)$&
	$z_{14} $& $(\alpha-1,\beta,\gamma+1)$\\
	\hline
	$v_{14}$& $(\alpha+1,\beta,\gamma-1)$ & 
	$y_{14}$& $(\alpha+1,\beta,\gamma+1)$&
	$z_{23} $& $(\alpha+1,\beta,\gamma-1)$\\
	\hline
	$v_{22}$& $(\alpha-2,\beta+2,\gamma)$ & 
	$y_{21}$& $(\alpha-2,\beta+1,\gamma)$&
	$z_{24} $& $(\alpha+1,\beta-1,\gamma+1)$\\
	\hline
	$v_{23}$& $(\alpha-1,\beta,\gamma+1)$ & 
	$y_{22}$& $(\alpha,\beta,\gamma)$&
	$z_{34} $& $(\alpha,\beta+1,\gamma)$\\
	\hline
	$v_{24}$& $(\alpha-1,\beta+1,\gamma-1)$ & 
	$y_{23}$& $(\alpha-1,\beta+2,\gamma-1)$&
	& \\
	\hline
	$v_{33}$& $(\alpha,\beta-2,\gamma+2)$ & 
	$y_{24}$& $(\alpha-1,\beta+1,\gamma+1)$&
	& \\
	\hline
	$v_{34}$& $(\alpha,\beta-1,\gamma)$ & 
	$y_{31}$& $(\alpha-1,\beta-1,\gamma+1)$&
	& \\
	\hline
	$v_{44}$& $(\alpha,\beta,\gamma-2)$ & 
	$y_{32}$& $(\alpha+1,\beta-2,\gamma+1)$&
	& \\
	\hline
	& & 
	$y_{33}$& $(\alpha,\beta,\gamma)$&
	& \\
	\hline
	& & 
	$y_{34}$& $(\alpha,\beta-1,\gamma+2)$&
	& \\
	\hline
	& & 
	$y_{41}$& $(\alpha-1,\beta,\gamma-1)$&
	& \\
	\hline
	& & 
	$y_{42}$& $(\alpha+1,\beta-1,\gamma-1)$&
	& \\
	\hline
	& & 
	$y_{43}$& $(\alpha,\beta+1,\gamma-2)$&
	& \\
	\hline
	& & 
	$y_{44}$& $(\alpha,\beta,\gamma)$&
	& \\
	\hline
\end{tabular}

\label{TableWeights}
\end{table}
As explained in Section \ref{sec:sing}, the action of the subalgebra $N$ of $\phat$ on vectors $v_{ij},y_{ij}$ and $z_{ij}$ is as follows:
\[
(x_i \de_j)v_{hk}=\delta_{jh}v_{ik}+\delta_{jk}v_{ih},
\]
\[
(x_i \de_j)y_{hk}=\delta_{jh}y_{ik}-\delta_{ik}y_{hj},
\] 
\[
(x_i \de_j) z_{hk}=-\delta_{ih}z_{jk}-\delta_{ik}z_{hj},
\] 
where it is meant that if $i>j$ then $v_{ij}=v_{ji}$, $z_{ij}=-z_{ji}$, and $z_{ii}=0$ for all $i$. Besides,
\begin{align*}
b_{ij}v_{hk}&=0,\\
b_{ij}y_{hk}&=-\delta_{ik}v_{hj}-\delta_{jk}v_{ik},\\
b_{ij}z_{hk}&=-\delta_{ik}y_{jh}-\delta_{jk}y_{ih}+\delta_{ih}y_{jk}+\delta_{jh}y_{ik}.
\end{align*}
\begin{lemma}
	If $w$ is as in \eqref{12a} then $v_{ij}=0$ for all $i,j$.
\end{lemma}
\begin{proof}
	By Lemma \ref{barw} we have that 
	\[\bar w =\frac{1}{2}\sum_{i=1}^4 \de_i^2\otimes v_{ii}+\sum_{i<j}\de_i \de_j \otimes v_{ij}\]
	is a singular vector of degree 2 for $S(4)$. If $\bar w\neq 0$, this implies, by Theorem \ref{S4sing}, that $\bar w= \de_4(\de_1 v_{14}+\de_2 v_{24}+\de_3 v_{34}+\de_4v_{44})$ with all $v_{i4}\neq 0$ and $v_{14}$ is a $\slq$-highest weight vector of weight $(1,0,0)$. Therefore $v_{14}$ is a $\phat$ singular vector and therefore $w$ lies in $M_t(1,0,0)$. This is impossible because $v_{14}=-b_{24}y_{12}=-b_{24}b_{31}z_{32}=0$, since in $W_t(1,0,0)$ (the standard module) the operator $b_{24}b_{31}$ is identically zero as one can easily check.
\end{proof}
\begin{lemma}
	Let $w$ be as in \eqref{12a}. Then $y_{ij}=0$ for all $i<j$. Moreover $y_{42}=y_{43}=y_{32}=0$ and $y_{22}=y_{33}=y_{44}$.
\end{lemma}
\begin{proof}
	For $i<j$ we compute $(x_i^2\de_j)w= 2d_i\otimes y_{ij}$ since $v_{hk}=0$ for all $h,k$, and the first part follows. Considering the coefficient of $d_1$ in $(x_4^2 d_3)w$ we obtain the condition
	\[
	a_{34}y_{41}-2(x_4 \de_1)z_{12}=0.
	\]
	Applying $b_{11}$ to this equation and recalling that $y_{12}=0$ by the first part, we obtain $y_{42}=0$. The result follows since $y_{32}=(x_3 \de_4)y_{42}$ and $y_{43}=-(x_2 \de_3)y_{42}$. Finally, we have $(x_2\de_3)y_{32}=y_{22}-y_{33}$ hence $y_{22}=y_{33}$. Similarly $y_{33}=y_{44}$.
\end{proof}
\begin{lemma}\label{Ewgrado2} Let $w$ be as in \eqref{12a}. Then one of the following cases occurs:
\begin{itemize}
	\item [a)] $(\alpha,\beta,\gamma)=(1,0,0)$ and  $y_{22}=y_{33}=y_{44}=0$;
	\item [b)] $(\alpha,\beta,\gamma)=(0,0,0)$ and $y_{11}=y_{22}=y_{33}=y_{44}$; 
	\item [c)] $(\alpha,\beta,\gamma)\neq (0,0,0), (1,0,0)$ and  $y_{11}=y_{22}=y_{33}=y_{44}=0$.
\end{itemize}
\end{lemma}
\begin{proof}
Let us consider the condition $Ew=0$. We have: %Andiamo ora ad applicare $E$. Faccio un termine per volta sennò mi sbaglio solo quelli che contengono $\de_3,\de_4,d_1,d_2$, altrimenti è 0.
\begin{itemize}
	%\item $E\de_1 \de_3 v_{13}=\de_1(\frac{1}{2}b_{44})v_{13}=0$;
	%\item $E\de_1 \de_4 v_{14}=\de_1 (-\frac{1}{2}b_{34}+\frac{3}{2}a_{34})v_{14}=\de_1\frac{3}{2}a_{34}v_{14}$;
	%\item $E\de_2\de_3 v_{23}=\de_2(\frac{1}{2}b_{44})v_{23}=0$;
	%\item $E\de_2\de_4 v_{24}= \de_2 (-\frac{1}{2}b_{34}+\frac{3}{2}a_{34})v_{24}=\de_2 \frac{3}{2}a_{34}v_{24}$;
	%\item $E\frac{1}{2}\de_3^2v_{33}=\frac{1}{2}(\frac{1}{2}b_{44}\de_3 v_{33}+\frac{1}{2}\de_3b_{44}v_{33})=0$;
	%\item $E\de_3\de_4v_{34}=\frac{1}{2}b_{44}\de_4 v_{34}+\de_3(-\frac{1}{2}b_{34}+\frac{3}{2}a_{34})v_{34}=-d_4v_{34}+\de_3\frac{3}{2}a_{34}v_{34}$;
	%\item $E\frac{1}{2}\de_4^2 v_{44}=\frac{1}{2}((-\frac{1}{2}b_{34}+\frac{3}{2}a_{34})\de_4 v_{44}+\de_4 (-\frac{1}{2}b_{34}+\frac{3}{2}a_{34})v_{44})=d_3v_{44}+\frac{3}{2}\de_4 a_{34}v_{44}$
	\item $E\de_1d_1\otimes y_{11}=3\de_1 \otimes (x_4\de_2)y_{11}$;
	%\item $E\de_1 d_2 y_{12}=\de_1 (-3x_4 \de_1 y_{12})$;
	\item $E\de_2 d_1\otimes y_{21}=3\de_2 \otimes (x_4 \de_2)y_{21}$;
	\item $E\de_2 d_2\otimes y_{22}=-3\de_2\otimes (x_4\de_1)y_{22}$;
	\item $E\de_3 d_1\otimes y_{31}=3\de_3 \otimes (x_4\de_2)y_{31}$;
	%\item $E\de_3 d_2 y_{32}=\frac{1}{2}b_{44}d_2 y_{32}+\de_3 (-3x_4 \de_1) y_{32}=\de_3 (-3x_4 \de_1) y_{32}$;
	\item $E\de_3 d_3\otimes y_{33}=0$;
	%\item $E\de_3 d_4 y_{34}=\frac{1}{2}b_{44} d_4 y_{34}=-\frac{1}{2}d_4 b_{44}y_{34}=d_4 v_{34}$;
	\item $E\de_4 d_1 \otimes y_{41}=-3\de_2\otimes y_{41}+d_1 \otimes(- \frac{3}{2}a_{34}y_{41})+3\de_4 \otimes (x_4 \de_2) y_{41}$
	%\item $E\de_4 d_2 y_{42}=(-\frac{1}{2}b_{34}+\frac{3}{2}a_{34})d_2y_{42}+\de_4 (-3x_4\de_1)y_{42}=\de_1 3 y_{42}+d_2(-\frac{3}{2}a_{34}y_{42})+\de_4 (-3x_4\de_1y_{42})$;
	%\item $E\de_4 d_3 y_{43}=(-\frac{1}{2}b_{34}+\frac{3}{2}a_{34})d_3 y_{43}=-\frac{3}{2}d_3 a_{34}y_{43}$%-\frac{1}{2}d_3v_{44}$
	\item $E\de_4 d_4 \otimes y_{44}=-\frac{3}{2}d_4 \otimes a_{34}y_{44}$%-\frac{1}{2}d_4 v_{43}
	\item $Ed_1d_2z_{12}=d_4 \otimes 3 z_{12}+3d_2 \otimes (x_4 \de_2)\otimes z_{12}+3d_1 \otimes (x_4 \de_1)z_{12}$;
	\item $Ed_1d_3\otimes z_{13}=3d_3 \otimes (x_4\de_2) z_{13}$;
	\item $Ed_1d_4\otimes z_{14}=3d_4\otimes  (x_4 \de_2)z_{14}$;
	\item $Ed_2d_3\otimes z_{23}=-3d_3\otimes (x_4 \de_1)z_{23}$;
	\item $Ed_2d_4\otimes z_{24}=-3d_4\otimes (x_4\de_1)z_{24}$.
\end{itemize}
Looking at the coefficients of $\de_1, \de_2, d_1, d_4$, respectively, in $Ew$ we obtain the following four conditions
\begin{equation}\label{de1} (x_4\de_2)y_{11}=0 \end{equation}
\begin{equation}\label{de2} (x_4\de_2)y_{21}-(x_4\de_1)y_{22}-y_{41}=0 \end{equation}
%\begin{equation}\label{de3} x_4 \de_2 y_{31}-x_4 \de_1 y_{32}=0 \end{equation}
%\begin{equation}\label{de4} x_4 \de_2 y_{41}-x_4 \de_1 y_{42}=0 \end{equation}
\begin{equation}\label{d1} -\frac{1}{2}a_{34}y_{41}+(x_4\de_1)z_{12}=0 \end{equation}
%\begin{equation}\label{d2} x_4 \de_2 z_{12}=0 \end{equation}
%\begin{equation}\label{d3}	-\frac{1}{2}a_{34}y_{43}+x_4 \de_2 z_{13}-x_4 \de_1 z_{23}=0\end{equation}
\begin{equation}\label{d4} -\frac{1}{2}a_{34}y_{22}+z_{12}+(x_4\de_2) z_{14}-(x_4\de_1)z_{24}=0. \end{equation}
%
%Coefficiente di $\de_1$
%\begin{equation}\label{de1}
%	\frac{3}{2}a_{34}v_{14}+3x_4 \de_2 y_{11}-3x_4 \de_1 y_{12}+3y_{42}=0
%\end{equation}
%Coefficiente di $\de_2$
%\begin{equation}\label{de2}
%	\frac{3}{2}a_{34}v_{24}+3x_4 \de_2 y_{21}-3x_4 \de_1 y_{22}-3y_{41}=0
%\end{equation}
%Coefficiente di $\de_3$
%\begin{equation}\label{de3}
%	\frac{3}{2}a_{34}v_{34}+3x_4 \de_2 y_{31}-3x_4 \de_1 y_{32}=0
%\end{equation}
%Coefficiente di $\de_4$
%\begin{equation}\label{de4}
%	\frac{3}{2}a_{34}v_{44}+3x_4 \de_2 y_{41}-3x_4 \de_1 y_{42}=0
%\end{equation}
%Coefficiente di $d_1$
%\begin{equation}\label{d1}
%	-\frac{3}{2}a_{34}y_{41}+3x_4 \de_1 z_{12}=0
%\end{equation}
%Coefficiente di $d_2$
%\begin{equation}\label{d2}
%	-\frac{3}{2}a_{34}y_{42}+3x_4 \de_2 z_{12}=0
%\end{equation}
%Coefficiente di $d_3$
%\begin{equation}\label{d3}
%	\frac{1}{2}v_{44}-\frac{3}{2}a_{34}y_{43}+3x_4 \de_2 z_{13}-3x_4 \de_1 z_{23}=0
%\end{equation}
%Coefficiente di $d_4$
%\begin{equation}\label{d4}
%	-\frac{1}{2}v_{34}-\frac{3}{2}a_{34}y_{44}+3z_{12}+3x_4 \de_2 z_{14}-3x_4 \de_1 z_{24}=0.
%\end{equation}
%
%Le prime 4 equazioni, applicando $b_{12}$ diventano, per ogni $i=1,2,3,4$:
%\begin{equation}\label{compact}
%	h_{12}(v_{i4})-x_4 \de_2 v_{i2}+x_4 \de_1 v_{i1}+(-\delta_{i1}+\delta_{i2})v_{i4}=0
%\end{equation}
%Applico $x_1 \de_4$ a \eqref{compact} con $i=1$ e abbiamo $(2\alpha+\beta+\gamma+1)v_{11}=0$ per cui $v_{11}=0.$
%
%\medskip

If we apply $b_{12}$ to equation (\ref{d4}) we obtain:

%\[-h_{12}y_{22}+y_{22}-y_{11}-b_{14}z_{14}+x_4\de_2b_{12}z_{14}+b_{24}z_{24}-x_4\de_1b_{12}z_{24}=0\]
%\[-\alpha y_{44}+y_{22}-y_{11}-y_{44}+y_{11}+y_{44}-y_{22}=0\]
\begin{equation}\label{i}
\alpha y_{22}=0.
\end{equation}

%If  we apply $x_1\de_4$, we get:
%
%\[-\alpha y_{14}+x_1\de_2y_{24}-h_{14}y_{14}=0\]
%\[(-\alpha+1-\alpha-\beta-\gamma-2)y_{14}=0\]
%hence $y_{14}=0$, since $\alpha,\beta,\gamma$ are nonnegative integers.

Now if we apply $x_1\de_4$ to (\ref{de2}) we obtain:
\[(x_1\de_2)y_{21}-h_{14}y_{22}-y_{11}+y_{22}=-(\alpha+\beta+\gamma)y_{22}=0,\]
since $y_{24}=0$, i.e.\
\begin{equation}\label{ii}
(\beta+\gamma)y_{22}=0.
\end{equation}

If we apply $x_2\de_4$ to  (\ref{de1})  we get
\[h_{24}y_{11}=0\]
i.e.,
\begin{equation}
\label{yii4}
(\beta+\gamma)y_{11}=0.
\end{equation}

Now we apply $b_{12}$ to equation (\ref{d1}):

\[-h_{12}y_{41}-b_{24}z_{12}+(x_4\de_1)(y_{22}-y_{11})=0\]
\[-(\alpha-1)y_{41}+y_{41}+(x_4\de_1)(y_{22}-y_{11})=0.\]

If we apply $x_1\de_4$ we get:

\[-(\alpha-2)(y_{11}-y_{44})+h_{14}(y_{22}-y_{11})+(x_4\de_1)y_{14}=0\]

hence
\begin{equation}\label{yii1}
(1-\alpha)(y_{11}-y_{22})=0.
%(2-\alpha)(y_{11}-y_{22})+(\alpha+\beta+\gamma)(y_{22}-y_{11})=0
\end{equation}
since $y_{14}=0$.

%Now if we apply $x_2\de_4$, we get:
%\[(\alpha-2-\alpha-\beta-\gamma)y_{24}=0\]
%hence $y_{24}=0$.

%Now we apply $b_{12}$ to equation (\ref{d2}):
%
%\[-h_{12}y_{42}-b_{14}z_{12}+x_4\de_2(y_{22}-y_{11})=0\]
%\[-(\alpha+1)y_{42}-y_{42}+x_4\de_2(y_{22}-y_{11})=0.\]
%
%If we apply $x_2\de_4$ we get:
%
%\[-(\alpha+2)(y_{22}-y_{44})+h_{24}(y_{22}-y_{11})-x_4\de_2y_{24}=0\]
%
%hence
%\begin{equation}\label{yii2}
%-(2+\alpha)(y_{22}-y_{44})+(\beta+\gamma)(y_{22}-y_{11})=0
%\end{equation}
%since $y_{24}=0$. 
%
%Now we apply $x_1\de_4$ to equation (\ref{de1}):
%
%\[x_1\de_2y_{11}-x_4\de_2y_{14}-h_{14}y_{12}+y_{12}=0\]
%\[-y_{12}-x_4\de_2y_{14}-h_{14}y_{12}+y_{12}=0\]
%i.e. \[h_{14}y_{12}=(\alpha+\beta+\gamma+1)y_{12}=0\]
%since $y_{14}=0$. Therefore $y_{12}=0$.
%By applying $x_2\de_3$ we thus get $y_{13}=0$.
%
%
%
%Now we apply $x_1\de_4$ to equation (\ref{de3}):
%
%\[x_1\de_2y_{31}-x_4\de_2y_{34}-h_{14}y_{32}=0\]
%i.e. \[-y_{32}-(\alpha+\beta+\gamma)y_{32}=0\]
%since $y_{34}=0$. Therefore $y_{32}=0$.
%By applying $x_2\de_3$ we thus get $y_{22}=y_{33}$.
%
%Now we apply $b_{11}$ to equation (\ref{d1}) and we get:
%
%\[2x_1\de_2y_{41}-2b_{14}z_{12}+2x_4\de_1y_{12}=0\]
%i.e., $-y_{42}-y_{42}=0$ since $y_{12}=0$. It follows that $y_{42}=0$ and therefore, by applying $x_2\de_3$, $y_{43}=0$. Applying $x_3 \de_4$ we also obtain $y_{33}=y_{44}$.
%
%
%\color{black}Recalling that $y_{22}=y_{33}=y_{44}$, this relation, together with (\ref{yii1}) and (\ref{yii2})  and (\ref{yii4}) provides:
We have thus obtained the following conditions:
\[
\begin{cases}
(1-\alpha)(y_{11}-y_{22})=0\\
(\beta+\gamma)y_{22}=0\\
(\beta+\gamma)y_{11}=0\\
\alpha y_{22}=0.
\end{cases}
\]
The result follows.
\end{proof}

\begin{theorem}\label{listagrado2}
	The following is, up to multiplication by scalars, a complete list of $E(4,4)$-singular vectors  of degree 2 in Verma modules $M_t(a,b,c)$:
\begin{itemize}
		\item $w{[2_{DA}]}=4\sum_{i=2}^4 d_1\de_i\otimes a_{1i}-2d_1(d_4\otimes a_{12}a_{13}+d_3\otimes a_{14}a_{12}+d_2\otimes a_{13}a_{14})\in M_t(0,1,0)$ where for $t\neq 0$ it is meant that $M_t(0,1,0)\cong M_t(0,0,0)$;
	\item $w{[2_{EA}]}=d_1w[1_E]\in M_1(1,0,0)$
where $w[1_E]$ is as in Theorem \ref{listagrado1}.	
\end{itemize}
\end{theorem}
\begin{proof}
Let us first assume $y_{11}\neq 0$, so that we are necessarily in case a) or b) of Lemma \ref{Ewgrado2}.
In case a)  there is a unique (up to a scalar factor) vector $w$ with $y_{11}\neq 0$ that satisfies $Ew=0$ and it is 
\[
w=\de_1 d_1 \otimes d_1+\de_2 d_1\otimes d_2+\de_3d_1\otimes d_3+\de_4d_1\otimes d_4+d_1d_2\otimes \de_2+d_1d_3\otimes \de_3+d_1d_4\otimes\de_4.
\]
This vector is indeed singular if and only if $t=1$ and in this case it is $d_1w[1_E]$.
In case b) of Lemma \ref{Ewgrado2}, $y_{11}\neq 0$ implies  $y_{21}=0$ (otherwise it would be a highest weight vector with non dominant weight). Similarly also $y_{31}=0$ and $y_{41}=0$. Under these conditions there is  a unique (up to a scalar factor) such vector that satisfies $Ew=0$. This vector is
\[
w=(\de_1 d_1+\de_2 d_2+\de_3 d_3+\de_4 d_4)\otimes 1
\]
if $t=0$ and
\[
w=2(\de_1 d_1+\de_2 d_2+\de_3 d_3+\de_4 d_4)1-d_1d_2 a_{34}+d_1d_3a_{24}-d_1d_3a_{23}-d_2d_3a_{14}+d_2d_4a_{13}-d_3d_4a_{12}
\]
if $t\neq 0$. Such vectors are not singular as they do not satisfy $(x_4C)w=0$.

 So we can assume $y_{11}=0$, i.e., $y_{ii}=0$ by Lemma \ref{Ewgrado2}.
Hence we have:

\begin{align*}
	w&=\de_2 d_1 y_{21}+	\de_3 d_1 y_{31}+	\de_4 d_1 y_{41}\\
	&+d_1 d_2 z_{12}+ d_1 d_3 z_{13}+d_1 d_4 z_{14}+d_2d_3 z_{23}+d_2d_4z_{24}+d_3d_4z_{34}.
\end{align*}
Let us compute 

\begin{itemize}
\item $4(x_4C)\de_2 d_1 \otimes y_{21}=
-4d_1\otimes (x_4\de_2)y_{21}+5\de_2\otimes a_{14}y_{21}$
\item $4(x_4C)\de_3 d_1 \otimes y_{31}=
-4d_1\otimes (x_4\de_3)y_{31}+5\de_3\otimes a_{14}y_{31}$

\item $4(x_4C)\de_4 d_1 \otimes y_{41}=(\alpha+2\beta+3\gamma-5t+2)d_1\otimes y_{41}+5\de_4\otimes a_{14}y_{41}$

\item $4(x_4C)d_1d_2\otimes z_{12}=
d_2\otimes y_{42}-10\de_3\otimes z_{12}-5d_2\otimes a_{14}z_{12}+d_1\otimes y_{41}+5d_1\otimes a_{24}z_{12}$

\item $4(x_4C)d_1d_3\otimes z_{13}=
d_3\otimes y_{43}+10\de_2\otimes z_{13}-5d_3\otimes a_{14}z_{13}+d_1\otimes y_{41}+5d_1\otimes a_{34}z_{13}$

\item $4(x_4C)d_1d_4\otimes z_{14}=
d_4\otimes (y_{44}-y_{11})-5d_4\otimes a_{14}z_{14}+2d_1\otimes y_{41}$

\item $4(x_4C)d_2d_3\otimes z_{23}=
d_3\otimes y_{43}-10\de_1\otimes z_{23}-5d_3\otimes a_{24}z_{23}+d_2\otimes y_{42}+5d_2\otimes a_{34}z_{23}$

\item $4(x_4C)d_2d_4\otimes z_{24}=
d_4\otimes (y_{44}-y_{22})-5d_4\otimes a_{24}z_{24}+2d_2\otimes y_{42}$

\item $4(x_4C)d_3d_4\otimes z_{34}=
d_4\otimes (y_{44}-y_{33})-5d_4\otimes a_{34}z_{34}+2d_3\otimes y_{43}$
\end{itemize}
Looking at the coefficient of $\de_1$ in $(x_4C)w$ one immediately gets $z_{23}=0$ and applying $x_3 \de_4$ and $x_2 \de_3$ we also have $z_{24}=z_{34}=0$.

Coefficients of $\de_2,\de_3,\de_4,d_1,d_2,d_3,d_4$ in $(x_4C)w$ provide the following equations

\begin{equation}\label{Cde2}
	a_{14}y_{21}+2z_{13}=0
\end{equation}
\begin{equation}\label{Cde3}
	a_{14}y_{31}-2z_{12}=0,
\end{equation}
\begin{equation}\label{Cde4}
	a_{14}y_{41}=0,
\end{equation}
\begin{equation}\label{Cd1}
	-4(x_4\de_2)y_{21}-4(x_4\de_3)y_{31}+(\alpha+2\beta+3\gamma-5t+6)y_{41}+5a_{24}z_{12}+5a_{34}z_{13}=0,
\end{equation}
\begin{equation}\label{Cd2}
	a_{14}z_{12}=0,
\end{equation}
\begin{equation}\label{Cd3}
	a_{14}z_{13}=0,
\end{equation}
\begin{equation}\label{Cd4}
	a_{14}z_{14}=0.
\end{equation}

\bigskip
Applying $b_{23}$ to \eqref{Cde2} we obtain 
\begin{equation}\label{1}
\beta y_{21}=0.
\end{equation}

Applying $b_{23}$ to \eqref{Cd4} we obtain
$h_{23}(z_{14})=\beta z_{14}=0$ and similarly using \eqref{Cde4} we obtain 
\begin{equation}\label{2}
\beta y_{41}=0.
\end{equation}

Now apply $x_2\de_4$ to \eqref{Cd1} to obtain
\[
-4(\gamma+1)y_{21}-4 y_{21}+(\alpha+3\gamma-5t+6)y_{21}-5a_{24}z_{14}+5a_{32}z_{13}=0
\]
\[
(-\gamma-2+\alpha-5t)y_{21}-5a_{24}z_{14}+5a_{32}z_{13}=0.
\]
Applying $b_{13}$ we have
\[
10h_{13}z_{14}+10x_3\de_4 z_{13}=0,
\]
i.e.
\begin{equation}\label{3}
(\alpha-2)z_{14}=0.
\end{equation}
Applying $x_2\de_4$ to \eqref{de2} we deduce 
\begin{equation}\label{4}
\gamma y_{21}=0.
\end{equation}
From  formulas \eqref{1}--\eqref{4} we deduce that if $y_{21}\neq 0$, then $(\alpha, \beta,\gamma)=(2,0,0)$ (noting that $y_{21}=-b_{24}z_{14})$. This case indeed provides the singular vector $w_{1,7}$ (using that $y_{21}$ is a $\phat$-singular vector of weight $(0,1,0)$ and that $z_{12}$ has weight $(2,-1,0)$). 

If $y_{21}=0$, then we have necessarily $y_{31}=0$. Indeed if $y_{31}\neq 0$, then it is a highest weight vector of weight $(\alpha-1,\beta-1, \gamma+1)$ (see Table \ref{TableWeights}). But $y_{31}\neq 0$ implies  $y_{41}\neq 0$, hence $\beta=0$ but this is a contradiction since the weight of $y_{31}$ is not dominant.
If $y_{21}=y_{31}=0$, then $y_{41}= 0$ by \eqref{de2}. 

Therefore we can assume that $y_{ij}=0$ for all $i,j$ and also $z_{23}=z_{24}=z_{34}=0$. Equation \eqref{Cde3} provides $z_{12}=0$, and hence applying $x_2\de_3$ and $x_2 \de_4$ we conclude that also $z_{13}=z_{14}=0$. 
\end{proof}

\begin{remark} The singular vector $w[2_{DA}]$ defines the following morphism of $E(4,4)$-modules for every $t\in\C$:
\[\varphi_{[2_{DA}]}: M_{t-2}{(2,0,0)}\rightarrow M_t(0,1,0).\]
Indeed if $t\neq 0,2$ we have $\varphi_{[2_{DA}]}=\varphi_{[1_D]}\circ\varphi_{[1_A]}$.

If $t=0$ the vector $w[2_{DA}]$ induces a morphism
\[\tilde{\varphi}_{[2_{DA}]}: \tilde{M}_{-2}{(2,0,0)}\rightarrow M_0(0,1,0).\]
An explicit computation shows that $a_{12}w[2_{DA}]=0$ in $M_0(0,1,0)$ hence $\tilde{\varphi}_{[2_{DA}]}$ induces the morphism 
\[M_{-2}{(2,0,0)}\rightarrow M_0(0,1,0).\]
By duality we also get the morphism
\[M_0(2,0,0) \rightarrow M_2(0,1,0).\]
\end{remark}

\section{Singular vectors of degree 3}\label{degree3}

In this section we classify $E(4,4)$ singular vectors of degree 3 in Verma modules $M_t(a,b,c)$. So let $w$ be such a singular vector of weight $(\alpha, \beta, \gamma)$ and note that, by Lemma \ref{barw}, it can be written as follows:
\begin{equation}
w=\sum_{i\leq j,k}\frac{1}{m_{ij}}\de_i\de_jd_k\otimes v_{ij,k}+\sum_{i,j<k}\de_id_jd_k\otimes y_{i,jk}+\sum_{i<j<k}d_id_jd_k\otimes z_{ijk},\label{31a}
\end{equation}
where
$v_{ij,k}=v_{ji,k}$ and $m_{ij}=1$ if $i\neq j$, $m_{ii}=2$, $y_{i,jk}=-y_{i,kj}$, and $z_{ijk}$ is alternating with respect to $i,j,k$. By \eqref{vIJaction} we have that by applying the elements $x_i\de_{j}$ with $i<j$ and the elements $b_{ij}$ we get the following relations:
\begin{align*}
(x_i\de_j)v_{rs,u}&=\delta_{jr}v_{is,u}+\delta_{js}v_{ri,u}-\delta_{iu}v_{rs,j};\\
(x_i\de_j)y_{r,su}&=\delta_{jr}y_{i,su}-\delta_{is}y_{r,ju}-\delta_{iu}y_{r,sj};\\
(x_i\de_j)z_{rsu}&=-\delta_{ir}z_{jsu}-\delta_{is}z_{rju}-\delta_{iu}z_{rsj};\\
b_{ij}v_{rs,u}&=0;\\
b_{ij}y_{r,hk}&=-\delta_{ik}v_{rj,h}-\delta_{jk}v_{ri,h}+\delta_{ih}v_{rj,k}+\delta_{jh}v_{ri,k};\\
b_{ij}z_{rhk}&=-\delta_{ir}y_{j,hk}-\delta_{jr}y_{i,hk}-\delta_{ik}y_{j,rh}-\delta_{jk}y_{i,rh}+\delta_{ih}y_{j,rk}+\delta_{jh}y_{i,rk}.
\end{align*}
\begin{lemma}\label{lemmadeg3} \begin{itemize}
\item[(a)]	Let $i,j,k$ be distinct with $k=\max\{i,j,k\}$. Then
	\[
	y_{i,jk}=0 \textrm{ and }y_{i,ik}=y_{j,jk}.\]
	\item[(b)] Let $r,s,u$ be such that $u\neq r$ and $u\neq s$ with $(r,s)\neq (4,4)$. Then
	\[
	v_{rs,u}=0.\]
\item[(c)]	Let $\{i,j,h,k\}=\{1,2,3,4\}$ with $i\neq 4$. Then
	\[\frac{1}{2}v_{ii,i}=v_{ij,j}=v_{ih,h}=v_{ik,k}.
	\]
	\end{itemize}
\end{lemma}
\begin{proof}
We apply the vector field $P=x_i^2\de_j$, with $i<j$, to  the summands of $w$, and obtain: for $k\neq j$, $h\neq i$, $k'\neq k,i,j$,
\begin{align*}
P \frac{1}{2}\de_i^2d_k\otimes v_{ii,k}&=\de_jd_k\otimes v_{ii,k}+2\delta_{ik}\de_id_k\otimes v_{ii,j},\\
P \frac{1}{2}\de_i^2d_j\otimes v_{ii,j}&=\de_jd_j\otimes v_{ii,j}-2\de_id_i\otimes v_{ii,j},\\
P \de_h\de_id_k\otimes v_{hi,k}&=-2\de_hd_k\otimes (\delta_{jh}v_{ii,k}-\delta_{ik}v_{hi,j}),\\
P \de_h\de_id_j\otimes v_{hi,j}&=-2\de_h(d_i\otimes v_{hi,j}+\delta_{jh}d_j\otimes v_{ii,j}),\\
P \de_id_kd_{k'}\otimes y_{i,kk'}&=2\delta_{ik}d_kd_{k'}\otimes y_{i,jk'},\\
P \de_id_kd_j\otimes y_{i,kj}&=-2d_kd_i\otimes y_{i,kj}+2\delta_{ik}\de_id_k\otimes v_{ii,j},\\
P \de_hd_kd_j\otimes y_{h,kj}&=2\delta_{ik}\de_hd_k\otimes v_{hi,j},\\
P d_kd_{k'}d_j\otimes z_{kk'j}&=-2\delta_{ik}d_kd_{k'}\otimes y_{i,k'j},
\end{align*}
and we consider the coefficient of $d_id_\ell$ with $\ell\neq i,j$. There are exactly 3 summands which contribute ($\de_i d_i d_\ell y_{i,i\ell}$, $ d_id_jd_\ell z_{ij\ell}$ and $\de_id_jd_\ell y_{i,j\ell}$) and we can deduce $y_{i,j\ell}=0$. By inverting the roles of $j$ and $\ell$, if necessary, we obtain $y_{i,jk}=0$. 

If $j<i<k$ we can apply $x_j \de_i$ to $y_{i,jk}=0$ to obtain $y_{i,ik}=y_{j,jk}$.
Applying $b_{jj},b_{kk},b_{ij},b_{ik}$ to $y_{i,jk}=0$ we obtain $v_{ij,k}=v_{ik,j}=v_{ii,k}=v_{ii,j}=0$.

If $i\neq 4$, $k>i$ and $j<k$ we can apply $b_{jk}$ to $y_{i,jk}=0$ to obtain $v_{ik,k}=v_{ij,j}$. Finally, for $h<i<4$, applying $x_h \de_i$ to $v_{ii,h}=0$ we have $2v_{ih,h}=v_{ii,i}$.

\end{proof}
\begin{lemma}\label{lemmadeg3bis}
	We have 
	\[
	v_{44,1}=v_{44,2}=v_{44,3}=0,
	\]
and 
	\[\frac{1}{2}v_{44,4}=v_{41,1}=v_{42,2}=v_{43,3}.
	\]
If $v_{14,4}= 0$, then  $v_{ij,k}=0$ for all $i,j,k$. If $v_{14,4}\neq 0$, then we have $w\in M_t(1,0,0)$ and, up to a scalar factor,
\[
w=\sum_{i, j} \de_i \de_j d_i\otimes d_j+\de_i d_i d_j \otimes \de_j+\rho(d_1d_2d_3\otimes d_4-d_1d_2d_4\otimes d_3+d_1d_3d_4\otimes d_2-d_2d_3d_4\otimes d_1)
\]
for some scalar $\rho$.
\end{lemma}
\begin{proof}
We apply $x_4^2 \de_1$ to $w$ and we consider the coefficient of $\de_1 d_r$, with $r\neq 1,4$. There is only one term which contributes  which is $\frac{1}{2}\de_4^2d_r\otimes v_{44,r}$ and it provides $v_{44,r}=0$. We can similarly obtain $v_{44,1}=0$ by considering the action of $x_4^2\de_2$ and the coefficient of $\de_2 d_1$. By application of $x_1\de_4, x_2 \de_4, x_3 \de_4$ to  $v_{44,1}$, $v_{44,2}$, $v_{44,3}$, respectively, we obtain
\[\frac{1}{2}v_{44,4}=v_{41,1}=v_{42,2}=v_{43,3}.
\]
Note that the space spanned by all $v_{ij,k}$ has dimension at most 4 and it is spanned by $v_{11,1}, v_{21,1},v_{31,1},v_{41,1}$.
Now we consider the coefficient of $\de_1 d_4$ in $(x_4^2 \de_1)w$: this provides $(x_4 \de_1)v_{14,4}=v_{41,1}$. Therefore, if $v_{14,4}=0$ then $v_{41,1}=0$ and we obtain $v_{r1,1}=0$ by applying $x_r\de_4$. 

So we can assume that $v_{14,4}\neq 0$ and so it is a $\phat$-singular vector. Apply $x_1 \de_4$ to $(x_4 \de_1)v_{14,4}-v_{41,1}=0$; we obtain
\[
h_{14}v_{14,4}-v_{11,1}+v_{14,4}=(\alpha+\beta+\gamma+1)v_{14,4}-v_{14,4}=(\alpha+\beta+\gamma)v_{14,4}=0,
\]
since $v_{11,4}=0$ and $v_{11,1}=2v_{14,4}$ and $v_{14,4}$ has weight $(\alpha,\beta,\gamma+1)$. Therefore we have $\alpha=\beta=\gamma=0$ and $v_{14,4}$ has weight $(1,0,0)$, i.e. we are in the standard module.

In particular, up to a nonzero scalar factor, $v_{14,4}=d_1$. It follows that 
	\[\frac{1}{2}v_{11,1}=v_{12,2}=v_{13,3}=v_{14,4}=d_1,
\] 
and the compatibility with the action of $x_i\de_j$ and $b_{ij}$ forces
\begin{equation}\label{43} 
w=\sum_{i, j} \de_i \de_j d_i\otimes d_j+\de_i d_i d_j \otimes \de_j+\rho(d_1d_2d_3\otimes d_4-d_1d_2d_4\otimes d_3+d_1d_3d_4\otimes d_2-d_2d_3d_4\otimes d_1).\end{equation}
\end{proof}
\begin{proposition}
	The vector $w\in M_t(1,0,0)$ given in  \eqref{43}
	%\[w=\sum_{i, j} \de_i \de_j d_i\otimes d_j+\de_i d_i d_j \otimes \de_j+\rho(d_1d_2d_3\otimes d_4-d_1d_2d_4\otimes d_3+d_1d_3d_4\otimes d_2-d_2d_3d_4\otimes d_1)\in M_t(1,0,0)\]
is a singular vector if and only if $\rho=1$ and $t=3$.
\end{proposition}
\begin{proof} If we apply  $2x_4^2 \de_4$ to \eqref{43}, we obtain
	\begin{align*}
	2(x_4^2 \de_4)w=&(t-2\rho- 1) (\de_1 d_1+\de_2 d_2+\de_3 d_3)\otimes d_4+(2\rho-2)d_4(\de_1 \otimes d_1+\de_2 \otimes d_2 +\de_3\otimes d_3)\\
	&+(t-3)\de_4 d_4\otimes d_4+(4\rho-t-1)d_4(d_3\otimes \de_3+d_2\otimes \de_2+d_1\otimes \de_1).
	\end{align*}
It follows that we necessarily have $\rho=1$ and $t=3$.	
An explicit computation shows that under these hypotheses $w$ is indeed singular.
\end{proof}

Now we can assume that all $v_{ij,k}=0$ and that the $y_{i,jk}$ satisfy the conditions in Lemma \ref{lemmadeg3}. The action of $x_4^2 \de_1$ on $w$ gives us further constraints.
\begin{lemma}\label{lemmadeg3ter}
Let $v_{ij,k}=0$ for all $i,j,k$ and $\alpha+\beta+\gamma\neq 0,1$. Then $y_{i,ij}=0$ for all $j\neq 1$.
\end{lemma}
\begin{proof}
Let us compute $(x_4^2\de_1)w$. We have
\begin{align*}
(x_4^2 \de_1)w=&2\de_1d_4\otimes y_{14,1}
-2d_1d_2 \otimes (x_4\de_1)y_{4,12}
 -2 d_1d_3\otimes (x_4\de_1)y_{4,13}\\
 & 
   -2 d_1d_4\otimes (x_4\de_1)y_{4,14} 
 -2d_2d_3\otimes (x_4\de_1)y_{4,23} 
 -2 d_2d_4\otimes (x_4\de_1)y_{4,24} 
 -2 d_3d_4\otimes (x_4\de_1)y_{4,34}.
\end{align*}

The coefficient of $d_3d_4$ in $(x_4^2 \de_1)w$ provides $(x_4 \de_1) y_{4,43}=0$ and applying $x_1 \de_4$ we have 
\[(\alpha+\beta+\gamma)y_{4,43}=0,\] since $y_{4,43}$ has weight $(\alpha,\beta+1,\gamma-1)$ and $y_{1,34}=0$ by Lemma \ref{lemmadeg3}.

The coefficient of $d_2d_4 $ in $(x_4^2 \de_1)w$ provides $(x_4 \de_1) y_{4,42}=0$ and applying $x_1 \de_4$ we have 
\[(\alpha+\beta+\gamma)y_{4,42}=0,\] since $y_{4,42}$ has weight $(\alpha+1,\beta-1,\gamma)$ and $y_{1,24}=0$ by Lemma \ref{lemmadeg3}.

The coefficient of $d_2d_3 $ in $(x_4^2 \de_1)w$ provides $(x_4 \de_1) y_{4,23}=0$ and applying $x_1 \de_4$ we have 
\[(\alpha+\beta+\gamma-1)y_{4,23}=0,\] since $y_{4,23}$ has weight $(\alpha+1,\beta,\gamma-2)$ and $y_{1,23}=0$ by Lemma \ref{lemmadeg3}. Applying $x_3 \de_4$ to the last equation we have
\[
0=(\alpha+\beta+\gamma-1)(y_{3,32}-y_{4,42})=(\alpha+\beta+\gamma-1)y_{3,32}+y_{4,42}.
\]

The coefficient of $d_1d_3 $ in $(x_4^2 \de_1) w$ provides $(x_4 \de_1) y_{4,13}=0$ and applying $x_1 \de_4$ twice we have 
\[(\alpha+\beta+\gamma-1)y_{1,13}=y_{4,43},\] since $y_{4,13}$ has weight $(\alpha-1,\beta+1,\gamma-2)$, $y_{1,34}=0$ by Lemma \ref{lemmadeg3}, and $y_{1,13}$ has weight $(\alpha,\beta+1,\gamma-1)$. 
Recall also that $y_{2,23}=y_{1,13}$ by Lemma \ref{lemmadeg3}.

The coefficient of $d_1d_4 $ in $(x_4^2 \de_1)w$ provides $(x_4 \de_1) y_{4,41}=0$, and applying $x_1 \de_4$ twice we have 
\[(\alpha+\beta+\gamma)y_{1,14}=0,\] since $y_{4,41}$ has weight $(\alpha-1,\beta,\gamma)$,  $y_{1,14}$ has weight $(\alpha,\beta,\gamma+1)$. 
By Lemma \ref{lemmadeg3} we have $y_{1,14}=y_{2,24}=y_{3,34}$.

Finally, the coefficient of $d_1d_2 $ in $(x_4^2 \de_1)w$ provides $(x_4 \de_1)y_{4,12}=0$ and applying $x_1 \de_4$ twice we have 
\[(\alpha+\beta+\gamma-1)y_{1,12}=y_{4,42},\] since $y_{4,12}$ has weight $(\alpha,\beta-1,\gamma-1)$, $y_{1,12}-y_{4,42}$ has weight $(\alpha+1,\beta-1,\gamma)$ and $y_{1,24}=0$ by Lemma \ref{lemmadeg3}. Using all these equations the result follows.
\end{proof}
\begin{lemma}\label{lemmadeg3quater} Assume that $v_{ij,k}=0$ in \eqref{31a}
for all $i,j,k$ and $\alpha+\beta+\gamma\neq 0,1$. Then $w=0$. 
\end{lemma}
\begin{proof}
	We consider the coefficient of $\de_1 d_1$ in $Ew$: since, by Lemma \ref{lemmadeg3ter}, $y_{1,12}=0$,  it provides $y_{4,12}=0$, and hence also $y_{4,13}=y_{4,23}=y_{3,12}=0$ by applying suitable $x_i \de_j$ with $i<j$. Therefore, by Lemma \ref{lemmadeg3ter}, we have
	\begin{align*}	
	w&=\de_2 d_2d_1\otimes y_{2,21}+\de_3d_3d_1\otimes y_{3,31}+\de_4d_4d_1\otimes y_{4,41}\\
	&+d_1d_2d_3\otimes z_{123}+d_1d_2d_4\otimes z_{124}+d_1d_3d_4\otimes z_{134}+d_2d_3d_4\otimes z_{234}.
	\end{align*}
	The coefficient of $\de_2 d_1$ in $Ew$ provides $(x_4 \de_1)y_{2,21}=0$ and applying $x_1 \de_4$ one gets $(\alpha+\beta+\gamma-1)y_{2,21}=0$ (since $y_{2,24}=0$) and hence $y_{2,21}=0$. 
	
	The coefficient of $\de_2 d_4$ provides $y_{4,41}=y_{2,21}$ and hence $y_{4,41}=0$.
	%The coefficient of $\de_3 d_3$ provides $x_4 \de_2 y_{3,31}=0$ which implies $(\beta+\gamma)y_{3,31}=0$.
	
	The coefficient of $d_3d_4$ provides $-z_{123}+(x_4\de_2)z_{134}-(x_4 \de_1)z_{234}=0$. Applying $x_1 \de_4$ we obtain $(\alpha+\beta+\gamma+1)z_{234}=0$ so that $z_{234}=0$.
	
	%Applying $x_2 \de_4$ to the same equation $-z_{123}+x_4\de_2 z_{134}-x_4 \de_1 z_{234}=0$ we obtain $(\beta+\gamma)z_{134}=0$, since the weight of $z_{134}$ is $(\alpha-1,\beta+1,\gamma)$.
	
	The coefficient of $d_1d_3$ provides $(x_4 \de_1)z_{123}=0$ and applying $x_1 \de_4$ we obtain
	$(\alpha+\beta+\gamma-1)z_{123}=0$ hence $z_{123}=0$, and hence all $z_{ijk}=0$ and also $y_{3,31}=-b_{23}z_{123}=0$.
\end{proof}

\begin{lemma} If $v_{ijk}=0$ in \eqref{31a} for every $i,j,k$, then $\alpha=1$, $\beta=\gamma=0$, $t=0$, and 
$w=(\sum_{i=1}^4\de_id_i)d_11$ is, up to a scalar factor, the unique $E(4,4)$-singular vector of degree 3.
%.weight $(1,0,0)$ in $M_0(0,0,0)$.
\end{lemma}
\begin{proof}
If $\alpha+\beta+\gamma=0$, by weight reasons which are evident in Figure \ref{weights},  we have $y_{i,ij}=y_{k,kj}$ for every $j, k$ and $y_{i,jk}=0$ for every $i,j,k$ distinct. 
We thus have
\[w= \sum_{i\neq j}\de_id_id_jy_{i,ij}+\sum_{i<j<k}d_id_jd_kz_{ijk}\]
with $y_{i,ij}=y_{k,kj}$. If $y_{1,14}$ is not zero, then it is a $\phat$-singular vector of weight $(0,0,1)$. We then compute $Ew=0$ and consider the coefficient of $\de_1d_4$. We obtain:
\[(x_4\de_2)y_{1,14}=0,\]
which implies $y_{1,14}=0$ by applying $x_2\de_4$.
Again by weight reasons $y_{1,14}=0$ implies $y_{i,jk}=0$ for every $i,j,k$. 

Now suppose $\alpha+\beta+\gamma=1$. We compute $Ew=0$.
The coefficient of $\de_1d_4$ gives $y_{1,12}=0$ which implies that $y_{1,13}=0$ and $y_{1,14}=0$.
Now the coefficient of $\de_1d_1$ in $Ew=0$ gives $y_{4,12}=0$. It follows, by Table \ref{weights}, that 
$y_{2,21}$ is a $\phat$-singular vector of weight $(0,0,0)$ and that $y_{2,21}=y_{3,31}=y_{4,41}$.
Now the coefficient of $d_1d_4$ in $Ew=0$ gives $a_{34}y_{2,21}=0$ which implies $t=0$.
The vector
\[w=(\sum_{i=1}^4\de_id_i)d_11\]
is indeed a singular vector of degree 3 and weight $(1,0,0)$ in $M_0(0,0,0)$.

We are left to show that there are no $E(4,4)$-singular vectors of degree 3 such that $y_{i,jk}=0$ for every $i,j,k$.
Suppose that $w$ is such vector. Then,
as in the proof of Lemma \ref{lemmadeg3quater}, the coefficient of $d_3d_4$ in $Ew=0$ provides equation
\begin{equation}\label{zeta} 
-z_{123}+x_4\de_2 z_{134}-x_4\de_1 z_{234}=0,
\end{equation}
from which we deduce that  $z_{234}=0$ by applying $x_1\de_4$.
It follows by weight reasons that if $\alpha+\beta+\gamma=0$ then $z_{ijk}=0$ for every $i,j,k$. Now let $\alpha+\beta+\gamma=1$.
If $z_{134}$ is not zero then it is a $\phat$-singular vector of weight $(1,0,0)$.
By applying $x_2\de_4$ to Equation \eqref{zeta}, we obtain
$z_{134}=0$, hence $z_{123}=0$. The result follows.   
\end{proof}
We summarize the result of this section in the following theorem.
\begin{theorem}\label{listagrado3}
A complete list of $E(4,4)$-singular vectors of degree 3 in Verma modules $M_t(a,b,c)$ is, up to a scalar factor, as follows:
\begin{itemize}
\item[i)] $w[3_F]=\sum_{i, j} \de_i \de_j d_i\otimes d_j+\de_i d_i d_j \otimes \de_j+d_1d_2d_3\otimes d_4-d_1d_2d_4\otimes d_3+d_1d_3d_4\otimes d_2-d_2d_3d_4\otimes d_1\in M_3(1,0,0)$;
\item[ii)] $w[3_G]=(\sum_{i=1}^4\de_id_i)d_1\otimes 1\in M_0(0,0,0)$.
\end{itemize}
\end{theorem}

\begin{remark} The singular vector $w[3_G]=(\sum_{i=1}^4\de_id_i)d_1\otimes 1\in M_0(0,0,0)$ provides a morphism
	\[
	K_{-3}(1,0,0)\rightarrow M_0(0,0,0).
	\]
	sending $x_1$ to $w$. 
Recalling that $W_{-3}(1,0,0)\cong K_{-3}(1,0,0)/M$, where $M$ is the submodule generated by $a_{12}x_1$, and observing that $a_{12}w=0$, we conclude that the above morphism factors through $W_{-3}(1,0,0)$ hence we have a morphism
\[
\varphi_{[3_G]}:	M_{-3}(1,0,0)\rightarrow M_0(0,0,0).
\]
By duality the singular vector \[w[3_F]=\sum_{i, j} \de_i \de_j d_i\otimes d_j+\de_i d_i d_j \otimes \de_j+d_1d_2d_3\otimes d_4-d_1d_2d_4\otimes d_3+d_1d_3d_4\otimes d_2-d_2d_3d_4\otimes d_1\] in $M_3(1,0,0)$ induces a morphism
\[
\varphi_{[3_F]}: M_0(0,0,0)\rightarrow 	M_{3}(1,0,0).
\]
\end{remark}

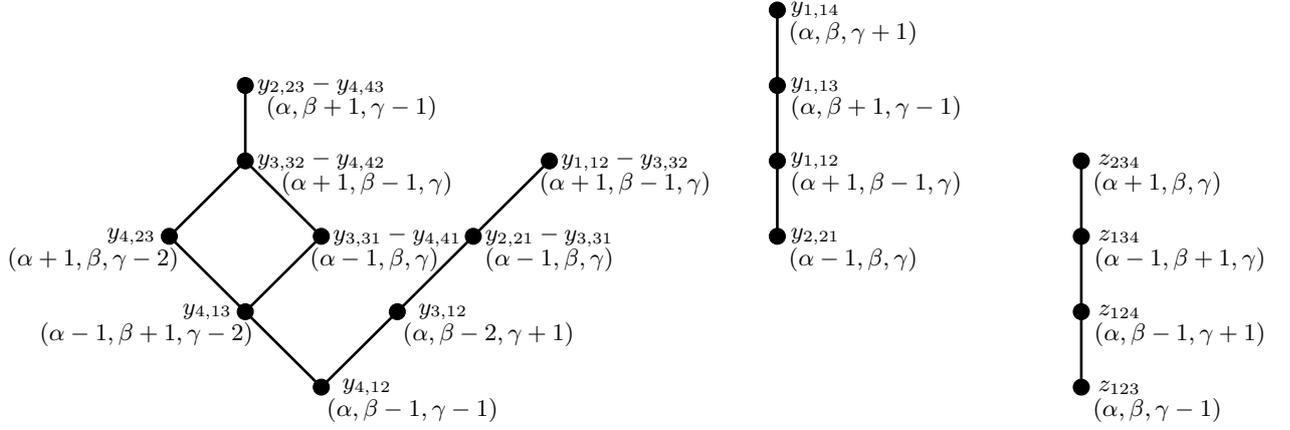
\begin{figure}[h]
	\begin{center}
		\scalebox{1}{	$$
			\begin{tikzpicture} 
				
				\draw[line width=1pt](0,0)--(3,3);
				\draw[line width=1pt](0,0)--(-2,2);
				\draw[line width=1pt](-1,1)--(0,2);
				\draw[line width=1pt](-2,2)--(-1,3);
				\draw[line width=1pt](0,2)--(-1,3);
				\draw[line width=1pt](-1,3)--(-1,4);
				\draw[line width=1pt](6,2)--(6,5);
				\draw[line width=1pt](10,0)--(10,3);
				\draw[fill=black]{(0,0) circle(3pt)};
				\draw[fill=black]{(1,1) circle(3pt)};
				\draw[fill=black]{(2,2) circle(3pt)};
				\draw[fill=black]{(3,3) circle(3pt)};
				\draw[fill=black]{(-1,1) circle(3pt)};
				\draw[fill=black]{(-2,2) circle(3pt)};
				\draw[fill=black]{(0,2) circle(3pt)};
				\draw[fill=black]{(-1,3) circle(3pt)};
				\draw[fill=black]{(-1,4) circle(3pt)};
				\draw[fill=black]{(6,2) circle(3pt)};
				\draw[fill=black]{(6,3) circle(3pt)};
				\draw[fill=black]{(6,4) circle(3pt)};
				\draw[fill=black]{(6,5) circle(3pt)};
				\draw[fill=black]{(10,0) circle(3pt)};
				\draw[fill=black]{(10,1) circle(3pt)};
				\draw[fill=black]{(10,2) circle(3pt)};
				\draw[fill=black]{(10,3) circle(3pt)};
				\node at (0.6,0) {\scriptsize $y_{4,12}$};
				\node at (1.2,-0.3) {\scriptsize $(\alpha,\beta-1,\gamma-1)$};
				\node at (1.6,1) {\scriptsize$y_{3,12}$};
				\node at (2.2,0.7) {\scriptsize$(\alpha,\beta-2,\gamma+1)$};
				\node at (3,2) {\scriptsize $y_{2,21}-y_{3,31}$};
				\node at (3,1.7) {\scriptsize $(\alpha-1,\beta,\gamma)$};
				\node at (4,3) {\scriptsize $y_{1,12}-y_{3,32}$};
				\node at (4,2.7) {\scriptsize $(\alpha+1,\beta-1,\gamma)$};
				\node at (-1.5,1) {\scriptsize $y_{4,13}$};
				\node at (-2.3,0.7) {\scriptsize $(\alpha-1,\beta+1,\gamma-2)$};
				\node at (-2.5,2) {\scriptsize $y_{4,23}$};
				\node at (-3,1.7) {\scriptsize $(\alpha+1,\beta,\gamma-2)$};
				\node at (1,2) {\scriptsize $y_{3,31}-y_{4,41}$};
				\node at (0.7,1.7) {\scriptsize $(\alpha-1,\beta,\gamma)$};
				\node at (0,3) {\scriptsize $y_{3,32}-y_{4,42}$};
				\node at (0.6,2.7) {\scriptsize $(\alpha+1,\beta-1,\gamma)$};
				\node at (0,4) {\scriptsize $y_{2,23}-y_{4,43}$};
				\node at (0.4,3.7) {\scriptsize $(\alpha,\beta+1,\gamma-1)$};
				\node at (6.5,2) {\scriptsize $y_{2,21}$};
				\node at (7,1.7) {\scriptsize $(\alpha-1,\beta,\gamma)$};
				\node at (6.5,3) {\scriptsize $y_{1,12}$};
				\node at (7.3,2.7) {\scriptsize $(\alpha+1,\beta-1,\gamma)$};
				\node at (6.5,4) {\scriptsize $y_{1,13}$};
				\node at (7.3,3.7) {\scriptsize $(\alpha,\beta+1,\gamma-1)$};
				\node at (6.5,5) {\scriptsize $y_{1,14}$};
				\node at (7,4.7) {\scriptsize $(\alpha,\beta,\gamma+1)$};
				\node at (10.5,3) {\scriptsize $z_{234}$};
				\node at (11,2.7) {\scriptsize $(\alpha+1,\beta,\gamma)$};
				\node at (10.5,2) {\scriptsize $z_{134}$};
				\node at (11.3,1.7) {\scriptsize $(\alpha-1,\beta+1,\gamma)$};
				\node at (10.5,1) {\scriptsize $z_{124}$};
				\node at (11.3,0.7) {\scriptsize $(\alpha,\beta-1,\gamma+1)$};
				\node at (10.5,0) {\scriptsize $z_{123}$};
				\node at (11,-0.3) {\scriptsize $(\alpha,\beta,\gamma-1)$};
			\end{tikzpicture}$$}
		
	\end{center}

\caption{Elements $y_{i,jk}$ and $z_{ijk}$ and their weights} 
\label{weights}
\end{figure}

\section{Singular vectors of degree 4 and 5}
Consider a singular vector $w$ of degree 4. By Lemma \ref{barw} such vector can be written in the following form
\[
w=\sum_{i\leq j\leq k, j}\frac{1}{m_{ijk}}\de_i\de_j\de_k d_l\otimes v_{ijk,l}+\sum_{i\leq j, k<l}\frac{1}{m_{ij}}\de_i \de_j d_k d_l\otimes y_{ij,kl}+\sum_{i,j<k<l}\de_i d_jd_kd_l\otimes z_{i,jkl}+d_1d_2d_3d_4\otimes u,
\]
where $m_{ijk}$ is the order of the stabilizer in $S_3$ of the triple $\{i,j,k\}$,
%$m_{ijk}=6$ if $i=j=k$, $m_{ijk}=1$ if $i,j,k$ are distinct and $m_{ijk}=2$ if $|\{i,j,k\}|=2$
and, similarly, $m_{ij}=1$ if $i\neq j$ and $m_{ij}=2$ if $i=j$. We only sketch the proof since most of the computations have been done using a computer. The computation of the action of $x_i \de_j $, $x_i^2 \de_j$ (with $i<j$), $b_{rs}$ (for all $r,s$) $x_4^2\de_4$ and $x_i^2 d_j$ (with $i\neq j$) on $w$ provides a system of homogeneous linear equations on the variables $v_{ijk,l},  y_{ij,kl}, z_{i,jkl}$. This system has the following solution:
\begin{itemize}
	\item $v_{ijk,l}=0$ for all $i,j,k,l$;
	\item $y_{ij,kl}=0$ unless $\{k,l\}=\{1,i\}$ or $\{k,l\}=\{1,j\}$;
	\item $z_{i,jkl}\neq 0$ only if $1\neq i$ and $1,i\in \{j,k,l\}$. 
	\item $y_{12,12}=y_{13,13}=y_{14,14}=:y_1$;
	\item $\frac{1}{2}y_{22,12}=y_{23,13}=y_{24,14}=:y_2$;
	\item $\frac{1}{2}y_{33,13}=y_{23,12}=y_{34,14}=:y_3$;
	\item $\frac{1}{2}y_{44,14}=y_{24,12}=y_{34,13}=:y_4$;
	\item $-z_{3,123}=-z_{4,124}=:z_2$;
	\item $z_{2,123}=-z_{4,134}=:z_3$;
	\item $z_{2,124}=z_{3,134}=:z_4$.
\end{itemize}
It follows that the vector $w$ must have the following form
\[
w=\sum_{i=1}^4 \de_i d_1(\de_2d_2+\de_3d_3+\de_4d_4)\otimes y_i+\sum_{i=2}^4(\de_2d_2+\de_3d_3+\de_4d_4)d_id_1\otimes z_i+d_1d_2d_3d_4\otimes u
\]
for some $y_1,y_2,y_3,y_4,z_2,z_3,z_4,u \in W_t(a,b,c)$. Note that if we let $\lambda(w)=(\alpha,\beta,\gamma)$ we have the following weights
\begin{itemize}
	\item $\lambda(y_1)=(\alpha,\beta,\gamma)$; 
	\item $\lambda(y_2)=(\alpha-2,\beta+1,\gamma+1)$; 
	\item $\lambda(y_3)=(\alpha-1,\beta-1,\gamma+1)$; 
	\item $\lambda(y_4)=(\alpha-1,\beta,\gamma-1)$; 
	\item $\lambda(z_4)=(\alpha-1,\beta,\gamma+1)$; 
	\item $\lambda(z_3)=(\alpha-1,\beta+1,\gamma-1)$; 
	\item $\lambda(z_2)=(\alpha,\beta-1,\gamma)$.
\end{itemize}
Observe that, applying suitable elements $x_i\de_j$ with $i<j$ or $b_{44}$ if any of the elements $y_1,y_2,y_3,y_4,z_4,z_3,z_2$ is zero, then the other elements to the left of it are also zero.
Hence we may assume that $z_2\neq 0$. Applying $x_4^2 \de_3$ to $w$, we obtain (coefficient of $d_1d_2d_4$) $(x_4 \de_3)z_2=0$, which forces $\gamma=0$. Similarly, applying $x_3^2 \de_1$ and considering the coefficient of $d_1d_2d_3$, we obtain $(x_3\de_1) z_2=0$ which forces $\alpha+\beta=1$. If also $z_3\neq 0$, we can apply $x_2^2 \de_1$ and looking at the coefficient of $d_1d_2d_3$ we obtain $(x_2\de_1)z_3=0$, which forces $\alpha=1$.
So we have the following 3 possibilities:\\
1) $z_3=0$ and the only nonzero element is $z_2$ (and $u$). In this case $z_2$ is a $\phat$-singular vector of weight $(0,0,0)$ and $w$ has weight $(0,1,0)$;\\
2) $z_2=0$;\\
3) $z_3\neq 0$ and $(\alpha,\beta, \gamma)=(1,0,0)$;\\

Case 1) does not occur. Indeed applying $x_4^2 d_3$ and considering the coefficient of $\de_3 d_3d_4$ we get $z_2=0$. Case 2) also does not occur since the equation $(x_4^2 \de_4)w$ would imply $u=0$.

We are therefore left to consider case 3). If $y_1\neq 0$, it is a $\phat$-singular vector of weight $(1,0,0)$ and therefore $w\in M_t(1,0,0)$, the Verma module induced by the standard module. Weight reasons and the compatibility of the action of the subalgebra  $N$ of $\phat$ forces $w$ to be of the following form
\[
w=\sum_{i=1}^4 \de_i d_1(\de_2d_2+\de_3d_3+\de_4d_4)\otimes d_i+\sum_{i=2}^4(\de_2d_2+\de_3d_3+\de_4d_4)d_id_1\otimes \de_i+d_1d_2d_3d_4\otimes \rho d_1
\]
for some $\rho \in \C$. Applying $x_4^2\de_4$ one deduce that $\rho=\frac{1}{2}(1-t)$ and one can show that with this choice we actually obtain a singular vector of every $t\in \C$. We have therefore proved the following.
\begin{theorem}\label{listagrado4} The following is a complete list of $E(4,4)$-singular vectors of degree 4, up to a scalar factor:
	\[w[4_H]=
	\sum_{i=1}^4 \de_i d_1(\de_2d_2+\de_3d_3+\de_4d_4)\otimes d_i+\sum_{i=2}^4(\de_2d_2+\de_3d_3+\de_4d_4)d_id_1\otimes \de_i+d_1d_2d_3d_4\otimes \frac{1}{2}(1-t) d_1 \]
	lying in $W_t(1,0,0)$
	for every $t\in \C$.
\end{theorem}
One can verify that $a_{12}w=0$ which implies that for all $t\in\C$ we have a morphism
\[
\varphi_{[4_H]}: M_{t-4}(1,0,0)\rightarrow M_t(1,0,0).
\]

Singular vectors of degree 5 can be treated as we did for singular vectors of degree 4. In this case, however,  the corresponding system of homogeneous linear equations has only the trivial solution and we conclude that there are no singular vectors of degree 5.

\medskip

We can now state a complete classification of singular vectors in Verma modules over $E(4,4)$.

\begin{theorem}
The following is a complete classification, up to a scalar factor, of singular vectors in  $E(4,4)$-Verma modules and the corresponding morphisms.
\begin{enumerate}
\item In degree one we have
\begin{itemize}
	\item $w[1_A]=d_1\otimes x_1^a\in M_t(a,0,0)$ with $a=t=0$ or $a\in\Z_{\geq 1}$, giving morphisms
	\[\varphi_{[1_A]}: M_{t-1}(a+1,0,0)\rightarrow M_t(a,0,0)\] 
	for $a=t=0$, or $a=1$ and $t\in\C$, $t\neq 1$, or $a\in\Z_{\geq 2}$ and all $t\in\C$, and
		\[\tilde \varphi_{[1_A]}:\tilde M_0(2,0,0)\rightarrow M_1(1,0,0);\]
	\item $w[1_B]=\sum_{i=1}^4 d_i\otimes x_i^*(x_4^*)^{c-1}\in M_t(0,0,c)$ for all $c\in\Z_{\geq 1}$, giving the morphism
\[\varphi_{[1_B]}: M_{t-1}(0,0,c-1)\rightarrow M_t(0,0,c)\]
 for all $c\in\Z_{\geq 1}$, $t\in\C$ and $(c,t)\neq (1,1)$;	
	\item $w[1_C]=2\de_3\otimes x_4^*-2\de_4 \otimes x_3^*-\sum_{i=1}^4 d_i\otimes a_{12}x_i^*\in M_t(0,0,1)$ for all $t\in\C$, giving the morphism
	\[\varphi_{[1_C]}:M_{t-1}(0,1,0)\rightarrow M_t(0,0,1)\] 
	for all $t\in\C$;
	\item $w[1_D]=4\sum_{i=2}^4 \de_i\otimes a_{1i}-2(d_4\otimes a_{12}a_{13}+d_3\otimes a_{14}a_{12}+d_2\otimes a_{13}a_{14})-d_1\otimes(a_{12}a_{34}+a_{24}a_{13}+a_{14}a_{23}-4)\in M_t(0,0,0)$ for all $t\in\C$, $t\neq 0$, giving the morphism
	\[\varphi_{[1_D]}:M_{t-1}(1,0,0)\rightarrow M_t(0,0,0)\]
	for all $t\in\C\setminus\{0\}$;
	\item $w[1_E]=\sum_{i=1}^4 (\de_i\otimes d_i+d_i\otimes \de_i )\in M_1(1,0,0)$ giving the morphism
\[\varphi_{[1_E]}: M_0(0,0,0)\rightarrow M_1(1,0,0).\]
	\end{itemize}
	We also have the morphism
	\[\tilde \varphi_{[1_D]}: M_{-1}(1,0,0)\rightarrow \tilde M_0(0,1,0).\]	\item In degree two we have
\begin{itemize}
		\item $w{[2_{DA}]}=4\sum_{i=2}^4 d_1\de_i\otimes a_{1i}-2d_1(d_4\otimes a_{12}a_{13}+d_3\otimes a_{14}a_{12}+d_2\otimes a_{13}a_{14})\in M_t(0,1,0)$ (where for $t\neq 0$ it is meant that $M_t(0,1,0)\cong M_t(0,0,0)$), giving the morphism
		\[\varphi_{[2_{DA}]}: M_{t-2}{(2,0,0)}\rightarrow M_t(0,1,0)\]
		for every $t\in\C$;
	\item $w{[2_{EA}]}=d_1w[1_E]\in M_1(1,0,0)$, giving the morphism
\[\varphi_{[2_{EA}]}:M_{-1}(1,0,0)\rightarrow M_1(1,0,0).\]	
\end{itemize}
\item
In degree three we have
\begin{itemize}
\item $w[3_F]=\sum_{i, j} \de_i \de_j d_i\otimes d_j+\de_i d_i d_j \otimes \de_j+d_1d_2d_3\otimes d_4-d_1d_2d_4\otimes d_3+d_1d_3d_4\otimes d_2-d_2d_3d_4\otimes d_1\in M_3(1,0,0)$,
giving the morphism
\[\varphi_{[3_F]}: M_0(0,0,0) \rightarrow M_3(1,0,0);\]
\item $w[3_G]=(\sum_{i=1}^4\de_id_i)d_1\otimes 1\in M_0(0,0,0)$,
giving the morphism
\[\varphi_{[3_G]}: M_{-3}(1,0,0) \rightarrow M_0(0,0,0).\]
\end{itemize}
\item
In degree four we have 
\begin{itemize}
\item $w[4_H]=
	\sum_{i=1}^4 \de_i d_1(\de_2d_2+\de_3d_3+\de_4d_4)\otimes d_i+\sum_{i=2}^4(\de_2d_2+\de_3d_3+\de_4d_4)d_id_1\otimes \de_i+d_1d_2d_3d_4\otimes \frac{1}{2}(1-t) d_1$ in $W_t(1,0,0)$
	for every $t\in \C$, giving the morphism
\[\varphi_{[4_H]}: M_{t-4}(1,0,0)\rightarrow M_t(1,0,0)\]
for every $t\in\C$.	
	\end{itemize}
	\end{enumerate}
\end{theorem}

\begin{corollary} A complete list of degenerate Verma modules for the Lie superalgebra $E(4,4)$ is the following:
\begin{itemize}
\item $M_t(a,0,0)$ for every $t\in\C$ and every $a\in\Z_{\geq 0}$;
\item $M_t(0,0,c)$ for every $t\in\C$ and every $c\in\Z_{\geq 0}$;
\item $M_t(0,1,0)$ for every $t\in\C$.
\end{itemize}
\end{corollary}

\begin{remark}
Notice that  in Figure \ref{all} as well as in Figure \ref{deRham2}, modules $M_t(1,0,0)$ appear twice: in green with $a=1$ and in black. So Figure \ref{all} shows that two arrows enter $M_1(1,0,0)$, namely, morphisms 
$\varphi[1_E]: M_0(0,0,0) \rightarrow M_1(1,0,0)$ and $\varphi[4_H]: 
M_{-3}(1,0,0) \rightarrow M_1(1,0,0)$.
\end{remark}

\begin{remark}
Duality reflects Figure \ref{all} with respect to the origin.
\end{remark}

\begin{figure}[h]\label{tuttiimorfismi}
	\begin{center}
		\scalebox{1}{	$$
			\begin{tikzpicture} 
				\draw[fill=green]{(1,0) circle(3pt)}; 
				\draw[fill=green]{(2,0) circle(3pt)};
				\draw[fill=green]{(3,0) circle(3pt)};
				\draw[fill=green]{(4,0) circle(3pt)};
				\draw[fill=green]{(5,0) circle(3pt)};
				\draw[fill=green]{(6,0) circle(3pt)};
				\draw[fill=green]{(1,1) circle(3pt)};
				\draw[fill=green]{(2,1) circle(3pt)};
				\draw[fill=green]{(3,1) circle(3pt)};
				\draw[fill=green]{(4,1) circle(3pt)};
				\draw[fill=green]{(5,1) circle(3pt)};
				\draw[fill=green]{(6,1) circle(3pt)};
				\draw[fill=green]{(1,2) circle(3pt)};
				\draw[fill=green]{(2,2) circle(3pt)};
				\draw[fill=green]{(3,2) circle(3pt)};
				\draw[fill=green]{(4,2) circle(3pt)};
				\draw[fill=green]{(5,2) circle(3pt)};
				\draw[fill=green]{(6,2) circle(3pt)};
				\draw[fill=green]{(1,3) circle(3pt)};
				\draw[fill=green]{(2,3) circle(3pt)};
				\draw[fill=green]{(3,3) circle(3pt)};
				\draw[fill=green]{(4,3) circle(3pt)};
				\draw[fill=green]{(5,3) circle(3pt)};
				\draw[fill=green]{(6,3) circle(3pt)};
				\draw[fill=green]{(1,4) circle(3pt)};
				\draw[fill=green]{(2,4) circle(3pt)};
				\draw[fill=green]{(3,4) circle(3pt)};
				\draw[fill=green]{(4,4) circle(3pt)};
				\draw[fill=green]{(5,4) circle(3pt)};
				\draw[fill=green]{(6,4) circle(3pt)};
				\draw[fill=green]{(1,5) circle(3pt)};
				\draw[fill=green]{(2,5) circle(3pt)};
				\draw[fill=green]{(3,5) circle(3pt)};
				\draw[fill=green]{(4,5) circle(3pt)};
				\draw[fill=green]{(5,5) circle(3pt)};
				\draw[fill=green]{(6,5) circle(3pt)};
				
				\draw[fill=green]{(1,-1) circle(3pt)};
				\draw[fill=green]{(2,-1) circle(3pt)};
				\draw[fill=green]{(3,-1) circle(3pt)};
				\draw[fill=green]{(4,-1) circle(3pt)};
				\draw[fill=green]{(5,-1) circle(3pt)};
				\draw[fill=green]{(6,-1) circle(3pt)};
				\draw[fill=green]{(1,-2) circle(3pt)};
				\draw[fill=green]{(2,-2) circle(3pt)};
				\draw[fill=green]{(3,-2) circle(3pt)};
				\draw[fill=green]{(4,-2) circle(3pt)};
				\draw[fill=green]{(5,-2) circle(3pt)};
				\draw[fill=green]{(6,-2) circle(3pt)};
				\draw[fill=green]{(1,-3) circle(3pt)};
				\draw[fill=green]{(2,-3) circle(3pt)};
				\draw[fill=green]{(3,-3) circle(3pt)};
				\draw[fill=green]{(4,-3) circle(3pt)};
				\draw[fill=green]{(5,-3) circle(3pt)};
				\draw[fill=green]{(6,-3) circle(3pt)};
				\draw[fill=green]{(1,-4) circle(3pt)};
				\draw[fill=green]{(2,-4) circle(3pt)};
				\draw[fill=green]{(3,-4) circle(3pt)};
				\draw[fill=green]{(4,-4) circle(3pt)};
				\draw[fill=green]{(5,-4) circle(3pt)};
				\draw[fill=green]{(6,-4) circle(3pt)};
				\draw[fill=green]{(1,-5) circle(3pt)};
				\draw[fill=green]{(2,-5) circle(3pt)};
				\draw[fill=green]{(3,-5) circle(3pt)};
				\draw[fill=green]{(4,-5) circle(3pt)};
				\draw[fill=green]{(5,-5) circle(3pt)};
				\draw[fill=green]{(6,-5) circle(3pt)};
				\draw[fill=yellow]{(0,0) circle(3pt)};
				
				\draw[fill=black]{(-1,5) circle(3pt)};
				\draw[fill=black]{(-1,4) circle(3pt)};
				\draw[fill=black]{(-1,3) circle(3pt)};
				\draw[fill=black]{(-1,2) circle(3pt)};
				\draw[fill=black]{(-1,1) circle(3pt)};
				\draw[fill=black]{(-1,0) circle(3pt)};
				\draw[fill=black]{(-1,-1) circle(3pt)};
				\draw[fill=black]{(-1,-2) circle(3pt)};
				\draw[fill=black]{(-1,-3) circle(3pt)};
				\draw[fill=black]{(-1,-4) circle(3pt)};
				\draw[fill=black]{(-1,-5) circle(3pt)};
				
				\draw[fill=blue]{(-2,5) circle(3pt)};
				\draw[fill=blue]{(-2,4) circle(3pt)};
				\draw[fill=blue]{(-2,3) circle(3pt)};
				\draw[fill=blue]{(-2,2) circle(3pt)};
				\draw[fill=blue]{(-2,1) circle(3pt)};
				\draw[fill=blue]{(-2,0) circle(3pt)};
				\draw[fill=blue]{(-2,-1) circle(3pt)};
				\draw[fill=blue]{(-2,-2) circle(3pt)};
				\draw[fill=blue]{(-2,-3) circle(3pt)};
				\draw[fill=blue]{(-2,-4) circle(3pt)};
				\draw[fill=blue]{(-2,-5) circle(3pt)};
				
				\draw[fill=red]{(-6,5) circle(3pt)};
				\draw[fill=red]{(-5,5) circle(3pt)};
				\draw[fill=red]{(-4,5) circle(3pt)};
				\draw[fill=red]{(-3,5) circle(3pt)};
				\draw[fill=red]{(-6,4) circle(3pt)};
				\draw[fill=red]{(-5,4) circle(3pt)};
				\draw[fill=red]{(-4,4) circle(3pt)};
				\draw[fill=red]{(-3,4) circle(3pt)};
				\draw[fill=red]{(-6,3) circle(3pt)};
				\draw[fill=red]{(-5,3) circle(3pt)};
				\draw[fill=red]{(-4,3) circle(3pt)};
				\draw[fill=red]{(-3,3) circle(3pt)};
				\draw[fill=red]{(-6,2) circle(3pt)};
				\draw[fill=red]{(-5,2) circle(3pt)};
				\draw[fill=red]{(-4,2) circle(3pt)};
				\draw[fill=red]{(-3,2) circle(3pt)};
				\draw[fill=red]{(-6,1) circle(3pt)};
				\draw[fill=red]{(-5,1) circle(3pt)};
				\draw[fill=red]{(-4,1) circle(3pt)};
				\draw[fill=red]{(-3,1) circle(3pt)};
				\draw[fill=red]{(-6,0) circle(3pt)};
				\draw[fill=red]{(-5,0) circle(3pt)};
				\draw[fill=red]{(-4,0) circle(3pt)};
				\draw[fill=red]{(-3,0) circle(3pt)};
				\draw[fill=red]{(-6,-1) circle(3pt)};
				\draw[fill=red]{(-5,-1) circle(3pt)};
				\draw[fill=red]{(-4,-1) circle(3pt)};
				\draw[fill=red]{(-3,-1) circle(3pt)};
				\draw[fill=red]{(-6,-2) circle(3pt)};
				\draw[fill=red]{(-5,-2) circle(3pt)};
				\draw[fill=red]{(-4,-2) circle(3pt)};
				\draw[fill=red]{(-3,-2) circle(3pt)};
				\draw[fill=red]{(-6,-3) circle(3pt)};
				\draw[fill=red]{(-5,-3) circle(3pt)};
				\draw[fill=red]{(-4,-3) circle(3pt)};
				\draw[fill=red]{(-3,-3) circle(3pt)};
				\draw[fill=red]{(-6,-4) circle(3pt)};
				\draw[fill=red]{(-5,-4) circle(3pt)};
				\draw[fill=red]{(-4,-4) circle(3pt)};
				\draw[fill=red]{(-3,-4) circle(3pt)};
				\draw[fill=red]{(-6,-5) circle(3pt)};
				\draw[fill=red]{(-5,-5) circle(3pt)};
				\draw[fill=red]{(-4,-5) circle(3pt)};
				\draw[fill=red]{(-3,-5) circle(3pt)};

				%assi
				\draw[gray](-6.9,0)--(-6.1,0);
				\draw[gray](-5.9,0)--(-5.1,0);
				\draw[gray](-4.9,0)--(-4.1,0);
				\draw[gray](-3.9,0)--(-3.1,0);
				\draw[gray](-2.9,0)--(-2.1,0);
				\draw[gray](-1.9,0)--(-1.1,0);
				\draw[gray](-0.9,0)--(-0.1,0);
				\draw[gray](0.1,0)--(0.9,0);
				\draw[gray](1.1,0)--(1.9,0);
				\draw[gray](2.1,0)--(2.9,0);
				\draw[gray](3.1,0)--(3.9,0);
				\draw[gray](4.1,0)--(4.9,0);
				\draw[gray](5.1,0)--(5.9,0);
				\draw[->,gray](6.1,0)--(6.9,0);
				\draw[gray](0,-6)--(0,-0.1);
				\draw[->,gray](0,0.1)--(0,6);

				%quadrante 1
				%\draw[->,gray](1.5,0.5)--(1.1,0.9);
				\draw[->,black](2.9,0.1)--(2.1,0.9);
				\draw[->,black](3.9,0.1)--(3.1,0.9);
				\draw[->,black](4.9,0.1)--(4.1,0.9);
				\draw[->,black](5.9,0.1)--(5.1,0.9);
				\draw[->,black](1.9,1.1)--(1.1,1.9);
				\draw[->,black](2.9,1.1)--(2.1,1.9);
				\draw[->,black](3.9,1.1)--(3.1,1.9);
				\draw[->,black](4.9,1.1)--(4.1,1.9);
				\draw[->,black](5.9,1.1)--(5.1,1.9);
				\draw[->,black](1.9,2.1)--(1.1,2.9);
				\draw[->,black](2.9,2.1)--(2.1,2.9);
				\draw[->,black](3.9,2.1)--(3.1,2.9);
				\draw[->,black](4.9,2.1)--(4.1,2.9);
				\draw[->,black](5.9,2.1)--(5.1,2.9);
				\draw[->,black](1.9,3.1)--(1.1,3.9);
				\draw[->,black](2.9,3.1)--(2.1,3.9);
				\draw[->,black](3.9,3.1)--(3.1,3.9);
				\draw[->,black](4.9,3.1)--(4.1,3.9);
				\draw[->,black](5.9,3.1)--(5.1,3.9);
				\draw[->,black](1.9,4.1)--(1.1,4.9);
				\draw[->,black](2.9,4.1)--(2.1,4.9);
				\draw[->,black](3.9,4.1)--(3.1,4.9);
				\draw[->,black](4.9,4.1)--(4.1,4.9);
				\draw[->,black](5.9,4.1)--(5.1,4.9);
				%quadrante 4
				\draw[->,black](1.9,-0.9)--(1.1,-0.1);
				\draw[->,black](2.9,-0.9)--(2.1,-0.1);
				\draw[->,black](3.9,-0.9)--(3.1,-0.1);
				\draw[->,black](4.9,-0.9)--(4.1,-0.1);
				\draw[->,black](5.9,-0.9)--(5.1,-0.1);
				\draw[->,black](1.9,-1.9)--(1.1,-1.1);
				\draw[->,black](2.9,-1.9)--(2.1,-1.1);
				\draw[->,black](3.9,-1.9)--(3.1,-1.1);
				\draw[->,black](4.9,-1.9)--(4.1,-1.1);
				\draw[->,black](5.9,-1.9)--(5.1,-1.1);
				\draw[->,black](1.9,-2.9)--(1.1,-2.1);
				\draw[->,black](2.9,-2.9)--(2.1,-2.1);
				\draw[->,black](3.9,-2.9)--(3.1,-2.1);
				\draw[->,black](4.9,-2.9)--(4.1,-2.1);
				\draw[->,black](5.9,-2.9)--(5.1,-2.1);
				\draw[->,black](1.9,-3.9)--(1.1,-3.1);
				\draw[->,black](2.9,-3.9)--(2.1,-3.1);
				\draw[->,black](3.9,-3.9)--(3.1,-3.1);
				\draw[->,black](4.9,-3.9)--(4.1,-3.1);
				\draw[->,black](5.9,-3.9)--(5.1,-3.1);
				\draw[->,black](1.9,-4.9)--(1.1,-4.1);
				\draw[->,black](2.9,-4.9)--(2.1,-4.1);
				\draw[->,black](3.9,-4.9)--(3.1,-4.1);
				\draw[->,black](4.9,-4.9)--(4.1,-4.1);
				\draw[->,black](5.9,-4.9)--(5.1,-4.1);
				
				%quadrante 3
				%\draw[->,gray](-1.1,-0.9)--(-1.5,-0.5);
				\draw[->,black](-2.1,-0.9)--(-2.9,-0.1);
				\draw[->,black](-3.1,-0.9)--(-3.9,-0.1);
				\draw[->,black](-4.1,-0.9)--(-4.9,-0.1);
				\draw[->,black](-5.1,-0.9)--(-5.9,-0.1);
				\draw[->,black](-1.1,-1.9)--(-1.9,-1.1);
				\draw[->,black](-2.1,-1.9)--(-2.9,-1.1);
				\draw[->,black](-3.1,-1.9)--(-3.9,-1.1);
				\draw[->,black](-4.1,-1.9)--(-4.9,-1.1);
				\draw[->,black](-5.1,-1.9)--(-5.9,-1.1);
				\draw[->,black](-1.1,-2.9)--(-1.9,-2.1);
				\draw[->,black](-2.1,-2.9)--(-2.9,-2.1);
				\draw[->,black](-3.1,-2.9)--(-3.9,-2.1);
				\draw[->,black](-4.1,-2.9)--(-4.9,-2.1);
				\draw[->,black](-5.1,-2.9)--(-5.9,-2.1);
				\draw[->,black](-1.1,-3.9)--(-1.9,-3.1);
				\draw[->,black](-2.1,-3.9)--(-2.9,-3.1);
				\draw[->,black](-3.1,-3.9)--(-3.9,-3.1);
				\draw[->,black](-4.1,-3.9)--(-4.9,-3.1);
				\draw[->,black](-5.1,-3.9)--(-5.9,-3.1);
				\draw[->,black](-1.1,-4.9)--(-1.9,-4.1);
				\draw[->,black](-2.1,-4.9)--(-2.9,-4.1);
				\draw[->,black](-3.1,-4.9)--(-3.9,-4.1);
				\draw[->,black](-4.1,-4.9)--(-4.9,-4.1);
				\draw[->,black](-5.1,-4.9)--(-5.9,-4.1);
				%quadrante 2
				\draw[->,black](-1.1,0.1)--(-1.9,0.9);
				\draw[->,black](-2.1,0.1)--(-2.9,0.9);
				\draw[->,black](-3.1,0.1)--(-3.9,0.9);
				\draw[->,black](-4.1,0.1)--(-4.9,0.9);
				\draw[->,black](-5.1,0.1)--(-5.9,0.9);
				\draw[->,black](-1.1,1.1)--(-1.9,1.9);
				\draw[->,black](-2.1,1.1)--(-2.9,1.9);
				\draw[->,black](-3.1,1.1)--(-3.9,1.9);
				\draw[->,black](-4.1,1.1)--(-4.9,1.9);
				\draw[->,black](-5.1,1.1)--(-5.9,1.9);
				\draw[->,black](-1.1,2.1)--(-1.9,2.9);
				\draw[->,black](-2.1,2.1)--(-2.9,2.9);
				\draw[->,black](-3.1,2.1)--(-3.9,2.9);
				\draw[->,black](-4.1,2.1)--(-4.9,2.9);
				\draw[->,black](-5.1,2.1)--(-5.9,2.9);
				\draw[->,black](-1.1,3.1)--(-1.9,3.9);
				\draw[->,black](-2.1,3.1)--(-2.9,3.9);
				\draw[->,black](-3.1,3.1)--(-3.9,3.9);
				\draw[->,black](-4.1,3.1)--(-4.9,3.9);
				\draw[->,black](-5.1,3.1)--(-5.9,3.9);
				\draw[->,black](-1.1,4.1)--(-1.9,4.9);
				\draw[->,black](-2.1,4.1)--(-2.9,4.9);
				\draw[->,black](-3.1,4.1)--(-3.9,4.9);
				\draw[->,black](-4.1,4.1)--(-4.9,4.9);
				\draw[->,black](-5.1,4.1)--(-5.9,4.9);
				
				%\draw[->,gray](0.9,-0.9)--(0.1,-0.1);
				%\draw[->,gray](-0.1,0.1)--(-0.9,0.9);
				
				\draw[->, black, line width=1pt] (0.95,-2.87)--(0.05,-0.13);
				\draw[->, black, line width=1pt] (-0.05,0.14)--(-0.95,2.90);
				\draw[-, red, dotted,  line width=1pt] (-0.5, -5)--(-0.61,-4.78 );
				\draw[->, red, line width=1pt] (-0.61, -4.78)--(-0.95,-4.13 );
				
				\draw[-, red, , dotted, line width=1pt] (0, -5)--(-0.14,-4.72 );
				\draw[->, red, line width=1pt] (-0.14, -4.72)--(-0.95,-3.13 );
				
				\draw[-, red,dotted,  line width=1pt] (0.45, -5)--(0.31,-4.68 );
				\draw[->, red, line width=1pt] (0.31, -4.68)--(-0.95,-2.13 );
				
				\draw[->, red, line width=1pt] (0.93, -4.88)--(-0.95,-1.13 );
				\draw[->, red, line width=1pt] (0.93, -3.88)--(-0.95,-0.13 );
				\draw[->, red, line width=1pt] (0.93, -2.88)--(-0.95,0.87 );
				\draw[->, red, line width=1pt] (0.93, -0.88)--(-0.95,2.85 );
				\draw[->, red, line width=1pt] (0.93, 0.12)--(-0.95,3.87 );
				\draw[->, red, line width=1pt] (0.93, 1.12)--(-0.95,4.87 );
				\draw[red, line width=1pt] (0.93, 2.12)--(-0.36,4.7);
				\draw[red, dotted,  line width=1pt] (-0.36, 4.7)--(-0.5,5);
				\draw[red, line width=1pt] (0.93, 3.12)--(0.19,4.62);
				\draw[red, dotted, line width=1pt] (0.19, 4.62)--(0,5);
				\draw[red, line width=1pt] (0.93, 4.12)--(0.60,4.78);
				\draw[red, dotted, line width=1pt] (0.60, 4.78)--(0.5,5);
				
				\draw[->, red, line width=1pt](0.93,-1.88) [out=112, in=308] to (-0.95,1.87);
				
				\draw[->, black, line width=1pt] (0.9, -0.9)--(0.1,-0.1);
				\draw[->, black, line width=1pt] (-0.1, 0.1)--(-0.9,0.9);
				\draw[->, black, line width=1pt] (1.9, -1.9)--(-1.9,-0.1 );
				\draw[->, black, line width=1pt] (1.9, 0.1)--(-1.9,1.9);

				\draw[fill=yellow]{(-8,-7) circle(3pt)};
				\node at (-6.7,-7){=$M_0(0,0,0)$};
				\draw[fill=green]{(-5,-7) circle(3pt)};
				\node at (-3.6,-7){=$M_t(a,0,0)$};
				\draw[fill=red]{(-2,-7) circle(3pt)};
				\node at (-0.1,-7){=$M_t(0,0,-a-2)$};
				\draw[fill=blue]{(2,-7) circle(3pt)};
				\node at (3.35,-7){=$M_t(0,1,0)$};
				\draw[fill=black]{(5,-7) circle(3pt)};
				\node at (6.3,-7){=$M_t(1,0,0)$};
				\node at (7,-0.3){$a$};
				\node at (-0.2,5.8){$t$};
			\end{tikzpicture}$$}
		
	\end{center}

	\caption{Morphisms between Verma modules.} \label{all}
\end{figure}

\newpage
${}$
\newpage

\section{Exceptional de Rham complexes for E(4,4)}
As anticipated in the introduction, if we replace morphisms of degree one whose composition is not zero with their composition morphism of degree 2, we get a complex that we describe in Figure \ref{deRham1} involving only morphisms of degree 1 and 2, and in Figure \ref{deRham2} involving morphisms of degree 1, 3 and 4. 

In Section \ref{degree2} we
proved that the unique non zero compositions of morphisms of degree 1 between finite Verma modules are $\varphi_{[1_D]}\circ \varphi_{[1_A]}:M_{t-2}(2,0,0)\rightarrow M_t(0,0,0)$ (for $t\neq 0$) and $\varphi_{[1_E]}\circ \varphi_{[1_A]}:M_{-1}(1,0,0)\rightarrow M_1(1,0,0)$.  The only other non-zero compositions are
$\varphi_{[3_F]}\circ\varphi_{[1_A]}: M_{-1}(1,0,0)\rightarrow M_3(1,0,0)$ and
$\varphi_{[1_E]}\circ\varphi_{[3_G]}: M_{-3}(1,0,0)\rightarrow M_1(1,0,0)$. One can check that $\varphi_{[3_F]}\circ\varphi_{[1_A]}=\varphi_{[4_H]}$ and
$\varphi_{[1_E]}\circ\varphi_{[3_G]}=\varphi_{[4_H]}$.
The classification of morphisms shows that all other compositions are zero.

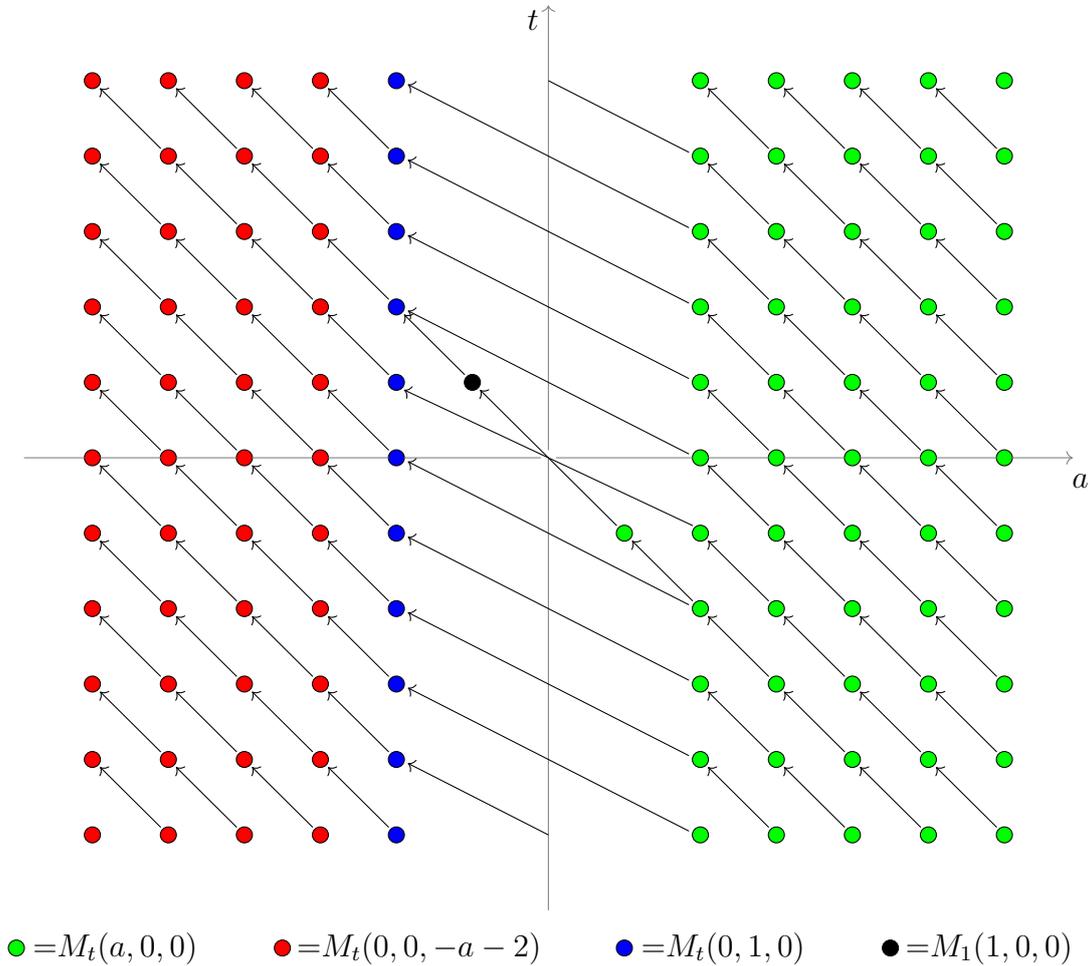
\begin{figure}[h]	
	\begin{center}
		\scalebox{1}{	$$
			\begin{tikzpicture} 
				%\draw[fill=green]{(1,0) circle(3pt)}; 
				\draw[fill=green]{(2,0) circle(3pt)};
				\draw[fill=green]{(3,0) circle(3pt)};
				\draw[fill=green]{(4,0) circle(3pt)};
				\draw[fill=green]{(5,0) circle(3pt)};
				\draw[fill=green]{(6,0) circle(3pt)};
				%\draw[fill=green]{(1,1) circle(3pt)};
				\draw[fill=green]{(2,1) circle(3pt)};
				\draw[fill=green]{(3,1) circle(3pt)};
				\draw[fill=green]{(4,1) circle(3pt)};
				\draw[fill=green]{(5,1) circle(3pt)};
				\draw[fill=green]{(6,1) circle(3pt)};
				%\draw[fill=green]{(1,2) circle(3pt)};
				\draw[fill=green]{(2,2) circle(3pt)};
				\draw[fill=green]{(3,2) circle(3pt)};
				\draw[fill=green]{(4,2) circle(3pt)};
				\draw[fill=green]{(5,2) circle(3pt)};
				\draw[fill=green]{(6,2) circle(3pt)};
				%\draw[fill=green]{(1,3) circle(3pt)};
				\draw[fill=green]{(2,3) circle(3pt)};
				\draw[fill=green]{(3,3) circle(3pt)};
				\draw[fill=green]{(4,3) circle(3pt)};
				\draw[fill=green]{(5,3) circle(3pt)};
				\draw[fill=green]{(6,3) circle(3pt)};
				%\draw[fill=green]{(1,4) circle(3pt)};
				\draw[fill=green]{(2,4) circle(3pt)};
				\draw[fill=green]{(3,4) circle(3pt)};
				\draw[fill=green]{(4,4) circle(3pt)};
				\draw[fill=green]{(5,4) circle(3pt)};
				\draw[fill=green]{(6,4) circle(3pt)};
				%\draw[fill=green]{(1,5) circle(3pt)};
				\draw[fill=green]{(2,5) circle(3pt)};
				\draw[fill=green]{(3,5) circle(3pt)};
				\draw[fill=green]{(4,5) circle(3pt)};
				\draw[fill=green]{(5,5) circle(3pt)};
				\draw[fill=green]{(6,5) circle(3pt)};
				
				\draw[fill=green]{(1,-1) circle(3pt)};
				\draw[fill=green]{(2,-1) circle(3pt)};
				\draw[fill=green]{(3,-1) circle(3pt)};
				\draw[fill=green]{(4,-1) circle(3pt)};
				\draw[fill=green]{(5,-1) circle(3pt)};
				\draw[fill=green]{(6,-1) circle(3pt)};
				%\draw[fill=green]{(1,-2) circle(3pt)};
				\draw[fill=green]{(2,-2) circle(3pt)};
				\draw[fill=green]{(3,-2) circle(3pt)};
				\draw[fill=green]{(4,-2) circle(3pt)};
				\draw[fill=green]{(5,-2) circle(3pt)};
				\draw[fill=green]{(6,-2) circle(3pt)};
				%\draw[fill=green]{(1,-3) circle(3pt)};
				\draw[fill=green]{(2,-3) circle(3pt)};
				\draw[fill=green]{(3,-3) circle(3pt)};
				\draw[fill=green]{(4,-3) circle(3pt)};
				\draw[fill=green]{(5,-3) circle(3pt)};
				\draw[fill=green]{(6,-3) circle(3pt)};
				%\draw[fill=green]{(1,-4) circle(3pt)};
				\draw[fill=green]{(2,-4) circle(3pt)};
				\draw[fill=green]{(3,-4) circle(3pt)};
				\draw[fill=green]{(4,-4) circle(3pt)};
				\draw[fill=green]{(5,-4) circle(3pt)};
				\draw[fill=green]{(6,-4) circle(3pt)};
				%\draw[fill=green]{(1,-5) circle(3pt)};
				\draw[fill=green]{(2,-5) circle(3pt)};
				\draw[fill=green]{(3,-5) circle(3pt)};
				\draw[fill=green]{(4,-5) circle(3pt)};
				\draw[fill=green]{(5,-5) circle(3pt)};
				\draw[fill=green]{(6,-5) circle(3pt)};
				%\draw[fill=yellow]{(0,0) circle(3pt)};
				
				%\draw[fill=black]{(-1,5) circle(3pt)};
				%\draw[fill=black]{(-1,4) circle(3pt)};
				%\draw[fill=black]{(-1,3) circle(3pt)};
				%\draw[fill=black]{(-1,2) circle(3pt)};
				\draw[fill=black]{(-1,1) circle(3pt)};
				%\draw[fill=black]{(-1,0) circle(3pt)};
				%\draw[fill=black]{(-1,-1) circle(3pt)};
				%\draw[fill=black]{(-1,-2) circle(3pt)};
				%\draw[fill=black]{(-1,-3) circle(3pt)};
				%\draw[fill=black]{(-1,-4) circle(3pt)};
				%\draw[fill=black]{(-1,-5) circle(3pt)};
				
				\draw[fill=blue]{(-2,5) circle(3pt)};
				\draw[fill=blue]{(-2,4) circle(3pt)};
				\draw[fill=blue]{(-2,3) circle(3pt)};
				\draw[fill=blue]{(-2,2) circle(3pt)};
				\draw[fill=blue]{(-2,1) circle(3pt)};
				\draw[fill=blue]{(-2,0) circle(3pt)};
				\draw[fill=blue]{(-2,-1) circle(3pt)};
				\draw[fill=blue]{(-2,-2) circle(3pt)};
				\draw[fill=blue]{(-2,-3) circle(3pt)};
				\draw[fill=blue]{(-2,-4) circle(3pt)};
				\draw[fill=blue]{(-2,-5) circle(3pt)};
				
				\draw[fill=red]{(-6,5) circle(3pt)};
				\draw[fill=red]{(-5,5) circle(3pt)};
				\draw[fill=red]{(-4,5) circle(3pt)};
				\draw[fill=red]{(-3,5) circle(3pt)};
				\draw[fill=red]{(-6,4) circle(3pt)};
				\draw[fill=red]{(-5,4) circle(3pt)};
				\draw[fill=red]{(-4,4) circle(3pt)};
				\draw[fill=red]{(-3,4) circle(3pt)};
				\draw[fill=red]{(-6,3) circle(3pt)};
				\draw[fill=red]{(-5,3) circle(3pt)};
				\draw[fill=red]{(-4,3) circle(3pt)};
				\draw[fill=red]{(-3,3) circle(3pt)};
				\draw[fill=red]{(-6,2) circle(3pt)};
				\draw[fill=red]{(-5,2) circle(3pt)};
				\draw[fill=red]{(-4,2) circle(3pt)};
				\draw[fill=red]{(-3,2) circle(3pt)};
				\draw[fill=red]{(-6,1) circle(3pt)};
				\draw[fill=red]{(-5,1) circle(3pt)};
				\draw[fill=red]{(-4,1) circle(3pt)};
				\draw[fill=red]{(-3,1) circle(3pt)};
				\draw[fill=red]{(-6,0) circle(3pt)};
				\draw[fill=red]{(-5,0) circle(3pt)};
				\draw[fill=red]{(-4,0) circle(3pt)};
				\draw[fill=red]{(-3,0) circle(3pt)};
				\draw[fill=red]{(-6,-1) circle(3pt)};
				\draw[fill=red]{(-5,-1) circle(3pt)};
				\draw[fill=red]{(-4,-1) circle(3pt)};
				\draw[fill=red]{(-3,-1) circle(3pt)};
				\draw[fill=red]{(-6,-2) circle(3pt)};
				\draw[fill=red]{(-5,-2) circle(3pt)};
				\draw[fill=red]{(-4,-2) circle(3pt)};
				\draw[fill=red]{(-3,-2) circle(3pt)};
				\draw[fill=red]{(-6,-3) circle(3pt)};
				\draw[fill=red]{(-5,-3) circle(3pt)};
				\draw[fill=red]{(-4,-3) circle(3pt)};
				\draw[fill=red]{(-3,-3) circle(3pt)};
				\draw[fill=red]{(-6,-4) circle(3pt)};
				\draw[fill=red]{(-5,-4) circle(3pt)};
				\draw[fill=red]{(-4,-4) circle(3pt)};
				\draw[fill=red]{(-3,-4) circle(3pt)};
				\draw[fill=red]{(-6,-5) circle(3pt)};
				\draw[fill=red]{(-5,-5) circle(3pt)};
				\draw[fill=red]{(-4,-5) circle(3pt)};
				\draw[fill=red]{(-3,-5) circle(3pt)};

				%assi
				\draw[gray](-6.9,0)--(-6.1,0);
				\draw[gray](-5.9,0)--(-5.1,0);
				\draw[gray](-4.9,0)--(-4.1,0);
				\draw[gray](-3.9,0)--(-3.1,0);
				\draw[gray](-2.9,0)--(-2.1,0);
				\draw[gray](-1.9,0)--(-0.1,0);
				%\draw[gray](-0.9,0)--(-0.1,0);
				\draw[gray](0.1,0)--(1.9,0);
				%\draw[gray](1.1,0)--(1.9,0);
				\draw[gray](2.1,0)--(2.9,0);
				\draw[gray](3.1,0)--(3.9,0);
				\draw[gray](4.1,0)--(4.9,0);
				\draw[gray](5.1,0)--(5.9,0);
				\draw[->,gray](6.1,0)--(6.9,0);
				\draw[gray](0,-6)--(0,-0.01);
				\draw[->,gray](0,0.1)--(0,6);

				%quadrante 1
				%\draw[->,black](1.4,0.6)--(1.1,0.9);
				%\draw[black, dotted](1.8,0.2)--(1.4,0.6);
				\draw[->,black](2.9,0.1)--(2.1,0.9);
				\draw[->,black](3.9,0.1)--(3.1,0.9);
				\draw[->,black](4.9,0.1)--(4.1,0.9);
				\draw[->,black](5.9,0.1)--(5.1,0.9);
				%\draw[->,black](1.9,1.1)--(1.1,1.9);
				\draw[->,black](2.9,1.1)--(2.1,1.9);
				\draw[->,black](3.9,1.1)--(3.1,1.9);
				\draw[->,black](4.9,1.1)--(4.1,1.9);
				\draw[->,black](5.9,1.1)--(5.1,1.9);
				%\draw[->,black](1.9,2.1)--(1.1,2.9);
				\draw[->,black](2.9,2.1)--(2.1,2.9);
				\draw[->,black](3.9,2.1)--(3.1,2.9);
				\draw[->,black](4.9,2.1)--(4.1,2.9);
				\draw[->,black](5.9,2.1)--(5.1,2.9);
				%\draw[->,black](1.9,3.1)--(1.1,3.9);
				\draw[->,black](2.9,3.1)--(2.1,3.9);
				\draw[->,black](3.9,3.1)--(3.1,3.9);
				\draw[->,black](4.9,3.1)--(4.1,3.9);
				\draw[->,black](5.9,3.1)--(5.1,3.9);
				%\draw[->,black](1.9,4.1)--(1.1,4.9);
				\draw[->,black](2.9,4.1)--(2.1,4.9);
				\draw[->,black](3.9,4.1)--(3.1,4.9);
				\draw[->,black](4.9,4.1)--(4.1,4.9);
				\draw[->,black](5.9,4.1)--(5.1,4.9);
				%quadrante 4
				%\draw[->,black](1.9,-0.9)--(1.1,-0.1);
				\draw[->,black](2.9,-0.9)--(2.1,-0.1);
				\draw[->,black](3.9,-0.9)--(3.1,-0.1);
				\draw[->,black](4.9,-0.9)--(4.1,-0.1);
				\draw[->,black](5.9,-0.9)--(5.1,-0.1);
				\draw[->,black](1.9,-1.9)--(1.1,-1.1);
				\draw[->,black](2.9,-1.9)--(2.1,-1.1);
				\draw[->,black](3.9,-1.9)--(3.1,-1.1);
				\draw[->,black](4.9,-1.9)--(4.1,-1.1);
				\draw[->,black](5.9,-1.9)--(5.1,-1.1);
				%\draw[->,black](1.9,-2.9)--(1.1,-2.1);
				\draw[->,black](2.9,-2.9)--(2.1,-2.1);
				\draw[->,black](3.9,-2.9)--(3.1,-2.1);
				\draw[->,black](4.9,-2.9)--(4.1,-2.1);
				\draw[->,black](5.9,-2.9)--(5.1,-2.1);
				%\draw[->,black](1.9,-3.9)--(1.1,-3.1);
				\draw[->,black](2.9,-3.9)--(2.1,-3.1);
				\draw[->,black](3.9,-3.9)--(3.1,-3.1);
				\draw[->,black](4.9,-3.9)--(4.1,-3.1);
				\draw[->,black](5.9,-3.9)--(5.1,-3.1);
				%\draw[->,black](1.9,-4.9)--(1.1,-4.1);
				\draw[->,black](2.9,-4.9)--(2.1,-4.1);
				\draw[->,black](3.9,-4.9)--(3.1,-4.1);
				\draw[->,black](4.9,-4.9)--(4.1,-4.1);
				\draw[->,black](5.9,-4.9)--(5.1,-4.1);
				
				%quadrante 3
				%\draw[black](-1.1,-0.9)--(-1.4,-0.6);
				%\draw[->,dotted,black](-1.4,-0.6)--(-1.8,-0.2);
				\draw[->,black](-2.1,-0.9)--(-2.9,-0.1);
				\draw[->,black](-3.1,-0.9)--(-3.9,-0.1);
				\draw[->,black](-4.1,-0.9)--(-4.9,-0.1);
				\draw[->,black](-5.1,-0.9)--(-5.9,-0.1);
				%\draw[->,black](-1.1,-1.9)--(-1.9,-1.1);
				\draw[->,black](-2.1,-1.9)--(-2.9,-1.1);
				\draw[->,black](-3.1,-1.9)--(-3.9,-1.1);
				\draw[->,black](-4.1,-1.9)--(-4.9,-1.1);
				\draw[->,black](-5.1,-1.9)--(-5.9,-1.1);
				%\draw[->,black](-1.1,-2.9)--(-1.9,-2.1);
				\draw[->,black](-2.1,-2.9)--(-2.9,-2.1);
				\draw[->,black](-3.1,-2.9)--(-3.9,-2.1);
				\draw[->,black](-4.1,-2.9)--(-4.9,-2.1);
				\draw[->,black](-5.1,-2.9)--(-5.9,-2.1);
				%\draw[->,black](-1.1,-3.9)--(-1.9,-3.1);
				\draw[->,black](-2.1,-3.9)--(-2.9,-3.1);
				\draw[->,black](-3.1,-3.9)--(-3.9,-3.1);
				\draw[->,black](-4.1,-3.9)--(-4.9,-3.1);
				\draw[->,black](-5.1,-3.9)--(-5.9,-3.1);
				%\draw[->,black](-1.1,-4.9)--(-1.9,-4.1);
				\draw[->,black](-2.1,-4.9)--(-2.9,-4.1);
				\draw[->,black](-3.1,-4.9)--(-3.9,-4.1);
				\draw[->,black](-4.1,-4.9)--(-4.9,-4.1);
				\draw[->,black](-5.1,-4.9)--(-5.9,-4.1);
				%quadrante 2
				%\draw[->,black](-1.1,0.1)--(-1.9,0.9);
				\draw[->,black](-2.1,0.1)--(-2.9,0.9);
				\draw[->,black](-3.1,0.1)--(-3.9,0.9);
				\draw[->,black](-4.1,0.1)--(-4.9,0.9);
				\draw[->,black](-5.1,0.1)--(-5.9,0.9);
				\draw[->,black](-1.1,1.1)--(-1.9,1.9);
				\draw[->,black](-2.1,1.1)--(-2.9,1.9);
				\draw[->,black](-3.1,1.1)--(-3.9,1.9);
				\draw[->,black](-4.1,1.1)--(-4.9,1.9);
				\draw[->,black](-5.1,1.1)--(-5.9,1.9);
				%\draw[->,black](-1.1,2.1)--(-1.9,2.9);
				\draw[->,black](-2.1,2.1)--(-2.9,2.9);
				\draw[->,black](-3.1,2.1)--(-3.9,2.9);
				\draw[->,black](-4.1,2.1)--(-4.9,2.9);
				\draw[->,black](-5.1,2.1)--(-5.9,2.9);
				%\draw[->,black](-1.1,3.1)--(-1.9,3.9);
				\draw[->,black](-2.1,3.1)--(-2.9,3.9);
				\draw[->,black](-3.1,3.1)--(-3.9,3.9);
				\draw[->,black](-4.1,3.1)--(-4.9,3.9);
				\draw[->,black](-5.1,3.1)--(-5.9,3.9);
				%\draw[->,black](-1.1,4.1)--(-1.9,4.9);
				\draw[->,black](-2.1,4.1)--(-2.9,4.9);
				\draw[->,black](-3.1,4.1)--(-3.9,4.9);
				\draw[->,black](-4.1,4.1)--(-4.9,4.9);
				\draw[->,black](-5.1,4.1)--(-5.9,4.9);
				
				%\draw[->,black](0.9,-0.9)--(0.1,-0.1);
				%\draw[->,black](1.1,-1.1)--(-0.9,0.9);
				
				%\draw[fill=yellow]{(-7,-6.5) circle(3pt)};
				%\node at (-5.7,-6.5){=$M_0(0,0,0)$};
				\draw[fill=green]{(-7,-6.5) circle(3pt)};
				\node at (-5.7,-6.5){=$M_t(a,0,0)$};
				\draw[fill=red]{(-3.5,-6.5) circle(3pt)};
				\node at (-1.7,-6.5){=$M_t(0,0,-a-2)$};
				\draw[fill=blue]{(1,-6.5) circle(3pt)};
				\node at (2.3,-6.5){=$M_t(0,1,0)$};
				\draw[fill=black]{(4.5,-6.5) circle(3pt)};
				\node at (5.8,-6.5){=$M_1(1,0,0)$};
				
				\node at (7,-0.3){$a$};
				\node at (-0.2,5.8){$t$};

				%\node at (0.8,-0.4){\scriptsize$\varphi_{[1_A]}$};
				%\node at (-2.2,-0.4){\scriptsize$\varphi_{{[1_{\!B, C}]}}$};
				%\node at (-3.2,-0.4){\scriptsize$\varphi_{[1_B]}$};
				%\node at (-4.2,-0.4){\scriptsize$\varphi_{[1_B]}$};
				%\node at (-5.2,-0.4){\scriptsize$\varphi_{[1_B]}$};
					%\node at (1.8,-0.4){\scriptsize$\varphi_{[1_A]}$};
				%\node at (2.8,-0.4){\scriptsize$\varphi_{[1_A]}$};
				%\node at (3.8,-0.4){\scriptsize$\varphi_{[1_A]}$};
				%\node at (4.8,-0.4){\scriptsize$\varphi_{[1_A]}$};
				
%				\node at (-0.3,0.7){\scriptsize$\varphi_{[1_E]}$};
	%			\node at (-1.3,0.7){\scriptsize$\varphi_{[1_D]}$};
		%		\node at (-2.3,0.7){\scriptsize$\varphi_{[1_C]}$};
			%	\node at (-3.3,0.7){\scriptsize$\varphi_{[1_B]}$};
				%\node at (-4.3,0.7){\scriptsize$\varphi_{[1_B]}$};
				%\node at (-5.3,0.7){\scriptsize$\varphi_{[1_B]}$};
				%\node at (1.7,0.7){\scriptsize$\tilde \varphi_{[1_A]}$};
				%\node at (2.7,0.7){\scriptsize$\varphi_{[1_A]}$};
				%\node at (3.7,0.7){\scriptsize$\varphi_{[1_A]}$};
				%\node at (4.7,0.7){\scriptsize$\varphi_{[1_A]}$};

				%\node at (-1.3,1.7){\scriptsize$\varphi_{[1_D]}$};
				%\node at (-1.1,-0.4){\scriptsize$\tilde \varphi_{[1_D]}$};
				%\node at (-2.3,1.7){\scriptsize$\varphi_{{[1_{\!B, C}]}}$};
				%\node at (-3.3,1.7){\scriptsize$\varphi_{[1_B]}$};
				%\node at (-4.3,1.7){\scriptsize$\varphi_{[1_B]}$};
				%\node at (-5.3,1.7){\scriptsize$\varphi_{[1_B]}$};
				%\node at (1.7,1.7){\scriptsize$\varphi_{[1_A]}$};
				%\node at (2.7,1.7){\scriptsize$\varphi_{[1_A]}$};
				%\node at (3.7,1.7){\scriptsize$\varphi_{[1_A]}$};
				%\node at (4.7,1.7){\scriptsize$\varphi_{[1_A]}$};
				%morfismi di grado 2
				\draw[->,black](0,-5)--(-1.85,-4.05);
				\draw[->,black](1.85,-4.95)--(-1.85,-3.05);
				\draw[->,black](1.85,-3.95)--(-1.85,-2.05);
				\draw[->,black](1.85,-2.95)--(-1.85,-1.05);
				\draw[->,black](1.85,-1.95)--(-1.85,-0.05);
				%\draw[->,black](1.85,-0.95)--(-1.85,0.95);
				\draw[->,black](1.85,0.05)--(-1.85,1.95);
				\draw[->,black](1.85,1.05)--(-1.85,2.95);
				\draw[->,black](1.85,2.05)--(-1.85,3.95);
				\draw[->,black](1.85,3.05)--(-1.85,4.95);
				\draw[-,black](1.85,4.05)--(0,5);
				\draw[->, black](1.9,-0.9)--(-1.9,0.9);
				\draw[->, black](0.9,-0.9)--(-0.9,0.9);
					\end{tikzpicture}$$}
				
			\end{center}
		
\caption{Exceptional de Rham complex with morphisms of degree 1 and 2 between Verma modules.} \label{deRham1}
\end{figure}

\begin{figure}[h]
	\begin{center}
		\scalebox{1}{	$$
			\begin{tikzpicture} 
				\draw[fill=green]{(1,0) circle(3pt)}; 
				\draw[fill=green]{(2,0) circle(3pt)};
				\draw[fill=green]{(3,0) circle(3pt)};
				\draw[fill=green]{(4,0) circle(3pt)};
				\draw[fill=green]{(5,0) circle(3pt)};
				\draw[fill=green]{(6,0) circle(3pt)};
				\draw[fill=green]{(1,1) circle(3pt)};
				\draw[fill=green]{(2,1) circle(3pt)};
				\draw[fill=green]{(3,1) circle(3pt)};
				\draw[fill=green]{(4,1) circle(3pt)};
				\draw[fill=green]{(5,1) circle(3pt)};
				\draw[fill=green]{(6,1) circle(3pt)};
				\draw[fill=green]{(1,2) circle(3pt)};
				\draw[fill=green]{(2,2) circle(3pt)};
				\draw[fill=green]{(3,2) circle(3pt)};
				\draw[fill=green]{(4,2) circle(3pt)};
				\draw[fill=green]{(5,2) circle(3pt)};
				\draw[fill=green]{(6,2) circle(3pt)};
				\draw[fill=green]{(1,3) circle(3pt)};
				\draw[fill=green]{(2,3) circle(3pt)};
				\draw[fill=green]{(3,3) circle(3pt)};
				\draw[fill=green]{(4,3) circle(3pt)};
				\draw[fill=green]{(5,3) circle(3pt)};
				\draw[fill=green]{(6,3) circle(3pt)};
				\draw[fill=green]{(1,4) circle(3pt)};
				\draw[fill=green]{(2,4) circle(3pt)};
				\draw[fill=green]{(3,4) circle(3pt)};
				\draw[fill=green]{(4,4) circle(3pt)};
				\draw[fill=green]{(5,4) circle(3pt)};
				\draw[fill=green]{(6,4) circle(3pt)};
				\draw[fill=green]{(1,5) circle(3pt)};
				\draw[fill=green]{(2,5) circle(3pt)};
				\draw[fill=green]{(3,5) circle(3pt)};
				\draw[fill=green]{(4,5) circle(3pt)};
				\draw[fill=green]{(5,5) circle(3pt)};
				\draw[fill=green]{(6,5) circle(3pt)};
				
				\draw[fill=green]{(1,-1) circle(3pt)};
				\draw[fill=green]{(2,-1) circle(3pt)};
				\draw[fill=green]{(3,-1) circle(3pt)};
				\draw[fill=green]{(4,-1) circle(3pt)};
				\draw[fill=green]{(5,-1) circle(3pt)};
				\draw[fill=green]{(6,-1) circle(3pt)};
				\draw[fill=green]{(1,-2) circle(3pt)};
				\draw[fill=green]{(2,-2) circle(3pt)};
				\draw[fill=green]{(3,-2) circle(3pt)};
				\draw[fill=green]{(4,-2) circle(3pt)};
				\draw[fill=green]{(5,-2) circle(3pt)};
				\draw[fill=green]{(6,-2) circle(3pt)};
				\draw[fill=green]{(1,-3) circle(3pt)};
				\draw[fill=green]{(2,-3) circle(3pt)};
				\draw[fill=green]{(3,-3) circle(3pt)};
				\draw[fill=green]{(4,-3) circle(3pt)};
				\draw[fill=green]{(5,-3) circle(3pt)};
				\draw[fill=green]{(6,-3) circle(3pt)};
				\draw[fill=green]{(1,-4) circle(3pt)};
				\draw[fill=green]{(2,-4) circle(3pt)};
				\draw[fill=green]{(3,-4) circle(3pt)};
				\draw[fill=green]{(4,-4) circle(3pt)};
				\draw[fill=green]{(5,-4) circle(3pt)};
				\draw[fill=green]{(6,-4) circle(3pt)};
				\draw[fill=green]{(1,-5) circle(3pt)};
				\draw[fill=green]{(2,-5) circle(3pt)};
				\draw[fill=green]{(3,-5) circle(3pt)};
				\draw[fill=green]{(4,-5) circle(3pt)};
				\draw[fill=green]{(5,-5) circle(3pt)};
				\draw[fill=green]{(6,-5) circle(3pt)};
				\draw[fill=yellow]{(0,0) circle(3pt)};
				
				\draw[fill=black]{(-1,5) circle(3pt)};
				\draw[fill=black]{(-1,4) circle(3pt)};
				\draw[fill=black]{(-1,3) circle(3pt)};
				\draw[fill=black]{(-1,2) circle(3pt)};
				\draw[fill=black]{(-1,1) circle(3pt)};
				\draw[fill=black]{(-1,0) circle(3pt)};
				\draw[fill=black]{(-1,-1) circle(3pt)};
				\draw[fill=black]{(-1,-2) circle(3pt)};
				\draw[fill=black]{(-1,-3) circle(3pt)};
				\draw[fill=black]{(-1,-4) circle(3pt)};
				\draw[fill=black]{(-1,-5) circle(3pt)};
				
				\draw[fill=blue]{(-2,5) circle(3pt)};
				\draw[fill=blue]{(-2,4) circle(3pt)};
				\draw[fill=blue]{(-2,3) circle(3pt)};
				\draw[fill=blue]{(-2,2) circle(3pt)};
				\draw[fill=blue]{(-2,1) circle(3pt)};
				\draw[fill=blue]{(-2,0) circle(3pt)};
				\draw[fill=blue]{(-2,-1) circle(3pt)};
				\draw[fill=blue]{(-2,-2) circle(3pt)};
				\draw[fill=blue]{(-2,-3) circle(3pt)};
				\draw[fill=blue]{(-2,-4) circle(3pt)};
				\draw[fill=blue]{(-2,-5) circle(3pt)};
				
				\draw[fill=red]{(-6,5) circle(3pt)};
				\draw[fill=red]{(-5,5) circle(3pt)};
				\draw[fill=red]{(-4,5) circle(3pt)};
				\draw[fill=red]{(-3,5) circle(3pt)};
				\draw[fill=red]{(-6,4) circle(3pt)};
				\draw[fill=red]{(-5,4) circle(3pt)};
				\draw[fill=red]{(-4,4) circle(3pt)};
				\draw[fill=red]{(-3,4) circle(3pt)};
				\draw[fill=red]{(-6,3) circle(3pt)};
				\draw[fill=red]{(-5,3) circle(3pt)};
				\draw[fill=red]{(-4,3) circle(3pt)};
				\draw[fill=red]{(-3,3) circle(3pt)};
				\draw[fill=red]{(-6,2) circle(3pt)};
				\draw[fill=red]{(-5,2) circle(3pt)};
				\draw[fill=red]{(-4,2) circle(3pt)};
				\draw[fill=red]{(-3,2) circle(3pt)};
				\draw[fill=red]{(-6,1) circle(3pt)};
				\draw[fill=red]{(-5,1) circle(3pt)};
				\draw[fill=red]{(-4,1) circle(3pt)};
				\draw[fill=red]{(-3,1) circle(3pt)};
				\draw[fill=red]{(-6,0) circle(3pt)};
				\draw[fill=red]{(-5,0) circle(3pt)};
				\draw[fill=red]{(-4,0) circle(3pt)};
				\draw[fill=red]{(-3,0) circle(3pt)};
				\draw[fill=red]{(-6,-1) circle(3pt)};
				\draw[fill=red]{(-5,-1) circle(3pt)};
				\draw[fill=red]{(-4,-1) circle(3pt)};
				\draw[fill=red]{(-3,-1) circle(3pt)};
				\draw[fill=red]{(-6,-2) circle(3pt)};
				\draw[fill=red]{(-5,-2) circle(3pt)};
				\draw[fill=red]{(-4,-2) circle(3pt)};
				\draw[fill=red]{(-3,-2) circle(3pt)};
				\draw[fill=red]{(-6,-3) circle(3pt)};
				\draw[fill=red]{(-5,-3) circle(3pt)};
				\draw[fill=red]{(-4,-3) circle(3pt)};
				\draw[fill=red]{(-3,-3) circle(3pt)};
				\draw[fill=red]{(-6,-4) circle(3pt)};
				\draw[fill=red]{(-5,-4) circle(3pt)};
				\draw[fill=red]{(-4,-4) circle(3pt)};
				\draw[fill=red]{(-3,-4) circle(3pt)};
				\draw[fill=red]{(-6,-5) circle(3pt)};
				\draw[fill=red]{(-5,-5) circle(3pt)};
				\draw[fill=red]{(-4,-5) circle(3pt)};
				\draw[fill=red]{(-3,-5) circle(3pt)};

				%assi
				\draw[gray](-6.9,0)--(-6.1,0);
				\draw[gray](-5.9,0)--(-5.1,0);
				\draw[gray](-4.9,0)--(-4.1,0);
				\draw[gray](-3.9,0)--(-3.1,0);
				\draw[gray](-2.9,0)--(-2.1,0);
				\draw[gray](-1.9,0)--(-1.1,0);
				\draw[gray](-0.9,0)--(-0.1,0);
				\draw[gray](0.1,0)--(0.9,0);
				\draw[gray](1.1,0)--(1.9,0);
				\draw[gray](2.1,0)--(2.9,0);
				\draw[gray](3.1,0)--(3.9,0);
				\draw[gray](4.1,0)--(4.9,0);
				\draw[gray](5.1,0)--(5.9,0);
				\draw[->,gray](6.1,0)--(6.9,0);
				\draw[gray](0,-6)--(0,-0.1);
				\draw[->,gray](0,0.1)--(0,6);

				%quadrante 1
				%\draw[->,gray](1.5,0.5)--(1.1,0.9);
				\draw[->,black](2.9,0.1)--(2.1,0.9);
				\draw[->,black](3.9,0.1)--(3.1,0.9);
				\draw[->,black](4.9,0.1)--(4.1,0.9);
				\draw[->,black](5.9,0.1)--(5.1,0.9);
				\draw[->,black](1.9,1.1)--(1.1,1.9);
				\draw[->,black](2.9,1.1)--(2.1,1.9);
				\draw[->,black](3.9,1.1)--(3.1,1.9);
				\draw[->,black](4.9,1.1)--(4.1,1.9);
				\draw[->,black](5.9,1.1)--(5.1,1.9);
				\draw[->,black](1.9,2.1)--(1.1,2.9);
				\draw[->,black](2.9,2.1)--(2.1,2.9);
				\draw[->,black](3.9,2.1)--(3.1,2.9);
				\draw[->,black](4.9,2.1)--(4.1,2.9);
				\draw[->,black](5.9,2.1)--(5.1,2.9);
				\draw[->,black](1.9,3.1)--(1.1,3.9);
				\draw[->,black](2.9,3.1)--(2.1,3.9);
				\draw[->,black](3.9,3.1)--(3.1,3.9);
				\draw[->,black](4.9,3.1)--(4.1,3.9);
				\draw[->,black](5.9,3.1)--(5.1,3.9);
				\draw[->,black](1.9,4.1)--(1.1,4.9);
				\draw[->,black](2.9,4.1)--(2.1,4.9);
				\draw[->,black](3.9,4.1)--(3.1,4.9);
				\draw[->,black](4.9,4.1)--(4.1,4.9);
				\draw[->,black](5.9,4.1)--(5.1,4.9);
				%quadrante 4
				\draw[->,black](1.9,-0.9)--(1.1,-0.1);
				\draw[->,black](2.9,-0.9)--(2.1,-0.1);
				\draw[->,black](3.9,-0.9)--(3.1,-0.1);
				\draw[->,black](4.9,-0.9)--(4.1,-0.1);
				\draw[->,black](5.9,-0.9)--(5.1,-0.1);
				\draw[->,black](1.9,-1.9)--(1.1,-1.1);
				\draw[->,black](2.9,-1.9)--(2.1,-1.1);
				\draw[->,black](3.9,-1.9)--(3.1,-1.1);
				\draw[->,black](4.9,-1.9)--(4.1,-1.1);
				\draw[->,black](5.9,-1.9)--(5.1,-1.1);
				\draw[->,black](1.9,-2.9)--(1.1,-2.1);
				\draw[->,black](2.9,-2.9)--(2.1,-2.1);
				\draw[->,black](3.9,-2.9)--(3.1,-2.1);
				\draw[->,black](4.9,-2.9)--(4.1,-2.1);
				\draw[->,black](5.9,-2.9)--(5.1,-2.1);
				\draw[->,black](1.9,-3.9)--(1.1,-3.1);
				\draw[->,black](2.9,-3.9)--(2.1,-3.1);
				\draw[->,black](3.9,-3.9)--(3.1,-3.1);
				\draw[->,black](4.9,-3.9)--(4.1,-3.1);
				\draw[->,black](5.9,-3.9)--(5.1,-3.1);
				\draw[->,black](1.9,-4.9)--(1.1,-4.1);
				\draw[->,black](2.9,-4.9)--(2.1,-4.1);
				\draw[->,black](3.9,-4.9)--(3.1,-4.1);
				\draw[->,black](4.9,-4.9)--(4.1,-4.1);
				\draw[->,black](5.9,-4.9)--(5.1,-4.1);
				
				%quadrante 3
				%\draw[->,gray](-1.1,-0.9)--(-1.5,-0.5);
				\draw[->,black](-2.1,-0.9)--(-2.9,-0.1);
				\draw[->,black](-3.1,-0.9)--(-3.9,-0.1);
				\draw[->,black](-4.1,-0.9)--(-4.9,-0.1);
				\draw[->,black](-5.1,-0.9)--(-5.9,-0.1);
				\draw[->,black](-1.1,-1.9)--(-1.9,-1.1);
				\draw[->,black](-2.1,-1.9)--(-2.9,-1.1);
				\draw[->,black](-3.1,-1.9)--(-3.9,-1.1);
				\draw[->,black](-4.1,-1.9)--(-4.9,-1.1);
				\draw[->,black](-5.1,-1.9)--(-5.9,-1.1);
				\draw[->,black](-1.1,-2.9)--(-1.9,-2.1);
				\draw[->,black](-2.1,-2.9)--(-2.9,-2.1);
				\draw[->,black](-3.1,-2.9)--(-3.9,-2.1);
				\draw[->,black](-4.1,-2.9)--(-4.9,-2.1);
				\draw[->,black](-5.1,-2.9)--(-5.9,-2.1);
				\draw[->,black](-1.1,-3.9)--(-1.9,-3.1);
				\draw[->,black](-2.1,-3.9)--(-2.9,-3.1);
				\draw[->,black](-3.1,-3.9)--(-3.9,-3.1);
				\draw[->,black](-4.1,-3.9)--(-4.9,-3.1);
				\draw[->,black](-5.1,-3.9)--(-5.9,-3.1);
				\draw[->,black](-1.1,-4.9)--(-1.9,-4.1);
				\draw[->,black](-2.1,-4.9)--(-2.9,-4.1);
				\draw[->,black](-3.1,-4.9)--(-3.9,-4.1);
				\draw[->,black](-4.1,-4.9)--(-4.9,-4.1);
				\draw[->,black](-5.1,-4.9)--(-5.9,-4.1);
				%quadrante 2
				\draw[->,black](-1.1,0.1)--(-1.9,0.9);
				\draw[->,black](-2.1,0.1)--(-2.9,0.9);
				\draw[->,black](-3.1,0.1)--(-3.9,0.9);
				\draw[->,black](-4.1,0.1)--(-4.9,0.9);
				\draw[->,black](-5.1,0.1)--(-5.9,0.9);
				\draw[->,black](-1.1,1.1)--(-1.9,1.9);
				\draw[->,black](-2.1,1.1)--(-2.9,1.9);
				\draw[->,black](-3.1,1.1)--(-3.9,1.9);
				\draw[->,black](-4.1,1.1)--(-4.9,1.9);
				\draw[->,black](-5.1,1.1)--(-5.9,1.9);
				\draw[->,black](-1.1,2.1)--(-1.9,2.9);
				\draw[->,black](-2.1,2.1)--(-2.9,2.9);
				\draw[->,black](-3.1,2.1)--(-3.9,2.9);
				\draw[->,black](-4.1,2.1)--(-4.9,2.9);
				\draw[->,black](-5.1,2.1)--(-5.9,2.9);
				\draw[->,black](-1.1,3.1)--(-1.9,3.9);
				\draw[->,black](-2.1,3.1)--(-2.9,3.9);
				\draw[->,black](-3.1,3.1)--(-3.9,3.9);
				\draw[->,black](-4.1,3.1)--(-4.9,3.9);
				\draw[->,black](-5.1,3.1)--(-5.9,3.9);
				\draw[->,black](-1.1,4.1)--(-1.9,4.9);
				\draw[->,black](-2.1,4.1)--(-2.9,4.9);
				\draw[->,black](-3.1,4.1)--(-3.9,4.9);
				\draw[->,black](-4.1,4.1)--(-4.9,4.9);
				\draw[->,black](-5.1,4.1)--(-5.9,4.9);
				
				%\draw[->,gray](0.9,-0.9)--(0.1,-0.1);
				%\draw[->,gray](-0.1,0.1)--(-0.9,0.9);
				
				\draw[->, black, line width=1pt] (0.95,-2.87)--(0.05,-0.13);
				\draw[->, black, line width=1pt] (-0.05,0.14)--(-0.95,2.90);
				\draw[-, red, dotted,  line width=1pt] (-0.5, -5)--(-0.61,-4.78 );
				\draw[->, red, line width=1pt] (-0.61, -4.78)--(-0.95,-4.13 );
				
				\draw[-, red, , dotted, line width=1pt] (0, -5)--(-0.14,-4.72 );
				\draw[->, red, line width=1pt] (-0.14, -4.72)--(-0.95,-3.13 );
				
				\draw[-, red,dotted,  line width=1pt] (0.45, -5)--(0.31,-4.68 );
				\draw[->, red, line width=1pt] (0.31, -4.68)--(-0.95,-2.13 );
				
				\draw[->, red, line width=1pt] (0.93, -4.88)--(-0.95,-1.13 );
				\draw[->, red, line width=1pt] (0.93, -3.88)--(-0.95,-0.13 );
				\draw[->, red, line width=1pt] (0.93, -2.88)--(-0.95,0.87 );
				\draw[->, red, line width=1pt] (0.93, -0.88)--(-0.95,2.85 );
				\draw[->, red, line width=1pt] (0.93, 0.12)--(-0.95,3.87 );
				\draw[->, red, line width=1pt] (0.93, 1.12)--(-0.95,4.87 );
				\draw[red, line width=1pt] (0.93, 2.12)--(-0.36,4.7);
				\draw[red, dotted,  line width=1pt] (-0.36, 4.7)--(-0.5,5);
				\draw[red, line width=1pt] (0.93, 3.12)--(0.19,4.62);
				\draw[red, dotted, line width=1pt] (0.19, 4.62)--(0,5);
				\draw[red, line width=1pt] (0.93, 4.12)--(0.60,4.78);
				\draw[red, dotted, line width=1pt] (0.60, 4.78)--(0.5,5);
				
				\draw[->, red, line width=1pt](0.93,-1.88) [out=112, in=308] to (-0.95,1.87);
				
				\draw[fill=yellow]{(-8,-7) circle(3pt)};
				\node at (-6.7,-7){=$M_0(0,0,0)$};
				\draw[fill=green]{(-5,-7) circle(3pt)};
				\node at (-3.6,-7){=$M_t(a,0,0)$};
				\draw[fill=red]{(-2,-7) circle(3pt)};
				\node at (-0.1,-7){=$M_t(0,0,-a-2)$};
				\draw[fill=blue]{(2,-7) circle(3pt)};
				\node at (3.35,-7){=$M_t(0,1,0)$};
				\draw[fill=black]{(5,-7) circle(3pt)};
				\node at (6.3,-7){=$M_t(1,0,0)$};
				\node at (7,-0.3){$a$};
				\node at (-0.2,5.8){$t$};
			\end{tikzpicture}$$}
		
	\end{center}
	
	\caption{Exceptional de Rham complex involving morphisms of degree 1, 3 and 4 between Verma modules.} \label{deRham2}
\end{figure}

\end{document}